\documentclass[12pt]{amsart}
\usepackage{a4wide}

\usepackage{epsfig}
\usepackage{mathtools}
\usepackage{amsmath}
\usepackage{amssymb}
\usepackage{amsfonts}
\usepackage{soul}
\usepackage{graphicx}

\usepackage{amscd}
\usepackage{graphicx}
\usepackage{color}
\usepackage{verbatim}
\ProvideTextCommandDefault{\cprime}{\tprime}
\newtheorem{theorem}{Theorem}[section]
\newtheorem{lemma}[theorem]{Lemma}
\newtheorem{prop}[theorem]{Proposition}
\newtheorem{cor}[theorem]{Corollary}

\newtheorem{question}[theorem]{Question}
\newtheorem{claim}[theorem]{Claim}

\newtheorem{definition}[theorem]{Definition}
\newtheorem{example}[theorem]{Example}

\renewcommand{\ge}{\geqslant}
\renewcommand{\geq}{\geqslant}
\renewcommand{\le}{\leqslant}
\renewcommand{\leq}{\leqslant}

\DeclareMathOperator{\esssup}{esssup}
\DeclareMathOperator{\essinf}{essinf}

\newtheorem{remark}[theorem]{Remark}

\numberwithin{equation}{section}

\def \N {\mathbb N}

\def \Z {\mathbb Z}

\def \M {\mathcal M}

\def \N {\mathbb N}

\def \Z {\mathbb Z}

\def \M {\mathcal M}

\makeatletter
\@namedef{subjclassname@2020}{\textup{2020} Mathematics Subject Classification}
\makeatother

\makeindex

\begin{document}

\title[curves in smooth surface systems with no measure of maximal entropy]{smooth surface systems may contain smooth curves which have no measure of maximal entropy}

\address{\vskip 2pt \hskip -12pt Xulei Wang}

\address{\hskip -12pt School of Mathematical Sciences, Fudan University, Shanghai 200433, China}

\email{20110180049@fudan.edu.cn}

\author{Xulei Wang and Guohua Zhang}

\address{\vskip 2pt \hskip -12pt Guohua Zhang}

\address{\hskip -12pt School of Mathematical Sciences and Shanghai Center for Mathematical Sciences, Fudan University, Shanghai 200433, China}

\email{chiaths.zhang@gmail.com}



\parindent=10pt
\begin{abstract}
In this paper, we study Borel probability measures of maximal entropy for analytic subsets in a dynamical system. It is well known that higher smoothness of the map over smooth space plays important role in the study of invariant measures of maximal entropy.
A famous theorem of Newhouse states that smooth diffeomorphisms on compact manifolds without boundary have invariant measures of maximal entropy.
However, we show that the situation becomes completely different when we study measures of maximal entropy for analytic subsets.
Namely, we shall construct a smooth surface system such that it contains a smooth curve which has no Borel probability measure of maximal entropy.
Another evidence to show this difference (between the study of measures of maximal entropy for analytic sets and that of invariant measures of maximal entropy for dynamical systems) is that, once an analytic set has one measure of maximal entropy, then the set has many measures of maximal entropy (no matter if we consider packing or Bowen entropy).

For a general dynamical system with positive entropy $h_\mathrm{top}(T)$, we shall show that the system contains not only a Borel subset such that the subset has Borel probability measures of maximal entropy and has entropy sufficiently close to $h_\mathrm{top}(T)$, but also a Borel subset such that the subset has no Borel probability measures of maximal entropy and has entropy equal to the arbitrarily given positive real number which is at most $h_\mathrm{top}(T)$.

We also provide in all $h$-expansive systems a full characterization for analytic subsets which have Borel probability measures of maximal entropy. Consequently, if let $Z\subset \mathbb{R}^n$ be any analytic subset with positive Hausdorff dimension in Euclidean space, then the set $Z$ either has a measure of full lower Hausdorff dimension, or it can be partitioned into a union of countably many analytic sets $\{Z_i\}_{i\in \mathbb{N}}$ with $\dim_{\mathcal{H}} (Z_i) < \dim_{\mathcal{H}} (Z)$ for each $i\in \mathbb{N}$.
\end{abstract}

\maketitle

\setcounter{tocdepth}{1}

\tableofcontents

\section{Introduction}

By a \emph{topological dynamical system} (TDS for short) we mean a pair $(X,T)$, where $X$ is a compact metric space and $T:X\rightarrow X$ is a continuous map.
\emph{Throughout the whole paper, we let $(X, T)$ be a TDS with a compatible metric $d$ if without any statement.}

Let $\M(X)$ and $\M(X,T)$ denote respectively the sets of all Borel probability measures and $T$-invariant Borel probability measures on $X$. And denote by $h_\mathrm{top} (T)$ and $h_\mu (T)$, respectively, topological entropy of $(X, T)$ and measure-theoretical entropy of $\mu\in \mathcal{M} (X, T)$.
The classical variational principle  for topological entropy gives the basic relationship between topological and measure-theoretical entropies of a dynamical system
\begin{equation} \label{VP}
h_\mathrm{top} (T)\ =\ \sup \{h_\mu (T) : \mu \in \mathcal{M} (X, T)\}\ =\ \sup \{h_\mu (T) : \mu \in \mathcal{M} (X, T)\ \text{is ergodic}\},
\end{equation}
which is due to Goodman, Goodwyn, and Dinaburg.

\smallskip

In 1989 Newhouse proved his famous theorem that, for any smooth self-map $f$ on compact manifolds without boundary $M$, the function of measure-theoretical entropy $h_\mu (f)$ is upper semicontinuous with respect to an invariant measure $\mu$ and hence the system $(M, f)$ has invariant measures of maximal entropy \cite{Newhouse1989}. Since then people paid much attention to the study of upper semicontinuity of the function of measure-theoretical entropy and invariant measures of maximal entropy particularly for smooth systems.

It turns out that higher smoothness of the continuous self-map over smooth spaces plays an important role in the study of invariant measures of maximal entropy for a dynamical system.
By exploring symbolic dynamics for surface diffeomorphisms Sarig proved in 2013 that, for any $C^{1+\beta}$ diffeomorphism $f$ (with $0< \beta< 1$) of a compact smooth surface with positive topological entropy, the diffeomorphism $f$ possesses at most countably many ergodic measures of maximal entropy \cite{Sarig2013}; furthermore, by generalizing Smale's spectral decomposition theorem, Buzzi, Crovisier and Sarig answered in \cite{Buzzi-Crovisier-Sarig2022-AM} a question of Newhouse, where they proved that $C^\infty$-surface diffeomorphisms with positive topological entropy have finitely many ergodic measures of maximal entropy in general, and exactly one in the topologically transitive case; however, Buzzi constructed on the disk for any $1 \le r < \infty$ a $C^r$-diffeomorphism which has no measure of maximal entropy \cite{Buzzi2014} (we remark that in fact Misiurewicz constructed in 1973 a $C^r$ diffeomorphism with some finite $r$ on a compact manifold which has no invariant measure of maximal entropy \cite{Misiurewicz1973}).
Boyle and Buzzi proved in 2017 various structural results for symbolic systems and surface diffeomorphisms under so-called almost Borel isomorphisms \cite{Boyle-Buzzi2017}; these results extend Hochman's work in \cite{Hochman2013}, and mostly pertain to countable-state topological Markov shifts and the treatment of surface diffeomorphisms following Sarig's symbolic covers.
Berger studied properties of invariant measures of maximal entropy for an abundant class of strongly regular non-uniformly hyperbolic $C^2$ H$\acute{e}$non-like diffeomorphisms which correspond to Benedicks-Carleson parameters \cite{Berger2019}; Burguet proved that periodic asymptotical expansiveness implies the equidistribution of periodic points with respect to measures of maximal entropy \cite{Burguet2020-JEMS}; Buzzi, Fisher and Tahzibi proved a dichotomy for measures of maximal entropy near time-one maps of transitive Anosov flows, that is, either all of the measures of maximal entropy are non-hyperbolic or there are exactly two ergodic measures of maximal entropy \cite{Buzzi-Fisher-Sarig2022}.
In 2022 Buzzi, Crovisier and Sarig proved entropic continuity properties of Lyapunov exponents
for $C^\infty$ surface diffeomorphisms \cite{Buzzi-Crovisier-Sarig2022-IM}, whose finite smooth version was established by Burguet soon later in \cite{Burguet2024-AHP}.

\smallskip

Motivated by the famous Brin-Katok formula \cite{Brin-Katok1981}, Feng and Huang introduced in \cite{Feng-Huang2012} the concepts of measure-theoretical upper and lower entropies, and established a variational principle between packing (Bowen, respectively) topological entropy of an \emph{analytic subset}\footnote{\ A subset of a Polish space is called an \emph{analytic set} or a \emph{Suslin set}, if it is a continuous image of the space $\mathcal{N}$, where $\mathcal{N}$ denotes the set of all infinite sequences of natural
	numbers (with the product topology), equivalently, it is a continuous image of a Borel set in a Polish space.

In a Polish space $X$, the analytic subsets are closed under countable unions and intersections,
continuous images, and inverse images. However, the complement of an analytic set is not necessarily analytic.
 Any Borel set is analytic, and that any analytic set $S$ is universally measurable (that is, $S$ is $\mu$-measurable for each $\mu\in \mathcal{M} (X)$). For more details see for example 
 \cite[14.A and 21.10]{Kechris1995}.} in a dynamical system and corresponding measure-theoretical upper (lower, respectively) entropy. Namely, for every analytic set $\emptyset\neq Z\subset X$,
if denote by $h_{\mathrm {top }}^P(T, Z)$ and $h_{\mathrm {top }}^B(T, Z)$, respectively, packing and Bowen topological entropies\footnote{\ For brevity, from now on we shall just write packing entropy and Bowen entropy, respectively.} of $Z$, and by $\overline{h}_\mu(T)$ and $\underline{h}_\mu(T)$, respectively, measure-theoretical upper and lower entropies of $\mu \in \M(X)$, then we have
\begin{equation*}
h_{\mathrm {top }}^P(T, Z)\ = \ \sup_{\mu \in \M(X), \mu(Z)=1} \ \overline{h}_\mu(T)\quad \text{and} \quad h_{\mathrm {top }}^B(T, Z)\ = \ \sup_{\mu \in \M(X), \mu(Z)=1} \ \underline{h}_\mu(T).
\end{equation*}

In this paper we study measures of maximal entropy for subsets in a dynamical system.

\begin{definition}
Let $Z\subset X$ be an analytic set. We say that the subset $Z$ has
\begin{enumerate}

\item \emph{Borel probability measures of maximal packing entropy} or just \emph{measures of maximal packing entropy}, if $\mu (Z) = 1$ and $\overline{h}_\mu(T) = h_{\mathrm {top }}^P(T, Z)> 0$ for some $\mu\in \mathcal{M} (X)$.

\item \emph{Borel probability measures of maximal Bowen entropy} or just \emph{measures of maximal Bowen entropy}, if $\mu (Z) = 1$ and $\underline{h}_\mu(T) = h_{\mathrm {top }}^B(T, Z)> 0$ for some $\mu\in \mathcal{M} (X)$.
\end{enumerate}
\end{definition}

Based on the famous Brin-Katok formula \cite{Brin-Katok1981}, it is easy to see from the definitions that, given a TDS $(X, T)$ with invariant measures of maximal entropy, the space $X$ has measures of maximal both Packing and Bowen entropies.
We may wonder which kind of analytic subsets in which kind of topological dynamical system may have measures of maximal packing or Bowen entropy if the subset has positive packing or Bowen entropy.

\smallskip

We have seen that higher smoothness of the continuous self-map ensures existence of invariant measures of maximal entropy for a smooth system. And so it seems reasonable to expect that, when we consider measures of maximal packing and Bowen entropies for analytic subsets in a dynamical system, smoothness of the map and smooth structure of the state space will also help to ensure such phenomenon for some or even all analytic sets.

However, the following example tells us that, even we consider smooth curves in a smooth surface system, it may happen that the curves have neither measures of maximal packing entropy nor measures of maximal Bowen entropy.

\begin{example} \label{submanifold counterexample}
There exists a $C^\infty$ self-map $T$ over $X = [0,1]^2$, such that the smooth curve $K\doteq \{(r, r): 0\le r < 1\}\subset X$ has positive values for both packing and Bowen entropies, however, it has neither measures of maximal packing entropy nor measures of maximal Bowen entropy.
\end{example}

There is another evidence to show that the situation of considering measures of maximal entropy for analytic sets is completely different from that of the study of invariant measures of maximal entropy for dynamical systems. Note that, for systems with invariant measures of maximal entropy, an important aspect in its study is to estimate the number of such invariant measures (i.e., with maximal entropy). It is shown that there is often a unique invariant measure of maximal entropy for many smooth systems. While, as explained by following result, for each analytic set with measures of maximal entropy, the set has many measures of maximal entropy (no matter if we consider packing or Bowen entropy).

\begin{prop} \label{uniqueness}
	Let $Y\subset Z$ both be analytic sets in $X$, and $\mu\in\M(X)$ with $\mu(Z)=1$ and $\mu (Y) > 0$. Let $\nu$ be the normalized measure of $\mu$ restricted onto $Y$.
 \begin{enumerate}

 \item Assume that $\mu$ is a measure of maximal packing entropy for the set $Z$. Then $\mu$ is non-atomic, and
  $\nu$ is also a measure of maximal packing entropy for the set $Z$, in particular, the set $Z$ has many measures of maximal packing entropy.

 \item Assume that $\mu$ is a measure of maximal Bowen entropy for the set $Z$. Then $\mu$ is non-atomic, and
  $\nu$ is also a measure of maximal Bowen entropy for the set $Z$, in particular, the set $Z$ has many measures of maximal Bowen entropy.
 \end{enumerate}
\end{prop}

It seems natural to ask if there exists a TDS with positive topological entropy such that all of its analytic sets with positive topological entropy have measures of maximal topological entropy, and ask similarly if there exists another TDS with positive topological entropy such that all of its analytic sets with positive topological entropy have no measures of maximal topological entropy. But, as shown below, both situations will never happen.

\begin{theorem} \label{general-system}
	Assume that TDS $(X,T)$ has positive topological entropy. Then
\begin{enumerate}

\item \label{exists} There exists a Borel subset $Z\subset X$ such that, it has measures that are of maximal packing and Bowen entropies at the same time, and that both $h_\mathrm{top}^B(T,Z)$ and $h_\mathrm{top}^P(T,Z)$ take a common strictly positive value sufficiently close to $h_\mathrm{top} (T)$.

\item \label{exists-not}
For each $0 < h \leq h_\mathrm{top}(T)$, there exists a Borel subset $Z\subset X$, with $h_\mathrm{top}^B(T,Z)=h_\mathrm{top}^P(T,Z)=h$, such that $Z$ has no measure of maximal packing or Bowen entropy.
\end{enumerate}
\end{theorem}

 Thus it arises naturally that which kind of subsets has measures of maximal entropy.

 \smallskip

 As we shall prove in later sections, the following class of analytic subsets plays an important role in our full characterizations for subsets having measures of maximal entropy.

 \begin{definition}
Let $Z\subset X$ be an analytic set. We say that the subset $Z$ has
\begin{enumerate}

\item  \emph{an increasing countable slice of packing entropy} if there exists a countable collection $\{ Z_i\}_{i\in \mathbb{N}}$ of analytic sets such that $Z=\bigcup_{i\in \mathbb{N}} Z_i$ and $h_\mathrm{top}^P(T, Z_i)<h_\mathrm{top}^P(T, Z), \forall i\in \mathbb{N}$.

\item \emph{an increasing countable slice of Bowen entropy} if there exists a countable collection $\{ Z_i\}_{i\in \mathbb{N}}$ of analytic sets such that $Z=\bigcup_{i\in \mathbb{N}} Z_i$ and $h_\mathrm{top}^B(T, Z_i)<h_\mathrm{top}^B(T, Z), \forall i\in \mathbb{N}$.
\end{enumerate}
\end{definition}

We have the following characterization for analytic sets having measures of maximal packing entropy.
Note that the notion of $h$-expansiveness was introduced by Bowen \cite{R.Bowen1972-TAMS} (see next section for its detailed definition), which provides non-trivial sufficient conditions for the existence of invariant measures of maximal entropy in a dynamical system.

\begin{theorem}\label{MME Packing h-expan}
Let $Z\subset X$ be an analytic subset with $h_\mathrm{top}^P(T, Z)>0$. Then $(1) \Rightarrow (2)$:
	\begin{enumerate}

		\item[(1)] The subset $Z$ has measures of maximal packing entropy.

		\item[(2)] The subset $Z$ has no increasing countable slice of packing entropy.
	\end{enumerate}
If, additionally, the system $(X,T)$ is $h$-expansive, then it also holds the converse $(2) \Rightarrow (1)$.
\end{theorem}

In our characterization (see Theorem \ref{MME Bowen} below) for analytic sets having measures of maximal Bowen entropy, we shall use the following variation of gauge function (also called dimension function) which has been discussed by many references (see for example \cite{Rogers1998}).

\begin{definition}
A function $b: \N\to \mathbb{R}_+$ is called a \emph{gauge function} or \emph{dimension function} if $b(n)>0$ for all $n\in\N$ and $b (n)$ is monotonously decreasing to $0$ as $n$ tends to $\infty$.
\end{definition}

Our characterization of Theorem \ref{MME Bowen} is stated as follows, where the definition of $\mathcal{M}^b(Z)$ (for a gauge function $b$ and a set $Z$) will be introduced in next section.

\begin{theorem}\label{MME Bowen}
Let $Z\subset X$ be an analytic subset with $h_\mathrm{top}^B(T, Z)>0$. Then $(1) \Rightarrow (2)$ and $(3) \Rightarrow (2)$, where
	\begin{enumerate}

		\item[(1)] The subset $Z$ has measures of maximal Bowen entropy.

		\item[(2)] The subset $Z$ has no increasing countable slice of Bowen entropy.

		\item[(3)] There exists some gauge function $b$ such that $\mathcal{M}^b(Z)>0$ and $\lim\limits_{n\to \infty}e^{ns}b(n)=0$ for any $s < h_\mathrm{top}^B (T, Z)$.
	\end{enumerate}
Furthermore, we have:
\begin{itemize}

\item If, additionally, $Z$ is compact, then $(3) \Rightarrow (1)$.

\item If, additionally, we assume again that the system $(X,T)$ is $h$-expansive, then the above three properties are equivalent for all analytic subsets $Z$ with $h_\mathrm{top}^B(T, Z)>0$.
\end{itemize}
\end{theorem}

As a consequence of $(1) \Rightarrow (2)$ of both Theorem \ref{MME Packing h-expan} and Theorem \ref{MME Bowen}, we have:

\begin{cor}\label{cor_countable_slice}
Let $(X, T)$ be a TDS with positive topological entropy and invariant measures of maximal entropy.\footnote{\ Any $h$-expansive TDS (including any symbolic system), with positive topological entropy, satisfies the requirements of Corollary \ref{cor_countable_slice}. The definition of $h$-expansiveness will be introduced till next section.} Then, once $X$ is written as a countable union of analytic subsets $\{Z_i\}_{i\in \mathbb{N}}$, one has that $h_\mathrm{top}^B(T, Z_j) = h_\mathrm{top}^P(T, Z_j) = h_\mathrm{top} (T)$ for some $j\in \mathbb{N}$. In particular, 
 the space $X$ has neither an increasing countable slice of packing entropy nor an increasing countable slice of Bowen entropy.
\end{cor}

Since the end of last century, people paid much attention to the study of dimensional version of the existence of invariant measures of maximal entropy, see for example \cite{Barreira2008, Barreira-Wolf2003, Edgar1998, Falconer1990, Kenyon-Peres1996}.
Let $\mu$ be a Borel probability measure on $\mathbb{R}^n$,
and let $\dim_{\mathcal{H}}(Z)$ be the Hausdorff dimension of a set $Z\subset \mathbb{R}^n$ whose definition is postponed till next section. Following \cite[Section 3.3]{Edgar1998} we introduce the \emph{upper Hausdorff dimension} and \emph{lower Hausdorff dimension} (or simply \emph{Hausdorff dimension}) of the measure $\mu$, respectively, as follows:
\begin{align}
\overline{\dim}_{\mathcal{H}} (\mu) &= \inf\ \{\dim_{\mathcal{H}}(E): E\subset \mathbb{R}^n\ \text{is $\mu$-measurable and}\ \mu(\mathbb{R}^n \setminus E)= 0\}, \label{add-number1} \\
\underline{\dim}_{\mathcal{H}} (\mu) &= \inf\ \{\dim_{\mathcal{H}}(E): E\subset \mathbb{R}^n\ \text{is $\mu$-measurable and}\ \mu(E)> 0\}.  \label{add-number2}
\end{align}
The set $Z$ \emph{has a measure of full lower Hausdorff dimension}, if there exists a Borel probability measure $\mu$ on $\mathbb{R}^n$ such that $\mu(Z)=1$ and $\dim_{\mathcal{H}}(Z) = \underline{\dim}_{\mathcal{H}} (\mu)$.

As a byproduct of Theorem \ref{MME Bowen}, we have the following dichotomy for analytic sets $Z\subset \mathbb{R}^n$ with positive Hausdorff dimension, where $\mathbb{R}^n$ is equipped with the Euclidean metric.

\begin{cor} \label{Hausdorff}
	Let $Z\subset \mathbb{R}^n$ be an analytic subset with $\dim_{\mathcal{H}}(Z)>0$. Then the set $Z$ has no measure of full lower Hausdorff dimension, if and only if the set $Z$ has an increasing countable slice of Hausdorff dimension, that is, there exists a countable collection $\{ Z_i\}_{i\in \mathbb{N}}$ of analytic sets such that $Z=\bigcup_{i\in \mathbb{N}} Z_i$ and $\dim_{\mathcal{H}}(Z_i)<\dim_{\mathcal{H}}(Z)$ for each $i\in \mathbb{N}$.
\end{cor}

In order to further understand the existence of invariant measures of maximal entropy in a dynamical system, Misiurewicz firstly introduced asymptotical $h$-expansiveness as a refinement of $h$-expansiveness by Bowen \cite{Misiurewicz1973}, and then introduced the concept of topological conditional entropy of a dynamical system which is an invariant of topological conjugacy \cite{Misiurewicz1976}. An $h$-expansive system is asymptotically $h$-expansive, a system is asymptotically $h$-expansive if and only if it has zero topological conditional entropy, and any TDS with zero topological conditional entropy has invariant measures of maximal entropy.

And so we have the following natural question, which remains mysterious for us till now:\footnote{\ Along the line of Misiurewicz, the concept of topological conditional entropy can be generalized directly from a TDS to any subset of the system, which is still an invariant of topological conjugacy. In particular, topological conditional entropy of any subset is bounded from above by the topological conditional entropy of the system, and so by Theorem \ref{general-system}, there are many analytic subsets with zero topological conditional entropy which have no measures of maximal packing or Bowen entropy.}

\begin{question}
Is there an invariant of topological conjugacy, defined for all analytic sets of a dynamical system, which provides sufficient conditions of ensuring the existence of measures of maximal either packing or Bowen entropy for these sets?
\end{question}

Another natural question is:

\begin{question}
In full characterizations of subsets with measures of maximal entropy in Theorem \ref{MME Packing h-expan} and Theorem \ref{MME Bowen}, is it possible to weaken (or even remove) the assumption of $h$-expansiveness for dynamical systems?
\end{question}

The paper is organized as follows. In \S \ref{Preliminaries} we make some preliminaries, including recalling the definitions of various topological and measure-theoretical entropies, and $h$-expansive and asymptotically $h$-expansive systems, and listing related properties which will be used in later discussions. In \S \ref{basic} we discuss some basic properties concerning these topological and measure-theoretic entropies of analytic sets, consequently, we prove Proposition \ref{uniqueness} that each analytic set has many measures of maximal entropy once it has at least one such measure. In \S \ref{easy} we prove firstly easier parts of our results, namely, we prove Theorem \ref{general-system}, $(1) \Rightarrow (2)$ of Theorem \ref{MME Packing h-expan}, $(1) \Rightarrow (2)$ and $(3) \Rightarrow (2)$ of Theorem \ref{MME Bowen}, and Corollary \ref{cor_countable_slice}, we also prove Corollary \ref{Hausdorff} (based on Theorem \ref{MME Bowen}).
In \S \ref{section-exam} we present the detailed construction of Example \ref{submanifold counterexample} which states that, there exists a smooth curve in a smooth surface system such that the curve have neither measures of maximal packing entropy nor measures of maximal Bowen entropy.
In \S \ref{Existence of measure of maximal Packing entropy} we finish the proof of Theorem \ref{MME Packing h-expan} by showing that, for $h$-expansive systems, an analytic set $Z$ has measures of maximal packing entropy if it has no increasing countable slice of packing entropy.
Then in last two sections we finish the proof of Theorem \ref{MME Bowen} by using either increasing countable slice of Bowen entropy or some special gauge function to characterize analytic subsets which have measures of maximal Bowen entropy, where we manage in \S \ref{compact} and \S \ref{Relationship between items (1) and (2)}, respectively, the case of compact subsets and general analytic subsets with positive Bowen entropy.

\section{Preliminariles}\label{Preliminaries}

In this section, let us make some preliminaries for later discussions.

\smallskip

For each $n \in \mathbb{N}$ we may define a new compatible metric, the $n$-th Bowen metric $d_{n}$, as
$$
d_{n}(x, y)=\max \left\{d\left(T^{k}(x), T^{k}(y)\right): k=0, \cdots, n-1\right\}\ \ \ \ \text{for all}\ x, y\in X.
$$
For every $\varepsilon>0$ and each $x\in X$, we denote by $B_{n}(x, \varepsilon)$ and $\overline{B}_{n}(x, \varepsilon)$, respectively, the open and closed ball with a radius $\varepsilon$ centered at $x$ with respect to the metric $d_{n}$, that is,
$$
B_{n}(x, \varepsilon)=\left\{y \in X: d_{n}(x, y)<\varepsilon\right\}\ \ \text{and}\ \quad \overline{B}_{n}(x, \varepsilon)=\left\{y \in X: d_{n}(x, y) \le \varepsilon\right\} .
$$
When $n=1$, the $n$-th Bowen ball degenerates into the normal open and closed ball which will be simply denoted by $B(x, \varepsilon)$ and $\overline{B}(x, \varepsilon)$, respectively.

\subsection{Packing topological entropy}
\

\smallskip

Let us recall the packing topological entropy of a subset following \cite{Feng-Huang2012} as follows.

Let $\emptyset\neq Z \subset X$. For all $s \ge 0, N \in \mathbb{N}$ and $\varepsilon>0$, we set
$$
P_{N, \varepsilon}^{s}(Z)=\sup\ \sum_{i\in \mathcal{I}} e^{-s n_{i}},
$$
where the supremum is taken over all  countable pairwise disjoint families $\left\{\overline{B}_{n_{i}}\left(x_{i}, \varepsilon\right)\right\}_{i\in \mathcal{I}}$ such that $x_{i} \in Z$ and $n_{i} \ge N$ for all $i$. Obviously, the quantity $P_{N, \varepsilon}^{s}(Z)$ increases as the parameters $N$ and $\varepsilon$ decrease, and hence we may set
$$
P_{\varepsilon}^{s}(Z)=\lim _{N \rightarrow \infty} P_{N, \varepsilon}^{s}(Z) = \inf_{N\in \mathbb{N}} P_{N, \varepsilon}^{s}(Z)
$$
which increases as the parameter $\varepsilon$ decreases. Now we define
$$
\mathcal{P}_{\varepsilon}^{s}(Z)=\inf \left\{\sum_{i=1}^{\infty} P_{\varepsilon}^{s}\left(Z_{i}\right): \bigcup_{i=1}^{\infty} Z_{i} \supset Z\right\} .
$$
It is not hard to show that there exists a critical value of the parameter $s> 0$, denoted by $h_{\mathrm {top }}^{P}(T, Z, \varepsilon)$, where the quantity $\mathcal{P}_{\varepsilon}^{s}(Z)$ jumps from $\infty$ to 0, that is,
$$
\mathcal{P}_{\varepsilon}^{s}(Z)= \begin{cases}0, & \ \ \ \ \ \ \ s>h_{\mathrm {top }}^{P}(T, Z, \varepsilon), \\ \infty, & \ \ \ \ \ \ \ s<h_{\mathrm {top }}^{P}(T, Z, \varepsilon).\end{cases}
$$
By convention we set
$$
h_{\mathrm {top }}^{P}(T, Z, \varepsilon)= \begin{cases}0, & \ \ \ \ \ \ \ \text{if}\ \mathcal{P}_{\varepsilon}^{s}(Z) = 0\ \text{for each}\ s> 0, \\ \infty, & \ \ \ \ \ \ \ \text{if}\ \mathcal{P}_{\varepsilon}^{s}(Z) = \infty\ \text{for each}\ s> 0.\end{cases}
$$
Again the quantity $h_{\mathrm {top }}^{P}(T, Z, \varepsilon)$ increases as the parameter $\varepsilon$ decreases, and so we may introduce the \emph{packing topological entropy of $Z$ with respect to $T$} or, simply, the \emph{packing topological entropy of $Z$} if there is no confusion about the map $T$, as follows
$$
h_{\mathrm {top }}^{P}(T, Z)=\lim _{\varepsilon \rightarrow 0} h_{\mathrm {top }}^{P}(T, Z, \varepsilon) = \sup_{\varepsilon > 0} h_{\mathrm {top }}^{P}(T, Z, \varepsilon).
$$
Though $h_{\mathrm {top }}^{P}(T, Z, \varepsilon)$ does depend on the metric $d$, the quantity $h_{\mathrm {top }}^{P}(T, Z)$ is independent of selection of compatible metrics over the space.
We also set $h_{\mathrm {top }}^{P}(T, \emptyset) = 0$ by convention.

We shall use the following result proved as \cite[Lemma 4.1]{Feng-Huang2012}.

\begin{lemma} \label{F H Lem4.1}
	Let $Z \subset X$ and $s, \varepsilon>0$ satisfy $P_{\varepsilon}^{s}(Z)=\infty$. If $0\le a< b< \infty$ and $N \in \mathbb{N}$, then there exists a finite disjoint collection $\left\{\overline{B}_{n_{i}}\left(x_{i}, \varepsilon\right)\right\}_{1\le i \le k}$ such that
$$a< \sum_{i= 1}^k e^{-n_{i} s} < b\ \ \ \text{and}\ \ \ x_{i} \in Z, n_{i} \ge N\ \text{for each}\ i.$$
\end{lemma}

\subsection{Hausdorff dimension and Bowen topological entropy}\label{subs-Bowen} \label{defin}\

\smallskip

Let us start by introducing the notion of Hausdorff measure and Hausdorff dimension (see for example \cite{Pesin1997} and \cite{Rogers1998}, see also \cite{Sion-Sjerve1962}).

Let $(Y,d)$ be a complete metric space and $Z \subset Y$. We define the \emph{$s$-dimensional Hausdorff measure of $Z$} for each $s \ge 0$ by
$$
\mathcal{H}^s(Z)=\lim_{\varepsilon \rightarrow 0} \mathcal{H}^s_\varepsilon (Z),\ \ \ \text{where}\ \mathcal{H}^s_\varepsilon (Z) = \inf \sum_{i\in \mathcal{I}}(\operatorname{diam} U_i)^s,
$$
here the infimum is taken over all  countable families $\{U_i\}_{i\in \mathcal{I}}$ covering $Z$ with each $\operatorname{diam} U_i \leq \varepsilon$ (here $\operatorname{diam} U_i$ denotes the diameter of the set $U_i$).
 Then there exists a critical value of the parameter $s> 0$, denoted by $\dim_\mathcal{H} (Z)$, called
 the \emph{Hausdorff dimension of the set $Z$}, where $\mathcal{H}^s(Z)$ jumps from $\infty$ to 0, that is,
$$
\mathcal{H}^s(Z) = \begin{cases}0, & \ \ \ \ \ \ \ s>\dim_\mathcal{H} (Z), \\ \infty, & \ \ \ \ \ \ \ s<\dim_\mathcal{H} (Z).\end{cases}
$$
By convention we set
$$
\dim_\mathcal{H} (Z)= \begin{cases}0, & \ \ \ \ \ \ \ \text{if}\ \mathcal{H}^s(Z) = 0\ \text{for each}\ s> 0, \\ \infty, & \ \ \ \ \ \ \ \text{if}\ \mathcal{H}^s(Z) = \infty\ \text{for each}\ s> 0.\end{cases}
$$

The following basic fact is easy to obtain.

\begin{prop} \label{fact-old}
If $Z = \bigcup_{i=1}^{\infty} Z_i$, then
$
\dim_\mathcal{H} (Z) = \sup _{i \ge 1} \dim_\mathcal{H} (Z_i).
$
\end{prop}

\smallskip

More generally, let $h: [0,\infty) \to [0,\infty)$ be a continuous function defined on $[0,\infty)$ which is monotonically increasing on $[0, \infty)$, strictly positive on $(0, \infty)$ and satisfies $h(0)=0$ (we also call such a function as a \emph{gauge function}).
We can define similarly the Hausdorff measure of $Z$ with respect to $h$ by
$$
\mathcal{H}^h(Y)=\lim _{\varepsilon \rightarrow 0} \inf \sum_{i\in \mathcal{I}}h(\operatorname{diam} U_i),
$$
where the infimum is taken over all  countable families $\{U_i\}_{i\in \mathcal{I}}$ covering $Z$ with each $\operatorname{diam} U_i \leq \varepsilon$. In particular, the classical $s$-dimensional Hausdorff measure is the Hausdorff measure with respect to the gauge function $h(t) = t^s$ (for brevity, we will slightly abuse the notation and still denote it as $\mathcal{H}^s (Z)$).

The following basic fact is also easy to obtain (see for example \cite[Theorem 40]{Rogers1998}).

\begin{prop}\label{comp_gauge_Haus}
Let $h_1$ and $h_2$ be gauge functions satisfying
$\lim\limits_{t\to 0}\frac{h_2(t)}{h_1(t)}=0$.
If a set $Z$ has $\sigma$-finite $\mathcal{H}^{h_1}$-measure, then $\mathcal{H}^{h_2}(Z)=0$. Particularly, if a set $Z$ has
$\sigma$-finite $\mathcal{H}^{s}$-measure for some $s\geq 0$, then $\dim_\mathcal{H} (Y)\leq s$.
\end{prop}

We remark that, in general, Hausdorff measure with respect to a gauge function (including the $s$-dimensional Hausdorff measure) depends on the metric over the underlying space. And so Hausdorff dimension also depends on the metric over the topological space, however, the Hausdorff dimension is a bi-Lipschitz invariant. Note that, when we are talking about Hausdorff dimension of a set in $\mathbb{R}^n$, the space $\mathbb{R}^n$ is always assumed to be equipped with the standard Euclidean metric.

\smallskip

Now let us recall the notion of Bowen topological entropy following \cite{R.Bowen1973-TAMS}.
Recall that we have assumed $(X, T)$ to be a TDS with a compatible metric $d$.

Let $Z \subset X$. For all $s \ge 0, N \in \mathbb{N}$ and $\varepsilon>0$, we set
$$
\M_{N, \varepsilon}^{s}(Z)=\inf\ \sum_{i\in \mathcal{I}} e^{-s n_{i}},
$$
where the infimum is taken over all countable families $\left\{ B_{n_{i}}\left(x_{i}, \varepsilon\right)\right\}_{i\in \mathcal{I}}$ covering $Z$ with each $n_i\ge N$.
Again it is obvious that $\M_{N, \varepsilon}^{s}(Z)$ increases as the parameter $N$ increases and the parameter $\varepsilon$ decreases. We put
$$
\M_{\varepsilon}^{s}(Z)=\lim _{N \rightarrow \infty} \M_{N, \varepsilon}^{s}(Z) = \sup_{N\in \mathbb{N}} \M_{N, \varepsilon}^{s}(Z),
$$
which increases as the parameter $\varepsilon$ decreases. Then there exists a critical value of the parameter $s> 0$, denoted by $h_{\mathrm {top }}^{B}(T, Z, \varepsilon)$, where $\mathcal{M}_{\varepsilon}^{s}(Z)$ jumps from $\infty$ to 0, that is,
$$
\mathcal{M}_{\varepsilon}^{s}(Z)= \begin{cases}0, & \ \ \ \ \ \ \ s>h_{\mathrm {top }}^{B}(T, Z, \varepsilon), \\ \infty, & \ \ \ \ \ \ \ s<h_{\mathrm {top }}^{B}(T, Z, \varepsilon).\end{cases}
$$
By convention we set
$$
h_{\mathrm {top }}^{B}(T, Z, \varepsilon)= \begin{cases}0, & \ \ \ \ \ \ \ \text{if}\ \mathcal{M}_{\varepsilon}^{s}(Z) = 0\ \text{for each}\ s> 0, \\ \infty, & \ \ \ \ \ \ \ \text{if}\ \mathcal{M}_{\varepsilon}^{s}(Z) = \infty\ \text{for each}\ s> 0.\end{cases}
$$
Again the quantity $h_{\mathrm {top }}^{B}(T, Z, \varepsilon)$ increases as the parameter $\varepsilon$ decreases, and so we may introduce the \emph{Bowen topological entropy of $Z$ with respect to $T$} or, simply, the \emph{Bowen topological entropy of $Z$} if there is no confusion about the map $T$, as follows
$$
h_{\mathrm {top }}^{B}(T, Z)=\lim _{\varepsilon \rightarrow 0} h_{\mathrm {top }}^{B}(T, Z, \varepsilon) = \sup_{\varepsilon > 0} h_{\mathrm {top }}^{B}(T, Z, \varepsilon).
$$
Similarly, though $h_{\mathrm {top }}^{B}(T, Z, \varepsilon)$ does depend on the metric $d$, the quantity $h_{\mathrm {top }}^{B}(T, Z)$ is independent of selection of compatible metrics over the space.

More generally, we could give a definition which resembles Hausdorff measure corresponding to some gauge function as follows.
Let $b: \N\to \mathbb{R}_+$ be a gauge function and $Z\subset X$. By extending the idea of defining $\M_{N, \varepsilon}^{s}(Z)$, we define
$$
\M_{N, \varepsilon}^{b}(Z)=\inf\ \sum_{i\in \mathcal{I}} b(n_i),
$$
where the infimum is taken over all countable families $\left\{ B_{n_{i}}\left(x_{i}, \varepsilon\right)\right\}_{i\in \mathcal{I}}$ covering $Z$ with each $n_i\ge N$.
Similarly, the quantity $\M_{N, \varepsilon}^{b}(Z)$ increases as the parameter $N$ increases and the parameter $\varepsilon$ decreases, and so we could define analogously
$$
\M_{\varepsilon}^{b}(Z)=\lim _{N \rightarrow \infty} \M_{N, \varepsilon}^{b}(Z)=\sup_{N \in \mathbb{N}} \M_{N, \varepsilon}^{b}(Z), \quad \M^b(Z)=\lim_{\varepsilon\rightarrow0} \M_{\varepsilon}^{b}(Z)= \sup_{\varepsilon> 0} \M_{\varepsilon}^{b}(Z).
$$

Similar to Proposition \ref{comp_gauge_Haus}, we have the following fact:

\begin{lemma}\label{com}
Let $K\subset X, \varepsilon>0$, and let $b$ and $b^*$ be gauge functions satisfying $\lim\limits_{n\to \infty}\frac{b^*(n)}{b(n)}=0$. If $\M_{\varepsilon}^{b} (K) < + \infty$ then $\M_{\varepsilon}^{b^*} (K) = 0$. In particular, if $\M^{b} (K) < + \infty$ then $\M^{b^*} (K) = 0$.
\end{lemma}

\begin{proof}
Fix arbitrarily given $\eta> 0$. By the assumption $\lim\limits_{n\to \infty}\frac{b^*(n)}{b(n)}=0$, we may choose $N\in \mathbb{N}$ large enough such that $\frac{b^*(n)}{b(n)} < \frac{\eta}{\M_{\varepsilon}^{b} (K) + 1}$ for all $n\ge N$. As $\M_{\varepsilon}^{b} (K) < + \infty$, for each integer $m\ge N$, we take a countable family $\left\{ B_{n_{i}}\left(x_{i}, \varepsilon\right)\right\}_{i\in \mathcal{I}}$ covering $K$ with each $n_i\ge m$ (and then $n_i\ge N$) satisfying $\sum\limits_{i\in \mathcal{I}} b(n_i) < \M_{\varepsilon}^{b} (K) + 1$, thus we have
$$\sum_{i\in \mathcal{I}} b^*(n_i) = \sum_{i\in \mathcal{I}} b(n_i)\cdot \frac{b^* (n_i)}{b (n_i)}\le \frac{\eta}{\M_{\varepsilon}^{b} (K) + 1} \sum_{i\in \mathcal{I}} b(n_i) < \eta.$$
So we conclude $\M_{\varepsilon}^{b^*} (K) \le \eta$ from arbitrariness of $m\ge N$, and then $\M_{\varepsilon}^{b^*} (K) = 0$ from arbitrariness of $\eta> 0$. This finishes the proof.
\end{proof}

We remark that, similar to packing topological entropy, the quantity $\M^b(Z)$ with respect to a given gauge function $b$ (including the quantity $h_{\mathrm {top }}^{B}(T, Z)$) is independent of selection of compatible metrics over the compact metric space.

\subsection{Measure-theoretical upper and lower entropies}\

\smallskip

Now let us recall the concepts of measure-theoretical upper and lower entropies following \cite{Feng-Huang2012}, which is motivated by the local entropy given by Brin and Katok in \cite{Brin-Katok1981}.

Let $\mu\in \mathcal{M} (X)$ and $\varepsilon> 0$. It is easy to see that for each $n\in \mathbb{N}$ the function
$$X \rightarrow [0, 1], x\mapsto \mu(B_n(x, \varepsilon))$$
is lower semi-continuous and so Borel measurable. So the following definition makes sense, as all involved functions are Borel measurable (by convention $\log 0= - \infty$ and $0 \log 0 = 0$).

\begin{definition}
Let $\mu \in \M(X)$.
For each $x\in X$ we set
\begin{eqnarray*}
\overline{h}_{\mu}(T, x) & = & \lim _{\varepsilon \rightarrow 0}  \overline{h}_{\mu}(T, x, \varepsilon),\ \ \ \text{with}\ \overline{h}_{\mu}(T, x, \varepsilon) = \limsup _{n \rightarrow\infty}-\frac{1}{n} \log \mu\left(B_{n}(x, \varepsilon)\right), \\
	\underline{h}_{\mu}(T, x) & = & \lim _{\varepsilon \rightarrow 0} \underline{h}_{\mu}(T, x, \varepsilon),\ \ \ \text{with}\ \underline{h}_{\mu}(T, x, \varepsilon) = \liminf _{n \rightarrow\infty}-\frac{1}{n} \log \mu\left(B_{n}(x, \varepsilon)\right).
\end{eqnarray*}
Then the \emph{measure-theoretical upper and lower entropies of $\mu$} are defined respectively by
$$
\overline{h}_{\mu}(T)=\int_X \overline{h}_{\mu}(T, x) d \mu(x)\ \ \text{and}\ \ \underline{h}_{\mu}(T)=\int_X \underline{h}_{\mu}(T, x) d \mu(x).
$$
\end{definition}

In 1981 Brin and Katok proved in \cite{Brin-Katok1981} that, for any $\mu \in \M(X, T)$, $\underline{h}_{\mu}(T, x)=\overline{h}_{\mu}(T, x)$ for $\mu$-a.e. $x \in X$ which is $T$-invariant, and $\underline{h}_{\mu}(T)=\overline{h}_{\mu}(T)=h_{\mu}(T)$, where $h_{\mu}(T)$ denotes the measure-theoretical $\mu$-entropy of $(X, T)$\footnote{\ As the measure-theoretical $\mu$-entropy (for a TDS $(X, T)$ and an invariant measure $\mu\in \mathcal{M} (X, T)$) plays only a marginal role in this paper, we shall skip its detailed definition.}. In particular, if $\mu \in \M(X, T)$ is ergodic then
\begin{equation} \label{f-2}
\underline{h}_{\mu}(T, x)=\overline{h}_{\mu}(T, x)=h_{\mu}(T)\ \ \ \text{for $\mu$-a.e. $x \in X$}.
\end{equation}
It was proved in \cite{Feng-Huang2012} that, for any analytic set $\emptyset\neq Z\subset X$ in a TDS $(X, T)$, one has\footnote{\ We remark that, though the second identity in \eqref{varia.prin} was proved in \cite{Feng-Huang2012} under the additional assumption that TDS $(X,T)$ has finite topological entropy, in fact it holds for arbitrary TDS with the help of a fact, due to Downarowicz and Huczek, that every TDS has a zero-dimensional principal extension \cite{Downarowicz-Huczek2013}.}
\begin{equation} \label{varia.prin}
h_{\mathrm {top }}^P(T, Z)\ = \ \sup_{\mu \in \M(X), \mu(Z)=1} \ \overline{h}_\mu(T)\quad \text{and} \quad h_{\mathrm {top }}^B(T, Z)\ = \ \sup_{\mu \in \M(X), \mu(Z)=1} \ \underline{h}_\mu(T).
\end{equation}

\subsection{$h$-expansive system and asymptotically $h$-expansive system}
\

\smallskip

Finally let us recall the notions of $h$-expansiveness and asymptotically $h$-expansiveness from \cite{R.Bowen1972-TAMS, Misiurewicz1973}. We need firstly to define topological entropy of subsets in a dynamical system.

Let $Z \subset X$. For $n\in \mathbb{N}$ and $\varepsilon> 0$, a subset $E\subset X$ is called to \emph{$(n, \varepsilon)$-span $Z$}, if for each $x \in Z$, there is $y \in E$ satisfying $d_n (x, y) \leq \varepsilon$; denote by $r_n(\varepsilon, Z)$ the minimum cardinality of a set which $(n, \varepsilon)$-spans $Z$. And then the \emph{topological entropy of $Z$ with respect to $T$} or, simply, the \emph{topological entropy of $Z$} if there is no confusion about the map $T$, is defined as
$$
h_{\mathrm{top}}(T, Z)=\lim\limits_{\varepsilon\to 0}\limsup\limits_{n \to \infty} \frac{1}{n} \log r_n(\varepsilon, Z).
$$
We also call $h_{\mathrm{top}}(T, X)$ the \emph{topological entropy of $(X, T)$} and denote it by $h_{\mathrm{top}}(T)$.

The following basic fact was proved as \cite[Proposition 2.1]{Feng-Huang2012}, except the statement about gauge function in the second item which follows directly from the definition.

\begin{prop} \label{fact}
Let $(X, T)$ be a TDS, and $b: \N\to \mathbb{R}_+$ be a gauge function.
\begin{enumerate}

\item If $Z \subset Z'\subset X$, then
$$h_{\mathrm{top }}^P (T, Z) \le h_{\mathrm{top}}^P (T, Z'), \quad h_{\mathrm{top}}^B(T, Z) \le h_{\mathrm{top}}^B (T, Z'), \quad h_{\mathrm{top}} (T, Z) \le h_{\mathrm{top }} (T, Z^{\prime}).$$

\item Assume that $Z \subset \bigcup_{i=1}^{\infty} Z_i\subset X$  and $s \ge 0, \varepsilon>0$. Then
$$
\begin{gathered}
\mathcal{M}_\varepsilon^b(Z) \le \sum_{i=1}^{\infty} \mathcal{M}_\varepsilon^b\left(Z_i\right), \quad \mathcal{M}_\varepsilon^s(Z) \le \sum_{i=1}^{\infty} \mathcal{M}_\varepsilon^s\left(Z_i\right), \\
h_{\mathrm {top }}^P(T, Z) \le \sup _{i \ge 1} h_{\mathrm {top }}^P\left(T, Z_i\right), \quad h_{\mathrm {top }}^B(T, Z) \le \sup _{i \ge 1} h_{\mathrm {top }}^B\left(T, Z_i\right).
\end{gathered}
$$

\item For any $Z \subset X$ one has $h_{\mathrm {top }}^B(T, Z) \le h_{\mathrm{top}}^P(T, Z) \le h_{\mathrm {top }} (T, Z)$. Furthermore, if $Z$ is a compact subset satisfying $T (Z)\subset Z$, then $h_{\mathrm {top }}^B(T, Z) = h_{\mathrm{top}}^P(T, Z) = h_{\mathrm {top }} (T, Z)$.
\end{enumerate}
\end{prop}

Let $(X, T)$ and $(Y, S)$ be TDSs. By a \emph{factor map} or an \emph{extension} we mean a continuous surjection $\pi:Y \rightarrow X$ satisfying $\pi (S y) = T (\pi y)$ for each $y\in Y$, which is also denoted by $\pi:(Y, S) \rightarrow(X, T)$. In this case, we also call $(Y, S)$ to be an \emph{extension} of $(X, T)$.

We shall also use the following result proved as \cite[Theorem 4.1]{Oprocha-Zhang2011}.

\begin{prop}\label{Bowen_ent_factor_inequality}
Let $\pi:(Y, S) \rightarrow(X, T)$ be a factor map between TDSs. Then
	$$
	h_{\mathrm {top }}^B(T, \pi(E)) \le h_{\mathrm {top }}^B(S, E) \le h_{\mathrm {top }}^B(T, \pi(E))+\sup _{x \in X} h_{\mathrm {top }}\left(S, \pi^{-1}(x)\right),\ \ \forall E \subset Y.
	$$
	
\end{prop}

A TDS $(X, T)$ is \emph{$h$-expansive} if $h^*_T (\varepsilon) = 0$ for some $\varepsilon> 0$, where
$$h^*_T (\varepsilon) = \sup_{x\in X} h_{\mathrm{top}}(T, \Phi_\varepsilon (x))\ \ \ \ \text{with}\ \ \ \ \Phi_\varepsilon (x) =
\bigcap_{i\in \mathbb{Z}_+} T^{-i} \overline{B} (T^i x, \varepsilon);$$
and is \emph{asymptotically $h$-expansive} if $\lim_{\delta\rightarrow 0} h^*_T (\delta) = 0$. Any expansive TDS is clearly $h$-expansive and any $h$-expansive TDS is asymptotically $h$-expansive.

The following fact is not hard to obtain \cite[Corollary 1.4]{Wang2022}:

\begin{lemma} \label{wangtao expan}
Assume that TDS $(X, T)$ is $h$-expansive. Then there exists $\varepsilon>0$ such that
$$h_\mathrm{top}^P(T,Z)=h_\mathrm{top}^P(T,Z,\varepsilon)\ \ \ \text{and}\ \ \ h_\mathrm{top}^B(T,Z)=h_\mathrm{top}^B(T,Z,\varepsilon)\ \ \ \text{for any}\ Z\subset X.$$
\end{lemma}

\section{Basic properties concerning entropy of analytic sets} \label{basic}

In this section, we discuss increasing countable slice of packing and Bowen entropy for
subsets, we also study measure-theoretical upper and lower entropies for Borel probability measures.
We prove that, given measure $\mu$ of maximal packing (Bowen, respectively) entropy, the local measure-theoretical entropy at $x$ for $\mu$-a.e. $x\in X$ is equal to the measure-theoretical entropy. Consequently, each analytic set has many
measures of maximal entropy once it has at least one such measure no matter if we consider packing or Bowen entropy (cf. Proposition \ref{uniqueness} for details), which is an easy evidence showing that the situation of considering measures of maximal entropy for analytic sets is completely different from that of the study of invariant measures of maximal entropy for dynamical systems.

\smallskip

We shall prove firstly the following easy property about increasing countable slices.

\begin{prop} \label{prop slice}
Let $\{U_i\}_{i\in \mathbb{N}}$ be a countable basis of $X$, and $Z\subset X$ be an analytic set with $h_\mathrm{top}^P(T, Z)> 0$ which has no increasing countable slice of packing entropy.
	\begin{enumerate}
	
		\item \label{slice1}
Assume that $Z\subset Z_1\subset X$ satisfies $h_\mathrm{top}^P(T, Z) = h_\mathrm{top}^P(T, Z_1)$. Then $Z_1$ has no increasing countable slice of packing entropy.

		\item \label{slice2}
Assume that $Z \subset \bigcup_{i\in \mathbb{N}} Z_i$ satisfies $h_{\mathrm {top }}^P(T, Z) = \sup_{i\in \mathbb{N}} h_{\mathrm {top }}^P (T, Z_i)$, where $\{Z_i\}_{i\in \mathbb{N}}$ is a sequence of analytic subsets of $X$. Then there exists $i\in \mathbb{N}$ such that $h_\mathrm{top}^P(T, Z)$ $= h_\mathrm{top}^P(T, Z_i)$ and $Z_i$ has no increasing countable slice of packing entropy.
		
		\item \label{slice3}
Then there exists a nonempty analytic set $Z'\subset Z$ such that
 \begin{enumerate}

 \item \label{add-1}
 for any $i\geq 1$, once $Z'\cap U_i\neq \emptyset$, we have that $h_\mathrm{top}^P(T, Z'\cap U_i) = h_\mathrm{top}^P (T, Z') = h_\mathrm{top}^P (T, Z)$ and $Z'\cap U_i$ has no increasing countable slice of packing entropy;

 \item \label{add-2} if an open set $U$ satisfies $Z'\cap U\neq \emptyset$ then $h_\mathrm{top}^P(T, Z'\cap U)  = h_\mathrm{top}^P (T, Z')$ $=~ h_\mathrm{top}^P (T, Z)$ and $Z'\cap U$ has no increasing countable slice of packing entropy.
 \end{enumerate}

	\end{enumerate}
The same conclusion holds if alternatively we consider Bowen entropy.
\end{prop}

\begin{proof}
 We only prove it for packing entropy. The same arguments work for Bowen entropy.

\eqref{slice1} Assume the contrary that $Z_1$ has an increasing countable slice of packing entropy, and so by the definition there exists a countable collection $\{Z_j^*\}_{j\in \mathbb{N}}$ of analytic sets such that $Z_1=\bigcup_{j\in \mathbb{N}} Z_j^*$ and $h_\mathrm{top}^P(T, Z_j^*)<h_\mathrm{top}^P(T, Z_1)$ for each $j\in \mathbb{N}$. Thus
$$Z=\bigcup_{j\in \mathbb{N}} Z_j^*\cap Z\ \ \text{and hence}\ \ h_\mathrm{top}^P(T, Z)= \sup_{j\in \mathbb{N}} h_\mathrm{top}^P (T, Z_j^*\cap Z)\ (\text{by Proposition \ref{fact}}).$$
Note that for each $j\in \mathbb{N}$ the subset $Z\cap Z_j^*$ is also an analytic set and
$$h_\mathrm{top}^P(T, Z)= h_\mathrm{top}^P(T, Z_1)> h_\mathrm{top}^P (T, Z_j^*) \ge h_\mathrm{top}^P (T, Z_j^*\cap Z).$$
In particular, the set $Z$ has an increasing countable slice of packing entropy, a contradiction to the assumption. Thus $Z_1$ has no increasing countable slice of packing entropy.

\eqref{slice2} As the set $Z$ has no increasing countable slice of packing entropy, by applying Proposition \ref{fact} to the assumption one has that there exists some $i\in \mathbb{N}$ such that
$$h_\mathrm{top}^P(T, Z) = h_\mathrm{top}^P(T, Z\cap Z_i)\ \ \text{and hence}\ \ h_\mathrm{top}^P(T, Z) = h_\mathrm{top}^P(T, Z_i),$$
as each $Z\cap Z_i$ is an analytic set.
If we assume the contrary that
for each $i\in \mathbb{N}$, once $h_\mathrm{top}^P(T, Z) = h_\mathrm{top}^P(T, Z_i)$ then the set $Z_i$ has an increasing countable slice of packing entropy.
Then for each $i\in \mathbb{N}$ there exists a countable collection $\{Z_{i, j}\}_{j\in \mathbb{N}}$ of analytic sets such that
$$
Z_i=\bigcup_{j\in \mathbb{N}} Z_{i, j}\ \text{and}\ h_\mathrm{top}^P(T, Z_{i, j})<h_\mathrm{top}^P(T, Z)\ \text{for each}\ j\in \mathbb{N}.
\footnote{\ If $h_\mathrm{top}^P(T, Z_i) < h_\mathrm{top}^P(T, Z)$, we take $Z_{i, j}= Z_i$ for each $j\in \mathbb{N}$; if $h_\mathrm{top}^P(T, Z_i) = h_\mathrm{top}^P(T, Z)$, then the set $Z_i$ has an increasing countable slice of packing entropy by the assumption, and so there exists a countable collection $\{Z_{i, j}\}_{j\in \mathbb{N}}$ of analytic sets such that $Z_i=\bigcup_{j\in \mathbb{N}} Z_{i, j}$ and $h_\mathrm{top}^P(T, Z_{i, j})<h_\mathrm{top}^P(T, Z_i)$ for each $j\in \mathbb{N}$.}
$$
In particular, the set $Z$ has an increasing countable slice of packing topological entropy (with a countable collection $\{Z_{i, j}\}_{i, j\in \mathbb{N}}$ of analytic sets), a contradiction to the assumption.

\eqref{add-1} Denote by $\mathcal{I}$ the set of all $i\in \mathbb{N}$ such that $Z\cap U_i$ has no increasing countable slice of packing entropy and $h_\mathrm{top}^P (T, Z\cap U_i) = h_\mathrm{top}^P (T, Z)$ (thus $Z\cap U_i\neq \emptyset$). Then $\mathcal{I}\neq \emptyset$ by item \eqref{slice2}.
For each $i\notin \mathcal{I}$ one has that, either $h_\mathrm{top}^P (T, Z\cap U_i) < h_\mathrm{top}^P (T, Z)$, or $h_\mathrm{top}^P (T, Z\cap U_i) = h_\mathrm{top}^P (T, Z)$ and $Z\cap U_i$ has an increasing countable slice of packing entropy.
Put
 $$Z^*= Z\cap \bigcup_{j\in \mathbb{N}\setminus \mathcal{I}} U_j = \bigcup_{j\in \mathbb{N}\setminus \mathcal{I}} (Z\cap U_j).$$
 Then either $h_\mathrm{top}^P (T, Z^*) < h_\mathrm{top}^P (T, Z)$, or $h_\mathrm{top}^P (T, Z^*) = h_\mathrm{top}^P (T, Z)$ and $Z^*$ has an increasing countable slice of packing entropy by the construction of the collection $\mathcal{I}$. Set $Z_i= (Z\cap U_i)\setminus Z^*$ for each $i\in \mathcal{I}$, which is clearly analytic and by the construction one has $h_\mathrm{top}^P (T, Z_i) = h_\mathrm{top}^P (T, Z\cap U_i)$ (and hence $h_\mathrm{top}^P (T, Z_i) =  h_\mathrm{top}^P (T, Z)$, in particular, $Z_i\neq \emptyset$) and $Z_i$ has no increasing countable slice of packing entropy.
 We define a nonempty analytic set as follows $$Z' = \bigcup_{i\in \mathcal{I}} Z_i.$$

 Now let us show that the subset $Z'$ is a required set. For each $i\in \mathcal{I}$, $h_\mathrm{top}^P (T, Z\cap U_i) = h_\mathrm{top}^P (T, Z)$, and $Z\cap U_i\ (\subset Z_i\cup Z^*)$ has no increasing countable slice of packing entropy and
$$h_{\mathrm {top }}^P(T, Z\cap U_i) = h_{\mathrm {top }}^P(T, Z) = \max \{h_{\mathrm {top }}^P (T, Z_i), h_{\mathrm {top }}^P (T, Z^*)\}.$$
As $Z'\cap Z^*= \emptyset$, one has that, given $i\in \mathbb{N}$, $Z'\cap U_i\neq \emptyset$ if and only if $i\in \mathcal{I}$. For each $i\in \mathcal{I}$, clearly $Z'\cap U_i = Z_i$ and hence $Z'\cap U_i$ has no increasing countable slice of packing entropy.

\eqref{add-2} Let $Z^{\prime} \subset Z$ be the nonempty analytic set constructed in \eqref{add-1}. Now assume that $U$ is an open set satisfying $Z'\cap U\neq \emptyset$. As $\{U_i\}_{i\in \mathbb{N}}$ is a countable basis of $X$, then $\emptyset\neq Z'\cap U_i\subset Z'\cap U\subset Z$ for some $i\in \mathbb{N}$, and so by the above construction \eqref{add-1} one has
$$h_\mathrm{top}^P(T, Z'\cap U_i) = h_\mathrm{top}^P(T, Z'\cap U) = h_\mathrm{top}^P (T, Z)$$ and that $Z'\cap U_i$ has no increasing countable slice of packing entropy, thus $Z'\cap U$ has no increasing countable slice of packing entropy by applying the item \eqref{slice1} to $Z'\cap U_i\subset Z'\cap U$.
 This finishes our proof.
\end{proof}

We also have the following property for measures of maximal topological entropy. Recall that $\mu\in\M(X)$ is \emph{non-atomic} if  $\mu (\{x\}) = 0$ for each $x\in X$.

\begin{prop} \label{dynamical density thm}
	Let $Z\subset X$ be an analytic set and $\mu\in\M(X)$ with $\mu(Z)=1$.
 \begin{enumerate}

 \item Assume $\overline{h}_\mu(T)=h_\mathrm{top}^P(T, Z)$. Then $\overline{h}_{\mu}(T, x) = h_\mathrm{top}^P(T, Z)$ for $\mu$-a.e. $x\in X$. In particular, if in addition $h_\mathrm{top}^P(T, Z) > 0$ then $\mu$ is non-atomic.

 \item Assume $\underline{h}_\mu(T)=h_\mathrm{top}^B(T, Z)$. Then $\underline{h}_{\mu}(T, x) = h_\mathrm{top}^B(T, Z)$ for $\mu$-a.e. $x\in X$. In particular, if in addition $h_\mathrm{top}^B(T, Z) > 0$ then $\mu$ is non-atomic.
 \end{enumerate}
\end{prop}

\begin{proof}
	Similarly we shall only prove the conclusion for packing entropy.

\smallskip

Firstly let us prove $\overline{h}_{\mu}(T, x) = h_\mathrm{top}^P(T, Z)$ for $\mu$-a.e. $x\in X$.
	As $\mu (Z) = 1$ and
$$\overline{h}_{\mu}(T)=\int_X \overline{h}_{\mu}(T, x) d \mu(x) = h_\mathrm{top}^P(T, Z)$$
by the assumption,
it suffices to show that $\mu(E)=0$, where
$$E \doteq \left\{x\in Z: \overline{h}_{\mu}(T, x)> h_\mathrm{top}^P(T, Z)\right\}\ \ \text{is an analytic set}.$$
	
	Assume the contrary that $\mu(E)>0$, and let $\nu$ be the normalized measure of $\mu$ restricted onto $E$ (in particular, $\nu\in \M (X)$ and $\nu (E)= 1$). For each $x\in X$ and any $\varepsilon>0$ one has
	\begin{equation*}
		\begin{aligned}
			\limsup_{n \rightarrow\infty}-\frac{1}{n} \log \nu\left(B_{n}(x, \varepsilon)\right)
			&=\limsup_{n \rightarrow\infty}-\frac{1}{n} \log \frac{\mu\left(B_{n}(x, \varepsilon)\cap E\right) }{\mu(E)}\\
			&\geq \limsup _{n \rightarrow\infty}-\frac{1}{n} \log \frac{\mu\left(B_{n}(x, \varepsilon)\right) }{\mu(E)} = \limsup _{n \rightarrow\infty}-\frac{1}{n} \log \mu\left(B_{n}(x, \varepsilon)\right)
		\end{aligned}
	\end{equation*}
and so $\overline{h}_{\nu}(T, x) \ge \overline{h}_{\mu}(T, x)$. As $\nu$ is the normalized measure of $\mu$ restricted onto $E$, then 		\begin{equation} \label{f-1}
			\overline{h}_{\nu}(T) = \frac{1}{\mu(E)} \int_E \overline{h}_{\nu}(T, x) d \mu(x) \geq \frac{1}{\mu(E)} \int_E \overline{h}_{\mu}(T, x) d \mu(x) > h_\mathrm{top}^P(T, Z),
	\end{equation}
 where the last inequality follows from the construction of $E$. This contradicts to the variational principle \eqref{varia.prin}, and so we obtain $\mu (E) = 0$.	

 \smallskip

Now assume in addition $h_\mathrm{top}^P(T, Z) > 0$. It is easy to show that $\mu (\{x\}) = 0$ for each $x\in X$. In fact, if we assume the contrary that $\mu (\{x_0\}) > 0$ for some $x_0\in X$, by definition one has that $\overline{h}_{\mu}(T, x_0) = 0$, which is a contradiction to the conclusion that $\overline{h}_{\mu}(T, x) = h_\mathrm{top}^P(T, Z)> 0$ for $\mu$-a.e. $x\in X$. This finishes our proof.
\end{proof}

As a direct corollary of Proposition \ref{dynamical density thm}, one has the following useful fact. We remark that Proposition \ref{uniqueness} follows directly from Proposition \ref{dynamical density thm} and Lemma \ref{subset}.

\begin{lemma} \label{subset}
	Let $Y\subset Z$ both be analytic sets in $X$, and $\mu\in\M(X)$ with $\mu(Z)=1$ and $\mu (Y) > 0$. Let $\nu$ be the normalized measure of $\mu$ restricted onto $Y$.
 \begin{enumerate}

 \item Assume $\overline{h}_\mu(T)=h_\mathrm{top}^P(T, Z)$. Then $\overline{h}_\nu(T) = h_\mathrm{top}^P(T, Z) = h_\mathrm{top}^P(T, Y)$, in particular, if $h_{\mathrm {top }}^P(T, Z)> 0$ then the subset $Y$ has measures of maximal packing entropy.

 \item Assume $\underline{h}_\mu(T)=h_\mathrm{top}^B(T, Z)$. Then $\underline{h}_\nu(T) = h_\mathrm{top}^B(T, Z) = h_\mathrm{top}^B(T, Y)$, in particular, if $h_{\mathrm {top }}^B(T, Z)> 0$ then the subset $Y$ has measures of maximal Bowen entropy.
 \end{enumerate}
\end{lemma}
\begin{proof}
Again we only prove the conclusion for packing entropy.
As $\overline{h}_\mu(T)=h_\mathrm{top}^P(T, Z)$, by Proposition \ref{dynamical density thm} one has $\overline{h}_\mu(T, x)=h_\mathrm{top}^P(T, Z)$ for $\mu$-a.e. $x\in X$, and then $\overline{h}_\nu(T, x)=h_\mathrm{top}^P (T, Z)$ for $\nu$-a.e. $x\in X$ by the construction of $\nu$. Thus $h_\mathrm{top}^P (T, Y)\ge \overline{h}_\nu(T) = h_\mathrm{top}^P (T, Z)$ by the variational principle \eqref{varia.prin}, as $\nu (Y)= 1$.
 It is trivial $h_\mathrm{top}^P(T, Z) \ge h_\mathrm{top}^P(T, Y)$, and then the identities $\overline{h}_\nu(T) = h_\mathrm{top}^P(T, Z) = h_\mathrm{top}^P(T, Y)$ follow readily.
\end{proof}

The following result shows that the converse of Proposition \ref{dynamical density thm} also holds.

\begin{prop}\label{ent_conn_meas_subset}
	Let $E\subset X$ be an analytic set and $\mu\in\M(X)$.
	\begin{enumerate}

		\item If $\overline{h}_\mu(T,x)\leq s$ for all $x\in E$, then $h_{\mathrm{top}}^P(T, E)\leq s$; if $\mu(E)>0$ and, additionally, $\overline{h}_\mu(T,x)\geq s$ for all $x\in E$, then $h_{\mathrm{top}}^P(T, E)\geq s$.

		\item If $\underline{h}_\mu(T,x)\leq s$ for all $x\in E$, then $h_{\mathrm{top}}^B(T, E)\leq s$; if $\mu(E)>0$ and, additionally, $\underline{h}_\mu(T,x)\geq s$ for all $x\in E$, then $h_{\mathrm{top}}^B(T, E)\geq s$.
	\end{enumerate}
\end{prop}

\begin{proof}
	As the second item is just \cite[Theorem 1.1]{Ma-Wen2008}, it suffices to prove the first item.
	
We assume firstly that $\overline{h}_\mu(T,x)\leq s$ for all $x\in E$, and aim to prove $h_{\mathrm{top}}^P(T, E)\leq s$.
 We fix arbitrarily given $\eta>0$.
From the definition of $\overline{h}_\mu(T,x)$, it is easy to see that the
subset $E$ can be written as a countable union of $\{E_m\}_{m\in \mathbb{N}}$, where each $E_m$ is defined as
$$\bigcap_{n\geq m}\left\lbrace x\in E: \mu\left(B_{n}(x, \varepsilon)\right)> e^{-n(s+\eta)} \right\rbrace .$$
By Proposition \ref{fact} (2), it suffices to prove $h_\mathrm{top}^P(T, E_m)\leq s+\eta$ for all $m\in \mathbb{N}$. We fix each $m\in \mathbb{N}$ and any $\varepsilon> 0$. For any $N\in \mathbb{N}$ with $N\geq m$, we take arbitrarily a countable pairwise disjoint family $\left\{\overline{B}_{n_{i}}\left(x_{i}, \varepsilon\right)\right\}_{i\in \mathcal{I}}$ such that $x_{i} \in E_m$ and $n_{i} \ge N$ for all $i$, one has
$$\sum_{i\in \mathcal{I}} e^{-n_i(s+\eta)} \leq \sum_{i\in \mathcal{I}} \mu\left(B_{n_i}(x_i, \varepsilon)\right) \leq \sum_{i\in \mathcal{I}}  \mu\left(\overline{B}_{n_i}(x_i, \varepsilon)\right)\leq 1,$$
which implies $P^{s+\eta}_{N,\varepsilon}(E_m)\leq 1$. Thus $\mathcal{P}^{s+\eta}_{\varepsilon}(E_m)\leq P^{s+\eta}_{\varepsilon}(E_m)\leq 1$ and then $h_\mathrm{top}^P(T, E_m,\varepsilon)\leq s+\eta$ by the definition. We obtain $h_\mathrm{top}^P(T, E_m)\leq s+\eta$ from the arbitrariness of $\varepsilon>0$.

Now we assume that $\mu(E)>0$ and, additionally, $\overline{h}_\mu(T,x)\geq s$ for all $x\in E$. As we did for \eqref{f-1} in Proposition \ref{dynamical density thm},
if let $\nu$ be the normalized measure of $\mu$ restricted onto $E$, then
 $$\overline{h}_{\nu}(T) \geq \frac{1}{\mu(E)} \int_E \overline{h}_{\mu}(T, x) d \mu(x) \ge s.$$
 Noting that $E\subset X$ is an analytic set and that $\nu\in \M (X)$ satisfies $\nu (E) =1$, we obtain $h_\mathrm{top}^P(T, E)\geq \overline{h}_{\nu}(T)\geq s$ by the variational principle (\ref{varia.prin}). This finishes the proof.
\end{proof}

\section{Proof of easier parts of our results} \label{easy}

In this section,  we prove easier parts of our results. Namely, we prove Theorem \ref{general-system}, $(1) \Rightarrow (2)$ of Theorem \ref{MME Packing h-expan}, $(1) \Rightarrow (2)$ and $(3) \Rightarrow (2)$ of Theorem \ref{MME Bowen}, and Corollary \ref{cor_countable_slice}; we also prove Corollary \ref{Hausdorff} assuming Theorem \ref{MME Bowen}.

\smallskip

The following result is a variation of \cite[Theorem 4.4]{Huang-Ye-Zhang2014}\footnote{\ \cite[Theorem 4.4]{Huang-Ye-Zhang2014} was proved for an invertible TDS $(X,T)$, that is, $X$ is a compact metric space and $T:X\rightarrow X$ is a homeomorphism.
However, as one did in \cite[Appendix A]{Huang-Ye-Zhang2010}, \cite[Theorem 4.4]{Huang-Ye-Zhang2014}
can be generalized for any TDS $(X, T)$, with the help of Theorem \ref{Bowen_ent_factor_inequality} (that is, \cite[Theorem 4.1]{Oprocha-Zhang2011}).} (combined with Proposition \ref{fact}).

\begin{lemma}\label{lowering-entropy-thm}
For each $0 \leq h \leq h_\mathrm{top}(T,X)$ there exists a compact subset $K_h \subset X$ such that $h_\mathrm{top}\left(T, K_h\right)= h_\mathrm{top}^P\left(T, K_h\right)= h_\mathrm{top}^B\left(T, K_h\right)=h$.
	\end{lemma}

Now we are ready to prove Theorem \ref{general-system}, which tells us that, each TDS with positive topological entropy contains a Borel subset such that the subset has measures that are both of maximal packing and Bowen entropies, and the system also contains many Borel subsets such that these subsets have no measures of maximal packing or Bowen entropy.

\begin{proof}[Proof of Theorem \ref{general-system}]
\eqref{exists}
	As TDS $(X, T)$ has positive topological entropy, by the classical variational principle concerning entropy \eqref{VP}, there exists an ergodic invariant measure $\mu\in \M (X)$ such that $h_\mu (T)$ is sufficiently close to $h_\mathrm{top} (T)$.
By applying to $\mu$ the well-known Brin-Katok formula \eqref{f-2}, we may find a Borel subset $Z\subset X$ such that $\mu(Z)=1$ and $\underline{h}_\mu (T, x)=\overline{h}_\mu (T, x) = h_\mu(T)$ for each $x\in Z$. And then by Proposition
\ref{ent_conn_meas_subset}, one has
$$h_\mathrm{top}^P(T, Z)= \overline{h}_\mu(T)=h_\mu(T) = \underline{h}_\mu(T)=h_\mathrm{top}^B(T, Z),$$
thus $\mu$ is the measure
that is both of maximal packing and Bowen entropies for the set $Z$.

\eqref{exists-not}
Fix arbitrarily given $0 < h \leq h_\mathrm{top} (T)$, and take $0< h_1< h_2< \cdots \nearrow h$. By Lemma \ref{lowering-entropy-thm} choose for each $i\in \mathbb{N}$ a compact subset $K_i$ with $h_\mathrm{top}\left(T, K_i\right)= h_\mathrm{top}^P\left(T, K_i\right) = h_\mathrm{top}^B\left(T, K_i\right) = h_i$. Clearly $Z\doteq \bigcup_{i\in \mathbb{N}} K_i$ is a Borel set and $h_\mathrm{top}^B(T,Z)=h_\mathrm{top}^P(T,Z)=h$ by Proposition \ref{fact}. Now we show that the set $Z$ has no measure of maximal packing or Bowen entropy.
As the proof is completely the same, we will only present it for packing entropy.
	
Assume the contrary that there exists $\mu\in \mathcal{M} (X)$ which is the measure of maximal packing entropy for $Z$. Then there exists some $i\in \mathbb{N}$ with $\mu(K_i)>0$, and so $h_\mathrm{top}^P\left(T, K_i\right) = h_\mathrm{top}^P\left(T, Z\right) = h$ by Lemma \ref{subset} (1), which is a contradiction to the above construction that $h_\mathrm{top}^P\left(T, K_i\right) = h_i < h$.
\end{proof}

It is easy to obtain the direction $(1) \Rightarrow (2)$ for both Theorem \ref{MME Packing h-expan} and Theorem \ref{MME Bowen}.

\begin{proof}[Proof of $(1) \Rightarrow (2)$ for both Theorem \ref{MME Packing h-expan} and Theorem \ref{MME Bowen}]
As the proof is completely the same, we only prove $(1) \Rightarrow (2)$ of Theorem \ref{MME Packing h-expan}.
	Assume that $Z$ has a measure of maximal packing entropy. Thus, if $\{ Z_i\}_{i\in \mathbb{N}}$ is a countable collection of analytic sets satisfying $Z=\bigcup_{i\in \mathbb{N}} Z_i$, then there exists some $i_0\in \mathbb{N}$ such that $h_\mathrm{top}^P(T, Z_{i_0}) = h_\mathrm{top}^P(T, Z)$ by Lemma \ref{subset} (as we did in the proof of Theorem \ref{general-system} (2)). That is, the subset $Z$ has no increasing countable slice of packing entropy.
\end{proof}

Now let us prove $(3) \Rightarrow (2)$ of Theorem \ref{MME Bowen} and Corollary \ref{cor_countable_slice}.

 \begin{proof}[Proof of $(3) \Rightarrow (2)$ in Theorem \ref{MME Bowen}]
 By the assumption, $Z\subset X$ is an analytic subset with $h_\mathrm{top}^B(T, Z)>0$, and
 the gauge function $b$ satisfies that $\mathcal{M}^b(Z)>0$ and $\lim\limits_{n\to \infty}e^{ns}b(n)=0$ for any $s < h_\mathrm{top}^B (T, Z)$.
Note that, once a set $E$ satisfies $h_\mathrm{top}^B (T, E)<h_\mathrm{top}^B(T, Z)$, then for any given real number $s$ satisfying $h_\mathrm{top}^B(T, E)<s<h_\mathrm{top}^B(T, Z)$, one has $\mathcal{M}^s_\varepsilon (E)=0$ for each $\varepsilon> 0$ and then $\mathcal{M}^s(E)=0$, finally $\mathcal{M}^b(E)=0$ by applying Lemma \ref{com} (as $\lim\limits_{n\to \infty}e^{ns}b(n)=0$). From this we conclude that the subset $Z$ has no increasing countable slice of Bowen entropy, else there exists a countable collection $\{ Z_i\}_{i\in \mathbb{N}}$ of analytic sets such that $Z=\bigcup_{i\in \mathbb{N}} Z_i$ and $h_\mathrm{top}^B(T, Z_i)<h_\mathrm{top}^B(T, Z)$ (and then $\mathcal{M}^b(Z_i)=0$) for each $i\in \mathbb{N}$, which implies by Proposition \ref{fact} (2) that
$$\mathcal{M}^b(Z) = \sup_{\varepsilon > 0} \mathcal{M}_\varepsilon^b (Z) \le \sup_{\varepsilon > 0} \sum_{i\in \mathbb{N}} \mathcal{M}_\varepsilon^b (Z_i) \leq \sum_{i\in \mathbb{N}} \sup_{\varepsilon > 0} \mathcal{M}_\varepsilon^b (Z_i) = \sum_{i \in \mathbb{N}} \mathcal{M}^b(Z_i)=0,$$
a contradiction to the assumption $\mathcal{M}^b(Z)>0$. This finishes the proof.
\end{proof}

\begin{proof}[Proof of Corollary \ref{cor_countable_slice}]
By the assumption we take $T$-invariant $\mu\in \mathcal{M} (X)$ with $h_\mu (T) = h_{\mathrm {top }} (T)\ (= h_{\mathrm {top }}^B(T, X))$, where the second identity follows from Proposition \ref{fact} (3). And so $\underline{h}_{\mu}(T)= h_{\mu}(T) = h_{\mathrm {top }}^B(T, X)$ by the Brin-Katok formula \cite{Brin-Katok1981}, that is, the subset $X$ has measures of maximal Bowen entropy, and then by $(1) \Rightarrow (2)$ of Theorem \ref{MME Bowen}, the subset $X$ has no increasing countable slice of Bowen entropy, i.e., once the space $X$ is written as a countable union of analytic subsets $\{Z_i\}_{i\in \mathbb{N}}$ then there exists $j\in \mathbb{N}$ such that $h_\mathrm{top}^B(T, Z_j) = h_\mathrm{top}^B(T, X) = h_\mathrm{top} (T)$. Applying again Proposition \ref{fact} (3) we have
$$h_\mathrm{top} (T) = h_\mathrm{top}^B(T, Z_j) \le h_\mathrm{top}^P(T, Z_j)\le h_\mathrm{top}^P(T, X) = h_\mathrm{top} (T),$$
thus $h_\mathrm{top}^B(T, Z_j) = h_\mathrm{top}^P(T, Z_j) = h_\mathrm{top} (T)$, and so the space $X$ has neither an increasing countable slice of packing entropy nor an increasing countable slice of Bowen entropy.
\end{proof}

Given a complete separable metric space $Y$, for each $y\in Y$ and $r> 0$ denote by $B (y, r)$ the open ball with a radius $r$ centered at $y$. Before proceeding, let us recall, for each Borel probability measure $\xi$ over $Y$, the \emph{lower pointwise dimension of $\xi$ at $y\in Y$} as follows
 $$ \underline{d}_\xi(y)=\liminf_{r\to 0}\frac{\log \xi(B(y,r))}{\log r},$$
where it was introduced via closed balls in \cite[Definition 3.3.12]{Edgar1998}
(for its equivalence to definition via open balls see last paragraph in \cite[Page 46]{Edgar1998}). We could also introduce $\overline{\dim}_{\mathcal{H}}(\xi)$ and $\underline{\dim}_{\mathcal{H}}(\xi)$ for the Borel probability measure $\xi$ similarly to \eqref{add-number1} and \eqref{add-number2}.

The following was proved in \cite[Proposition 3.3.8 and Theorem 3.3.14]{Edgar1998}:

        \begin{prop}\label{EDG-THM}
        Let $\xi$ be a Borel probability measure over a complete separable metric space $Y$, and denote by $\operatorname{esssup}$ and $\operatorname{essinf}$, respectively, the essential supremum and infimum of the given function with respect to the measure $\xi$. Then
$$\overline{\dim}_{\mathcal{H}}(\xi) = \mathop{\esssup}\limits_{y\in Y} \underline{d}_\xi(y)\ \ \text{and}\ \ \underline{\dim}_{\mathcal{H}}(\xi) = \mathop{\essinf}\limits_{y\in Y}  \underline{d}_\xi(y).$$
        \end{prop}

Now we are ready to prove Corollary \ref{Hausdorff} based on Theorem \ref{MME Bowen}.

\begin{proof}[Proof of Corollary \ref{Hausdorff}]
Firstly, we assume that the set $Z$ has an increasing countable slice of Hausdorff dimension, that is, there exists a countable collection $\{ Z_i\}_{i\in \mathbb{N}}$ of analytic sets such that $Z=\bigcup_{i\in \mathbb{N}} Z_i$ and $\dim_{\mathcal{H}}(Z_i)<\dim_{\mathcal{H}}(Z)$ for each $i\in \mathbb{N}$.
For any Borel probability measure $\mu$ support on $Z$, clearly there exists $i_0\in \mathbb{N}$ satisfying $\mu (Z_{i_0})> 0$, then one has
$$\underline{\dim}_{\mathcal{H}} (\mu)\le \dim_{\mathcal{H}}(Z_{i_0}) <\dim_{\mathcal{H}}(Z)$$
 by the definition. Thus the set $Z$ has no
 measure of full lower Hausdorff dimension.

 Now we assume that the set $Z$ has no measure of full lower Hausdorff dimension. Let us show that the set $Z$ has an increasing countable slice of Hausdorff dimension.

We manage firstly the case that the set $Z$ is contained in some cube $\prod_{i=1}^n [a_i,a_i+1]$ with $a_i\in \mathbb{Z}$ for all $1\leq i\leq n$.
 It makes no difference to write $Z\subset [0,1]^n$ as the Hausdorff dimension is translation invariant.
We consider a TDS $(X, T)$ given via
$$X = \mathbb{R}^n/\mathbb{Z}^n\ \text{and}\ T: X \rightarrow X, x\mapsto 2\cdot x\ \text{mod}\ \mathbb{Z}^n.$$
It is direct to verify
$$h_\mathrm{top}^B (T, E) = \log2\cdot \dim_{\mathcal{H}} (E)\ \ \ \text{for any}\ E\subset X.$$
Thus, if we assume the contrary that the set $Z$ has no increasing countable slice of Hausdorff dimension, then it has no increasing countable slice of Bowen entropy.
Since TDS $(X, T)$ is $h$-expansive, by Theorem \ref{MME Bowen} one has that the set $Z$ has measures of maximal Bowen entropy, that is, there exists $\mu \in \mathcal{M} (X)$ such that $\mu (Z) = 1$ and $\underline{h}_\mu (T) = h_{\mathrm {top }}^B (T, Z)$, thus $$\underline{h}_\mu (T, x) = h_{\mathrm {top }}^B (T, Z)\ \ \text{for $\mu$-a.e.}\ x\in X$$
(by Proposition \ref{dynamical density thm}), and so
$$\dim_{\mathcal{H}}(Z) = \frac{1}{\log 2} h_\mathrm{top}^B (T, Z) = \frac{1}{\log 2} \underline{h}_\mu (T, x) = \underline{d}_\mu(x)\ (\text{by direct computation})$$ for $\mu$-a.e. $x\in X$. This implies by Proposition \ref{EDG-THM} that $\dim_{\mathcal{H}}(Z) = \underline{\dim}_{\mathcal{H}} (\mu)$, that is, the measure $\mu$ is a measure of full lower Hausdorff dimension for the set $Z$, a contraction to the assumption. Thus the set $Z$ has an increasing countable slice of Hausdorff dimension.

Now let us prove the conclusion for a general set $Z\subset\mathbb{R}^n$.
We write
$$Z=\bigcup_{\boldsymbol{a}\in\mathbb{Z}^n} Z_{\boldsymbol{a}}\ \ \ \text{where}\ Z_{\boldsymbol{a}} = Z\cap \prod_{i=1}^n [a_i,a_i+1]$$
for each $\boldsymbol{a}=(a_1,\cdots,a_n)\in\mathbb{Z}^n$.
If we assume the contrary that the set $Z$ has no increasing countable slice of Hausdorff dimension, then by the definition there exists $\boldsymbol{a}^*\in\mathbb{Z}^n$ such that $\dim_{\mathcal{H}}(Z)=\dim_{\mathcal{H}}(Z_{\boldsymbol{a}^*})$.
In fact, as we did in the proof of Proposition \ref{prop slice} (2), there exists $\boldsymbol{b}\in\mathbb{Z}^n$ such that $\dim_{\mathcal{H}}(Z)=\dim_{\mathcal{H}}(Z_{\boldsymbol{b}})$ and, additionally, $Z_{\boldsymbol{b}}$ has no increasing countable slice of Hausdorff dimension. Applying above arguments to the set $Z_{\boldsymbol{b}}$ one has that $Z_{\boldsymbol{b}}$ has a measure of full lower Hausdorff dimension (and let $\mu$ be such a measure), then the measure $\mu$ is also  a measure of full lower Hausdorff dimension for the set $Z$ by the definition and the fact that $Z_{\boldsymbol{b}} \subset Z$ satisfies $\dim_{\mathcal{H}}(Z)=\dim_{\mathcal{H}}(Z_{\boldsymbol{b}})$. This arrives again at a contradiction, and so $Z$ has an increasing countable slice of Hausdorff dimension.
\end{proof}

Obviously if the set $Z$ has a measure of full lower Hausdorff dimension, then it also has a measure of full upper Hausdorff dimension. However, the converse is not necessarily true. We end this section by constructing such an analytic subset $Z\subset \mathbb{R}^n$ with $\dim_{\mathcal{H}}(Z)>0$.

\begin{prop} \label{upper}
There exists a Borel subset $Z\subset [0, 1]$ such that $\dim_{\mathcal{H}}(Z) = 1$, and that it has no measure of full lower Hausdorff dimension, however, it has a measure $\mu$ of full upper Hausdorff dimension (that is, $\dim_{\mathcal{H}}(Z) = \overline{\dim}_{\mathcal{H}} (\mu)$).
\end{prop}

\begin{proof}
We consider a TDS $(X, T)$ given via
$$X = \mathbb{R}/\mathbb{Z}\ \ \text{and}\ \ T: X \rightarrow X, x\mapsto 2 x\ \text{mod}\ \mathbb{Z}.$$
 It is well known that there exists for each $0<s<1$ a compact subset $K_s\subset X$ such that $T K_s= K_s$ and $\dim_{\mathcal{H}}(K_s)=s$.
Note that $(K_s, T)$ is an $h$-expansive TDS.
As we did in the proof of Corollary \ref{cor_countable_slice}, there exists $\mu_s \in \mathcal{M} (X)$ such that $\mu_s (K_s) = 1$ and $\underline{h}_{\mu_s} (T) = h_{\mathrm {top }}^B(T, K_s) = h_{\mathrm {top }} (T, K_s)$ (the last identity comes from  Proposition \ref{fact} (3)), and so by Proposition \ref{dynamical density thm} and Proposition \ref{EDG-THM} one has
\begin{equation} \label{202410091902}
\overline{\dim}_{\mathcal{H}} (\mu_s) \ge \underline{\dim}_{\mathcal{H}} (\mu_s) = \dim_{\mathcal{H}}(K_s) \ge \overline{\dim}_{\mathcal{H}} (\mu_s).
\end{equation}

Set
$Z=\bigcup_{n\geq 1} K_{s_n}$, which may be viewed naturally as a Borel subset in $[0, 1]$,
where $s_n=1-2^{-n}$ for each $n\geq 1$. Now we check that the subset $Z$ has the required properties.

On one hand, from the construction of $Z$, it is easy to see  by Proposition \ref{fact-old} that
$$\dim_{\mathcal{H}}(Z) = \sup_{n\geq 1}\dim_{\mathcal{H}}(K_{s_n}) = 1,$$
then the set $Z$ has an increasing countable slice of Hausdorff dimension as $\dim_{\mathcal{H}}(K_{s_n})=s_n< 1$ (and so it has no measure of full lower Hausdorff dimension  by Corollary \ref{Hausdorff}).

On the other hand, let us set $\mu=\sum_{n\geq 1}2^{-n}\mu_{s_n}$, which is a Borel probability measure supported on $Z$. All of the measures $\mu$ and $\mu_{s_n}$ for each $n\in \mathbb{N}$ may be viewed naturally as Borel probability measures on $\mathbb{R}$.
For any $\mu$-measurable $E\subset \mathbb{R}$ satisfying $\mu(\mathbb{R} \setminus E)= 0$, we may take a Borel set  $E_*\subset E\cap [0, 1]$ satisfying $\mu(\mathbb{R} \setminus E_*)= 0$. For each $n\geq 1$, one has by the construction $\mu_{s_n} (\mathbb{R} \setminus E_*)= 0$ and then
 by the definition
 $$1\geq \dim_{\mathcal{H}}(E) \geq \dim_{\mathcal{H}}(E_*) \geq \overline{\dim}_{\mathcal{H}} (\mu_{s_n}) = \dim_{\mathcal{H}}(K_{s_n}) \ (\text{by \eqref{202410091902}})\ = 1 - 2^{-n}.$$
This implies $\dim_{\mathcal{H}}(E) = 1$ and then $\overline{\dim}_{\mathcal{H}} (\mu) = 1= \dim_{\mathcal{H}}(Z)$ by the arbitrariness of $E$.
That is, the measure $\mu$ is a measure of full upper Hausdorff dimension for the set $Z$.
\end{proof}

\section{Construction of Example \ref{submanifold counterexample}} \label{section-exam}

In this section let us present the detailed construction of Example \ref{submanifold counterexample}.

\smallskip

Note that we have proved $(1) \Rightarrow (2)$ for both Theorem \ref{MME Packing h-expan} and Theorem \ref{MME Bowen},
that is, for any analytic set $Z \subset X$ with positive packing (and Bowen, respectively) entropy, if the set $Z$ has measures of maximal packing  (and Bowen, respectively) entropy then it has no increasing countable slice of packing (and Bowen, respectively) entropy.
Thus the idea of constructing Example \ref{submanifold counterexample} is that, we construct a $C^{\infty}$ self-map over the unit square $X=[0,1]^2$, such that the smooth curve $K\doteq \{(r, r): 0\le r < 1\}\subset X$ has an increasing countable slice
$K = \bigcup_{i\in \mathbb{N}} K_i$ of packing (and Bowen, respectively) entropy, with $K_j = \{(r, r): 0\le r \le 1-\frac{1}{j}\}$ satisfying
\begin{equation} \label{estimation}
h_\mathrm{top}^B(T, K_j) \le h_\mathrm{top}^P(T, K_j)< \log 2 =
h_\mathrm{top}^B(T, K) = h_\mathrm{top}^P(T, K)\ \text{for each $j\in \mathbb{N}$},
\end{equation}
and then, by $(1) \Rightarrow (2)$ of both Theorem \ref{MME Packing h-expan} and Theorem \ref{MME Bowen}, the smooth
 curve $K$ has neither measures of maximal packing entropy nor measures of maximal Bowen entropy.

The detailed construction of Example \ref{submanifold counterexample} relies strongly on the dynamics of interval maps. More concretely, we will consider the one-parameter family of unimodal maps
$$g_a: [0, 1]\rightarrow [0, 1], \ \ y\mapsto a y (1-y);\ \ a \in[0,4].$$
We split the unit square $X=[0,1]^2$ into fibers $\{x\}\times [0,1], x\in [0, 1]$, where the self-map $T$ acts exactly like $g_{a (x)}$ over the fiber $\{x\}\times [0,1]$ for some desired parameter $a(x)\in [0, 4]$. And then the map $T$ will be a skew product
acting on the space $[0,1]^2$ given by
\begin{equation} \label{skew}
T: [0,1]^2\rightarrow [0,1]^2,\ (x,y)\mapsto (x,g_{a(x)}(y))=(x,a(x)\cdot y(1-y)).
\end{equation}
\begin{enumerate}

\item Clearly the self-map $T$ is $C^{\infty}$ if the parameter $a(x)$ is $C^\infty$ with respect to $x\in [0, 1]$.

\item If the parameter $a(x)$ is locally constant in a small interval $\left[x_1, x_2\right]$, then the action $T$ over any inclining lines (contained in $\left[x_1, x_2\right] \times[0,1]$) behaves like the action $g_a$ over a non-degenerate interval (contained in $[0,1]$).

\item Moreover, in order to estimate topological entropy of inclining lines (contained in $\left[x_1, x_2\right] \times[0,1]$) for some small interval $\left[x_1, x_2\right]$ and even obtain the estimation \eqref{estimation}, we shall construct appropriate parameters $a (x)$ such that these inclining lines have topological entropy same to that of the action $g_a$ over the whole space $[0,1]$, and so we are interested in those locally eventually onto\footnote{\ We skip the detailed introduction of the locally eventually onto interval maps, as we shall not use its definition in later discussions. Interested readers are referred to \cite[Chapter 3]{Brucks-Bruin2004}.} quadratic maps $g_a$.
\end{enumerate}

\begin{figure}[htbp]
	\centering
	\begin{minipage}[t]{0.48\textwidth}
		\centering
		\includegraphics[width=8cm]{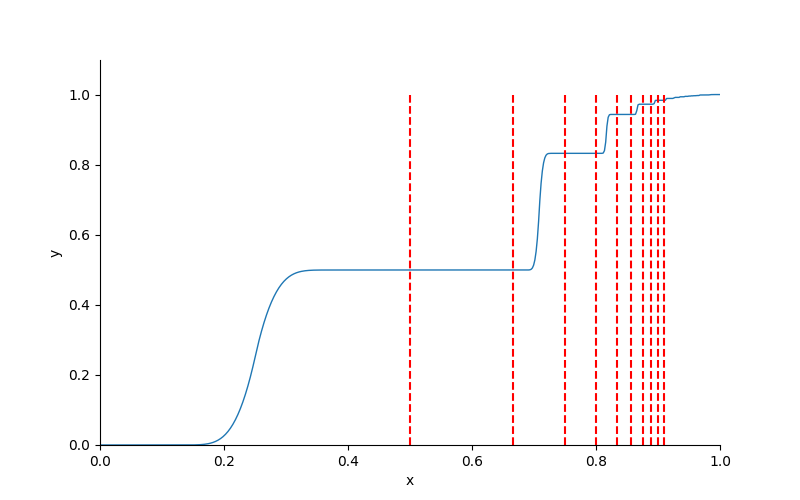}
		\caption{Function of $a(x)$}
	\end{minipage}
	\begin{minipage}[t]{0.48\textwidth}
		\centering
		\includegraphics[width=8cm]{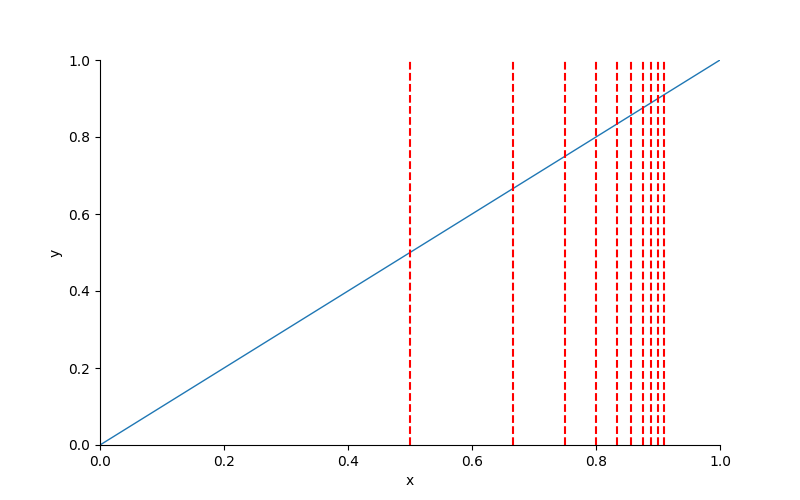}
		\caption{Diagonal set}
	\end{minipage}
\end{figure}

Denote by $\mathcal{FE}$ the set of all the parameters $a \in[0,4]$ such that $$
h_{\mathrm {top }}^B(g_a, U)=h_{\mathrm {top }}^P(g_a, U)=h_{\mathrm {top }}(g_a,[0,1])$$
for any non-degenerate interval $U\subset [0,1]$. It is well known that if the quadratic map $g_a$ is locally eventually onto then the corresponding parameter $a$ is contained in the set $\mathcal{FE}$, and that the point $a=4$ is an accumulation point of the set $\mathcal{FE}$.
For detailed introduction and further discussions the readers are encouraged to refer to \cite[Chapter 3]{Brucks-Bruin2004} and \cite{Melo-Strien1993}.

We shall use following well-known property about quadratic maps, that is, the monotonicity of topological entropy with respect to the parameters $a$ (see for example \cite[Chapter II, Section 10, Theorem 10.1 and Corollary 2]{Melo-Strien1993}, see also \cite{Brucks-Bruin2004, Bruin-Strien2015, Milnor-Thurston1988, Sullivan-Thurston1986}).

\begin{prop}\label{monotone_a}
	The function $[0,4]\rightarrow [0,\log2], a\mapsto h_{\mathrm {top }}(g_a,[0,1])$ is a non-decreasing continuous function, moreover, $h_{\mathrm {top }}(g_a,[0,1])=\log2$ if and only if $a=4$.
\end{prop}

We also use the following technical result, whose proof is standard in some sense. However, we will present a detailed proof for it in Appendix for completeness.

\begin{prop}\label{constuct_Local_const}
 Assume that $0<u_1<v_1<u_2<v_2<\cdots \nearrow 1$, and that for a given set $E\subset [0,1)$ its closure $\overline{E}$ contains the point $1$. Then there exists a non-decreasing $C^\infty$ function $\phi: [0,1]\rightarrow [0,1]$ such that
  \begin{enumerate}

\item \label{clc_1} $\phi(0)=0$, $\phi(1)=1$,

 \item \label{clc_2}
 $\phi$ is locally constant on each $[u_{n},v_{n}]$, i.e., $\phi(x)=\alpha_n$ on $[u_{n},v_{n}]$ (for some $\alpha_n$), and

 \item \label{clc_3} there exists a subsequence $n_i \rightarrow \infty$ such that each $\alpha_{n_i}$ is contained in $E$.
 \end{enumerate}
 \end{prop}

Now let us continue the construction of Example \ref{submanifold counterexample} in details as follows.

Firstly, applying above Proposition \ref{constuct_Local_const} we construct a series of real numbers $0<u_1<v_1<u_2<v_2<\cdots \nearrow 1$ and a non-decreasing $C^\infty$ function $a: [0,1]\rightarrow [0,4]$ such that
 $a(0)=0$, $a(1)=4$, and $a_n\in \mathcal{FE}\cap [0, 4)$ and
 $a(x)= a_n$ on $[u_{n},v_{n}]$ for each $n\in \mathbb{N}$. We remark that, these statements slightly differ from those in Proposition \ref{constuct_Local_const}, and can be obtained by choosing some subsequence if necessary.
Now the required map $T$ is a skew product acting on the space $[0,1]^2$ given by \eqref{skew}, which is clearly a $C^{\infty}$ map.

Let us prove \eqref{estimation} and so finish the construction of Example \ref{submanifold counterexample}, by showing:

\begin{claim}\label{ent_k_log2-}
$h_\mathrm{top}^B(T, K_j) \le h_\mathrm{top}^P(T, K_j)\le h_\mathrm{top}(T, K_j) < \log 2.$
\end{claim}

\begin{proof}
Note that each fiber $\{x\}\times [0,1]$ (with $x\in[0,1]$) is a $T$-invariant closed set (that is, $T (\{x\}\times [0,1])\subset \{x\}\times [0,1]$), every ergodic $\mu\in \mathcal{M} ([0,1]^2, T)$ is supported on one fiber $\{x\}\times [0,1]$ for some $x$. Thus, by the classical variational principle \eqref{VP} together with Proposition \ref{monotone_a}, entropy properties of quadratic maps, one has that for any $0\leq x_1\le x_2\leq 1$
 \begin{equation} \label{entropy1}
h_{\mathrm {top }}(T,[x_1,x_2]\times[0,1])=\sup_{x\in[x_1,x_2]}h_{\mathrm {top }}(T,\{x\}\times[0,1])=h_{\mathrm {top }}(g_{a (x_2)},[0,1]).
 \end{equation}
And then one has
\begin{eqnarray*}
h_\mathrm{top}^B(T, K_j) & \le & h_\mathrm{top}^P(T, K_j)\le h_\mathrm{top}(T, K_j)\ (\text{Proposition \ref{fact} (3)}) \\
& \le & h_{\mathrm {top }} \left(T,\left[0, 1- \frac{1}{j}\right]\times[0,1]\right) \\
 & = & h_{\mathrm {top }}(g_{a (1- \frac{1}{j})}, [0,1]) \ (\text{using \eqref{entropy1}})\ < \log 2\ (\text{Proposition \ref{monotone_a}}).
\end{eqnarray*}
This finishes the proof of Claim \ref{ent_k_log2-}.
\end{proof}

\begin{claim}\label{ent_k_log2}
$\log 2 =
h_\mathrm{top}^B(T, K) = h_\mathrm{top}^P(T, K) = h_\mathrm{top} (T, K).$
\end{claim}

\begin{proof}
Note that by a proof similar to that of Claim \ref{ent_k_log2-} one has
$$h_\mathrm{top}^B(T, K) \le h_\mathrm{top}^P(T, K) \le h_\mathrm{top} (T, K) \le h_\mathrm{top} (T, [0, 1]^2)= h_{\mathrm {top }}(g_{a (1)}, [0,1]) = \log 2.$$
It remains to prove
	$h_{\mathrm {top }}^B(T,K)\geq \log 2$. In fact, by above construction it suffices to prove
\begin{equation} \label{yao}
h_{\mathrm {top }}^B (T, K\cap [u_i,v_i]\times[0,1]) \geq h_{\mathrm {top }}(g_{a_i}, [0,1])\ \ \text{for each}\ i\in \mathbb{N}.
\end{equation}

Fix arbitrarily given $i\in \mathbb{N}$. Let us consider the following factor map
 $$\pi: ([u_i,v_i]\times [0,1], T)\to (\{u_i\}\times[0,1], T), \ \ (x,y)\mapsto (u_i, y).$$
As $\pi (K\cap [u_i,v_i]\times[0,1]) = \{u_i\}\times[u_i,v_i]$ is a non-degenerate interval, one has
\begin{align*}
	h_{\mathrm {top }}^B(T,K\cap [u_i,v_i]\times[0,1])&\geq h_{\mathrm {top }}^B(T,\{u_i\}\times [u_i,v_i]) \ (\text{Proposition \ref{Bowen_ent_factor_inequality}}) \\
	&=h_{\mathrm {top }}^B(g_{a_i},[u_i,v_i])\ (\text{by the definition})\ =h_{\mathrm {top }}(g_{a_i},[0,1]),
\end{align*}
where the last identity follows from the construction of $a_i$ that $a_i$ belongs to the set $\mathcal{FE}\cap [0, 4)$. This finishes the proof
of the inequality \eqref{yao} and then the proof of the claim.
\end{proof}

\section{Proof of Theorem \ref{MME Packing h-expan}}\label{Existence of measure of maximal Packing entropy}

In this section let us prove Theorem \ref{MME Packing h-expan}.
By \S \ref{easy}, it remains to prove $(2) \Rightarrow (1)$ of Theorem \ref{MME Packing h-expan} for an analytic set $Z\subset X$, with positive packing entropy, in an $h$-expansive system $(X,T)$. The proof is inspired by that of \cite[Theorem 1.3]{Feng-Huang2012}, which even goes back to work \cite{Joyce-Preiss1995} of Joyce and Preiss on the existence of subsets with finite packing measures.

\smallskip

From now on until the end of this section, we always assume that TDS $(X, T)$ is $h$-expansive.
As throughout the whole section we are only interested in packing entropy, for brevity, we will omit the term \emph{packing entropy} if there is no confusion. For example, we will just simply say that a set $Z$ has an increasing countable slice where we actually mean that the set $Z$ has an increasing countable slice of \emph{packing entropy}.

\smallskip

	Assume that the analytic set $Z$ has no increasing countable slice. We shall construct a compact set $K \subset Z$ and some $\mu\in\mathcal{M}(X)$ such that $$\mu(K)=1\ \ \ \text{and}\ \ \ \overline{h}_{\mu}(T) \ge h^P_\mathrm{top}(T, Z),$$
which implies by \eqref{varia.prin} that $\mu$ is a measure of maximal packing entropy for the set $Z$.
	
By the definition of an analytic set, there exists a continuous surjection $\phi: \mathcal{N}\rightarrow Z$, where recall that $\mathcal{N}$ denotes the set of all infinite sequences of natural numbers.
Given finitely many natural numbers $n_1, \cdots, n_p$, let $\Gamma_{n_1, \cdots,n_p}$ be the set of all $(m_1,m_2,\cdots)\in\mathcal{N}$ such that $m_1\leq n_1$, $\cdots$, $m_p\leq n_p$ and set $Z_{n_1, \cdots, n_p} = \phi (\Gamma_{n_1, \cdots,n_p})$.

Take arbitrarily $0< s_1<s_2<\cdots \nearrow h^P_\mathrm{top}(T, Z)$. We are going to construct inductively a sequence $\left\{K_{i}\right\}_{i\in \mathbb{N}}$ of finite subsets of $Z$ and a sequence $\left\{\mu_{i}\right\}_{i\in \mathbb{N}}$ of probability measures such that each $\mu_{i}$ is supported on $K_{i}$. The limit measure of any convergent sub-sequences of $\left\{\mu_{i}\right\}_{i\in \mathbb{N}}$ will be the required $\mu$, and we let $K$ be the support of $\mu$ which is obviously a compact set satisfying $\mu (K) = 1$. However, besides of these two sequences, we also construct a sequence $\left\{n_{i}\right\}_{i\in \mathbb{N}}$ in $\mathbb{N}$, a sequence $\left\{\gamma_{i}\right\}_{i\in \mathbb{N}}$ in $\mathbb{R}_{> 0}$ and a sequence $\left\{m_{i}: K_{i} \rightarrow \mathbb{N}\right\}_{i\in \mathbb{N}}$ of integer-valued functions, which will be used to conclude $K\subset Z$ and $\overline{h}_{\mu}(T) \ge h^P_\mathrm{top}(T, Z)$.
	
Now let us give detailed construction of $\mu, K$ (and other parameters) as follows.
Choose a countable base $\{U_i\}_{i\in \mathbb{N}}$ for the space $X$, and take by Lemma \ref{wangtao expan} the parameter $\varepsilon>0$.

\begin{proof}[Construction of $K_{1}, \mu_{1}, n_1, \gamma_1$ and the function $m_{1}: K_1\rightarrow \mathbb{N}$]
\
	
Applying Proposition \ref{prop slice} (3) to the analytic set $Z$ which has no increasing countable slice, there exists a nonempty analytic set $Z'\subset Z$ such that if an open set $U$ satisfies $Z'\cap U\neq \emptyset$ then $h_\mathrm{top}^P(T, Z'\cap U)  = h_\mathrm{top}^P (T, Z') = h_\mathrm{top}^P (T, Z)$ and $Z'\cap U$ has no increasing countable slice. Note $\mathcal{P}_{\varepsilon}^{s_1} (Z')=\infty$ and hence $P_{\varepsilon}^{s_1} (Z')=\infty$, as $0< s_1< h^P_\mathrm{top}(T, Z) = h^P_\mathrm{top}(T, Z') = h^P_\mathrm{top}(T, Z', \varepsilon)$ (by the construction of $\varepsilon$ from Lemma \ref{wangtao expan}). By Lemma \ref{F H Lem4.1}, there exists a finite subset $K_{1} \subset Z^{\prime}$ and a function $m_{1}: K_{1} \rightarrow \mathbb{N}$ such that the collection
$\left\{\overline{B}_{m_{1}(x)}(x, \varepsilon)\right\}_{x \in K_{1}}$ is disjoint and
	$$
1 < \xi_1 \doteq	\sum_{x \in K_{1}} e^{-m_{1}(x)\cdot s_1} < 1 + 2^{- 1}.
	$$
Set
$$\mu_{1}= \frac{1}{\xi_1} \sum_{x \in K_{1}} e^{-m_{1}(x) \cdot s_1} \delta_{x}.$$
Take small $\gamma_{1}>0$ such that if a map $z: K_{1} \rightarrow X$ satisfies $\max\limits_{x\in K_1} d(x, z(x)) \le \gamma_{1}$ then
	\begin{equation}\label{disjoint1}
		{\bigg (}\overline{B}\left(z(x), \gamma_{1}\right) \cup \overline{B}_{m_{1}(x)}(z(x), \varepsilon){\bigg )} \cap \bigcup_{y \in K_{1} \backslash\{x\}} {\bigg (}\overline{B}\left(z(y), \gamma_{1}\right) \cup
\overline{B}_{m_{1}(y)}(z(y), \varepsilon){\bigg )}=\emptyset
\end{equation}
 for each $x \in K_{1}$. Now it remains to construct $n_1$.

Firstly fix arbitrarily given $x_0 \in K_{1}\subset Z^\prime$. As $Z^{\prime}\cap B\left(x_0, \frac{\gamma_{1}}{4}\right) \supset \{x_0\} \neq \emptyset$, and so by the construction of the subset $Z^\prime$ in the previous paragraph, one has that
$$Z'\cap B\left(x_0, \frac{\gamma_{1}}{4}\right)\ \text{has no increasing countable slice and}\ h_\mathrm{top}^P {\bigg (}T, Z'\cap B\left(x_0, \frac{\gamma_{1}}{4}\right) {\bigg )} = h_\mathrm{top}^P (T, Z),$$
and then by Proposition \ref{prop slice} (1) so does $Z\cap B\left(x_0, \frac{\gamma_{1}}{4}\right)$. Note that $Z_1\subset Z_2\subset \cdots \nearrow Z$ by the construction at the beginning of the section, applying Proposition \ref{fact} (2) and Proposition \ref{prop slice} (1) and (2) one has that there exists large enough $n (x_0) \in \mathbb{N}$ so that, for all $n\ge n (x_0)$, $Z_n\cap B\left(x_0, \frac{\gamma_{1}}{4}\right)$ has no increasing countable slice and
$$h_\mathrm{top}^P {\bigg (}T, Z_n\cap B\left(x_0, \frac{\gamma_{1}}{4}\right) {\bigg )}  = h_\mathrm{top}^P (T, Z_n) = h_\mathrm{top}^P (T, Z).$$

Since $K_1\subset \left(Z^\prime \subset\right) Z$ is finite, by above arguments and also again by $Z_1\subset Z_2\subset \cdots \nearrow Z$, we can pick $n_{1} \in \mathbb{N}$ large enough so that $Z_{n_{1}} \supset K_{1}$, $$h_\mathrm{top}^P {\bigg (}T, Z_{n_{1}} \cap B\left(x, \frac{\gamma_{1}}{4}\right){\bigg )} = h^P_\mathrm{top}(T, Z_{n_1}) = h^P_\mathrm{top}(T, Z)$$
 for each $x\in K_1$, and that
$Z_{n_{1}} \cap B\left(x, \frac{\gamma_{1}}{4}\right)$ has no increasing countable slice.
\end{proof}

\begin{proof}[Construction of $K_{2}, \mu_{2}, n_2, \gamma_2$ and the function $m_{2}: K_2\rightarrow \mathbb{N}$]
\

Fix arbitrarily $x\in K_1$ and let us continue similarly to the previous step.
Firstly applying Proposition \ref{prop slice} (3) to the analytic set $Z_{n_{1}} \cap B\left(x, \frac{\gamma_{1}}{4}\right)$ which has no increasing countable slice, there exists a nonempty analytic set
$Z_{n_1}^\prime (x)\subset Z_{n_{1}} \cap B \left(x, \frac{\gamma_1}{4}\right)$
such that if an open set $U$ satisfies $Z_{n_1}' (x) \cap U\neq \emptyset$ then $h_\mathrm{top}^P(T, Z_{n_1}' (x)\cap U) = h_\mathrm{top}^P(T, Z_{n_1}' (x)) = h_\mathrm{top}^P (T, Z)$ and $Z_{n_1}' (x)\cap U$ has no increasing countable slice. Again by the construction of $\varepsilon$ one has
$$0< s_2< h^P_\mathrm{top}(T, Z) = h^P_\mathrm{top}(T, Z_{n_1}' (x)) = h^P_\mathrm{top}(T, Z_{n_1}' (x), \varepsilon),$$
then $\mathcal{P}_{\varepsilon}^{s_2} (Z_{n_1}' (x))=\infty$ and hence $P_{\varepsilon}^{s_2} (Z_{n_1}' (x))=\infty$,
thus by Lemma \ref{F H Lem4.1} we may find finite $E_{2}(x) \subset Z_{n_{1}}' (x)$ and a function $m_{2}: E_{2}(x) \rightarrow \mathbb{N} \cap\left[\max\limits_{y\in K_1} \ m_{1}(y), \infty\right)$ such that
\begin{equation} \label{w}
\mu_{1}(\{x\})<\sum_{y \in E_{2}(x)} e^{-m_{2}(y) \cdot s_2}<\left(1+2^{-2}\right) \mu_{1}(\{x\}),
\end{equation}
and that the collection $\left\{\overline{B}_{m_{2}(y)}(y, \varepsilon)\right\}_{y \in E_{2}(x)}$ is disjoint.
As $E_{2}(x)$ is a subset of $Z_{n_1}^\prime(x)$, one has that, by the selection of $Z_{n_1}^\prime(x)$, if
an open set $U$ satisfies $E_2 (x) \cap U\neq \emptyset$, then $Z_{n_1}' (x) \cap U\neq \emptyset$, and so $h_\mathrm{top}^P(T, Z_{n_1}\cap U) = h_\mathrm{top}^P (T, Z)$ and $Z_{n_1}\cap U$ has no increasing countable slice.

Now by applying \eqref{disjoint1} to the identity map $z: K_1\rightarrow X$, the collection $\left\{\overline{B}\left(x, \gamma_{1}\right)\right\}_{x \in K_{1}}$ is disjoint, and so the collection $\left\{E_{2}(x)\right\}_{x \in K_{1}}$ is disjoint. Particularly, if we set $K_{2}=\bigcup_{x \in K_{1}} E_{2}(x)$ then the map $m_2: K_2\rightarrow \mathbb{N} \cap\left[\max\limits_{y\in K_1} \ m_{1}(y), \infty\right)$ is well defined. Set
$$
	\mu_{2} = \frac{1}{\xi_2} \sum_{y \in K_{2}} e^{-m_{2}(y)\cdot s_2} \delta_{y}\ \ \ \ \text{where}\ \xi_2 = \sum_{y \in K_{2}} e^{-m_{2}(y)\cdot s_2}\in \left(1, 1 + 2^{-2}\right)\ \ (\text{by \eqref{w}}).
	$$

We remark that, by the above construction, the collection $\left\{\overline{B}_{m_{2}(y)} (y, \varepsilon)\right\}_{y \in K_{2}}$ is disjoint. In fact, let us firstly take arbitrarily different points $y_1, y_2\in K_2$ (if possible). If both $y_1$ and $y_2$ come from the same $E_2 (x_0)$ for some $x_0\in K_1$, then the disjointness of $\overline{B}_{m_{2}(y_1)} (y_1, \varepsilon)$ and $\overline{B}_{m_{2}(y_2)} (y_2, \varepsilon)$ comes directly from the selection of $E_2 (x_0)$. Now assume that $y_1\in E_2 (x_1)$ and $y_2\in E_2 (x_2)$ for different points $x_1$ and $x_2$ in $K_1$. In particular:
\begin{itemize}

\item[(a)] By the construction of $E_2 (x_i)$ one has $y_i\in B \left(x_i, \frac{\gamma_1}{4}\right)$ for each $i = 1, 2$.

\item[(b)] By the construction of the formula \eqref{disjoint1} and the function $m_2$, one has that if $z: K_{1} \rightarrow X$ is a map satisfying $\max\limits_{x\in K_1} d(x, z(x)) \le \gamma_{1}$, then it holds
\begin{equation} \label{what}
\emptyset= \overline{B}_{m_{1}(x_1)}(z(x_1), \varepsilon) \cap \overline{B}_{m_{1}(x_2)}(z(x_2), \varepsilon) \supset \overline{B}_{m_{2}(y_1)}(z(x_1), \varepsilon) \cap \overline{B}_{m_{2}(y_2)}(z(x_2), \varepsilon).
\end{equation}
\end{itemize}
Then the disjointness of $\overline{B}_{m_{2}(y_1)} (y_1, \varepsilon)$ and $\overline{B}_{m_{2}(y_2)} (y_2, \varepsilon)$ comes directly by considering the map $z: K_1\rightarrow X$ where $z (x_1) = y_1$, $z (x_2) = y_2$ and $z (x) = x$ for each $x\in K_1\setminus \{x_1, x_2\}$.

By the disjointness of the collection $\left\{\overline{B}_{m_{2}(y)} (y, \varepsilon)\right\}_{y \in K_{2}}$, we can choose $0< \gamma_{2}< \frac{\gamma_{1}}{4}$ small enough such that if a function $z: K_{2} \rightarrow X$ satisfies $\max\limits_{y\in K_2} d(y, z(y))\le \gamma_{2}$ then
\begin{equation} \label{}
	{\bigg (}\overline{B}\left(z(y), \gamma_{2}\right) \cup \overline{B}_{m_{2}(y)}\left(z(y), \varepsilon\right){\bigg )} \cap \bigcup_{y' \in K_{2} \backslash \{y\}} {\bigg (}\overline{B}\left(z(y'), \gamma_{2}\right) \cup \overline{B}_{m_{2}(y')}\left(z(y'), \varepsilon\right){\bigg )}=\emptyset
\end{equation}
	for each $y \in K_{2}$. Now it remains to define $n_2$ similarly to the construction of $n_1$.

Fix arbitrarily given $y_0\in K_2$. Write $y_0\in E_2 (x_0)$ for some $x_0\in K_1$. Note that $y_0$ belongs to $Z_{n_1}\cap B\left(y_0, \frac{\gamma_{2}}{4}\right)$.
By arguments below \eqref{w} one has that $Z_{n_1}\cap B\left(y_0, \frac{\gamma_{2}}{4}\right)$ has no increasing countable slice and $h_\mathrm{top}^P {\bigg (}T, Z_{n_1}\cap B\left(y_0, \frac{\gamma_{2}}{4}\right) {\bigg )} = h_\mathrm{top}^P (T, Z)$.
Note that $Z_{n_1, 1}\subset Z_{n_1, 2}\subset \cdots \nearrow Z_{n_1}$, as we did in previous step, there exists large enough $n (y_0) \in \mathbb{N}$ so that $Z_{n_1, n}\cap B\left(y_0, \frac{\gamma_{2}}{4}\right)$ has no increasing countable slice and $h_\mathrm{top}^P {\bigg (}T, Z_{n_1, n}\cap B\left(y_0, \frac{\gamma_{2}}{4}\right) {\bigg )} = h_\mathrm{top}^P (T, Z)$ for all $n\ge n (y_0)$.
Finally, we pick $n_{2} \in \mathbb{N}$ large enough so that $Z_{n_{1}, n_2} \supset K_{2}$, $Z_{n_{1}, n_2} \cap B\left(y, \frac{\gamma_{2}}{4}\right)$ has no increasing countable slice and $h_\mathrm{top}^P {\bigg (}T, Z_{n_{1}, n_2} \cap B\left(y, \frac{\gamma_{2}}{4}\right){\bigg )} = h^P_\mathrm{top}(T, Z), \forall y\in K_2$.
\end{proof}

\begin{proof}[Inductive construction of $K_{p+1}, \mu_{p+1}, n_{p+1}, \gamma_{p+1}$ and $m_{p+1}: K_{p+1}\rightarrow \mathbb{N}$ for integers $p \ge 2$]\

Assume that we have constructed $K_{q}, \mu_{q}, n_{q}, \gamma_{q}$ and the functions $m_{q}: K_{q}\rightarrow \mathbb{N}$ for each $q= 1, \cdots, p$, such that, if a function $z: K_{p} \rightarrow X$ satisfies $\max\limits_{x\in K_p} d(x, z(x))\le \gamma_{p}$ then
	\begin{equation}\label{disjointp}
		{\bigg (}\overline{B}\left(z(x), \gamma_{p}\right) \cup \overline{B}_{m_{p}(x)}(z(x), \varepsilon){\bigg )} \cap\bigcup_{y \in K_{p} \backslash\{x\}} {\bigg (}\overline{B}\left(z(y), \gamma_{p}\right) \cup \overline{B}_{m_{p}(y)}(z(y), \varepsilon){\bigg )}=\emptyset	
	\end{equation}
	for each $x \in K_{p}$; and that $Z\supset Z_{n_{1}} \supset Z_{n_{1}, n_{2}} \supset \cdots \supset Z_{n_{1}, \cdots, n_{p}} \supset K_{p}$ and $\mu_p (\{x\})> 0$ for each $x\in K_p$; furthermore, $Z_{n_{1}, \cdots, n_{p}} \cap B\left(x, \frac{\gamma_{p}}{4}\right)$ has no increasing countable slice and $h_\mathrm{top}^P {\bigg (}T, Z_{n_{1}, \cdots, n_{p}} \cap B\left(x, \frac{\gamma_{p}}{4}\right){\bigg )} = h^P_\mathrm{top}(T, Z)$ for each $x \in K_{p}$.
 Fix any given $x \in K_{p}$. As
 $$0 < s_{p + 1} < h^P_\mathrm{top}(T, Z) = h_\mathrm{top}^P {\bigg (}T, Z_{n_{1}, \cdots, n_{p}} \cap B\left(x, \frac{\gamma_{p}}{4}\right){\bigg )},$$
we could construct as in the previous step that a finite set
	$E_{p+1}(x) \subset Z_{n_{1}, \cdots, n_{p}} \cap B \left(x, \frac{\gamma_{p}}{4}\right)$ and a function
	$m_{p+1}: E_{p+1}(x) \rightarrow \mathbb{N} \cap \left[ \max\limits_{y\in K_p}\ m_{p}(y), \infty\right)$ such that
$$\mu_{p}(\{x\})<\sum_{y \in E_{p+1}(x)} e^{-m_{p+1}(y) \cdot s_{p+1}}<\left(1+2^{-p-1}\right) \mu_{p}(\{x\}),$$
and that the collection $\left\{\overline{B}_{m_{p+1}(y)} (y, \varepsilon)\right\}_{y \in E_{p+1}(x)}$ is disjoint and
if an open set $U$ satisfies $E_{p+1}(x) \cap U \neq \emptyset$ then $h_\mathrm{top}^P(T, Z_{n_{1}, \cdots, n_{p}} \cap U) = h^P_\mathrm{top}(T, Z)$ and $Z_{n_{1}, \cdots, n_{p}} \cap U$ has no increasing countable slice.
Again the collection $\left\{E_{p+1}(x)\right\}_{x\in K_p}$ is disjoint, then for $K_{p+1}= \bigcup_{x \in K_{p}} E_{p+1}(x)$ the map $m_{p+1}: K_{p+1}\rightarrow \mathbb{N} \cap \left[ \max\limits_{y\in K_p}\ m_{p}(y), \infty\right)$ is well defined. Set
$$\mu_{p+1}=\frac{1}{\xi_{p+1}} \sum_{y \in K_{p+1}} e^{-m_{p+1}(y)\cdot s_{p+1}} \delta_{y}$$
where
$$\xi_{p+1}= \sum_{y \in K_{p+1}} e^{-m_{p+1}(y)\cdot s_{p+1}}\in \left(1, 1+ 2^{- p -1}\right).$$

Furthermore, as did in previous step, by construction of the formula \eqref{disjointp} and the function $m_p$: the collection $\left\{ \overline{B}_{m_{p+1}(y)}(y, \varepsilon)\right\}_{y \in K_{p+1}}$ is disjoint, and hence we can choose small $0<\gamma_{p+1}< \frac{\gamma_{p}}{4}$ such that if a function $z: K_{p+1} \rightarrow X$ satisfies $\max\limits_{x\in K_{p+ 1}} d(x, z(x))\le \gamma_{p+1}$ then
$${\bigg (}\overline{B}\left(z(x), \gamma_{p+1}\right) \cup \overline{B}_{m_{p+1}(x)}(z(x), \varepsilon){\bigg )} \cap\bigcup_{y \in K_{p+1} \setminus\{x\}} {\bigg (}\overline{B}\left(z(y), \gamma_{p+1}\right) \cup \overline{B}_{m_{p+1}(y)}(z(y), \varepsilon){\bigg )}=\emptyset
$$
for each $x \in K_{p+1}$; additionally,
as $Z_{n_1, \cdots, n_p, 1}\subset Z_{n_1, \cdots, n_p, 2}\subset \cdots \nearrow Z_{n_1\cdots, n_p}$,
we can pick $n_{p+1} \in \mathbb{N}$ large enough such that $Z_{n_{1}, \cdots, n_{p+1}} \supset K_{p+1}$, $Z_{n_{1}, \cdots, n_{p+1}} \cap B\left(x, \frac{\gamma_{p+1}}{4}\right)$ has no increasing countable slice and $h_\mathrm{top}^P {\bigg (}T, Z_{n_{1}, \cdots, n_{p+1}} \cap B\left(x, \frac{\gamma_{p+1}}{4}\right) {\bigg )} = h^P_\mathrm{top}(T, Z)$ for each $x \in K_{p+1}$.
\end{proof}

In the following let us construct $\mu\in\mathcal{M}(X)$ and a compact set $K \subset Z$ such that $\mu(K)=1$ and $\overline{h}_{\mu}(T) \ge h^P_\mathrm{top}(T, Z)$ (hence $\mu$ is a measure of maximal packing entropy for the set $Z$).

In fact we let $\mu\in \mathcal{M} (X)$ be any limit point of any convergent sub-sequence of $\left\{\mu_{i}\right\}_{i\in \mathbb{N}}$, and let $K$ be the support of $\mu$ which is obviously a compact set satisfying $\mu (K) = 1$. Clearly the compact set $K$ is contained in the subset
	$$
	\bigcap_{p=1}^{\infty} \overline{\bigcup_{i\ge p} K_i}
 \subset \bigcap_{p=1}^{\infty} \overline{Z_{n_{1}, \cdots, n_{p}}} = \bigcap_{p=1}^{\infty} Z_{n_{1}, \cdots, n_{p}}
	$$
which is a subset of $Z$,
where the last identity follows by applying Cantor's diagonal argument along with the continuity of $\phi$. In the following let us check $\overline{h}_{\mu}(T) \ge h^P_\mathrm{top}(T, Z)$.

From above inductive construction we summarize some basic properties for finite $K_i\subset Z$, $\mu_i\in \mathcal{M}(Z)$, $n_i\in \mathbb{N}$, $\gamma_i> 0$, $\xi_i\in \left(1, 1+ 2^{- i}\right)$ and $m_i: K_i\rightarrow \mathbb{N}$ for each $i\in \mathbb{N}$ as follows:
	
	\begin{enumerate}

	\item[(c)] For each $i$, the collection $\mathcal{F}_{i}:= \left\{\overline{B}\left(x, \gamma_{i}\right): x \in K_{i}\right\}$ is disjoint, and each element of $\mathcal{F}_{i+1}$ is a subset of $\overline{B}\left(x, \frac{\gamma_{i}}{2}\right)$ for some $x \in K_{i}$.
	
	\item[(d)] For each $x \in K_{i}$
 the subset $E_{i+1}(x)$ is in fact exactly $B\left(x, \gamma_{i}\right) \cap K_{i+1}$ such that
 $$\overline{B}_{m_{i}(x)}(z, \varepsilon) \cap \bigcup_{y \in K_{i} \backslash\{x\}} \overline{B}\left(y, \gamma_{i}\right)=\emptyset\ \ \text { for any $z \in \overline{B}\left(x, \gamma_{i}\right)$},$$
 $$\xi_i^{- 1}\cdot e^{-m_{i}(x)\cdot s_i} = \mu_{i}\left(\overline{B}\left(x, \gamma_{i}\right)\right) \le \sum_{y \in E_{i+1}(x)} e^{-m_{i+1}(y)\cdot s_{i+1}} \le\left(1+2^{-i-1}\right) \mu_{i}\left(\overline{B}\left(x, \gamma_{i}\right)\right).
	$$
	\end{enumerate}
This implies easily that for all integers $i< j$ and each elements $F_i\in \mathcal{F}_i$ one has
$$\mu_{i+1}\left(F_{i}\right) \le \sum_{F \in \mathcal{F}_{i+1}, F \subset F_{i}} \xi_{i+ 1}\cdot \mu_{i+1}(F) \le\left(1+2^{-i-1}\right) \mu_{i}\left(F_{i}\right)$$
and then
	\begin{equation*}
		\mu_{j}\left(F_{i}\right) \le \prod_{n=i+1}^{j}\left(1+2^{-n}\right) \mu_{i}\left(F_{i}\right) \le C\cdot \mu_{i}\left(F_{i}\right)\ \ \ \ \text{where $C=\prod_{n=1}^{\infty}\left(1+2^{-n}\right)<\infty$},
	\end{equation*}
thus for each $x\in K_i$, by the construction of $\mu$ and $\xi_i\in \left(1, 1+ 2^{- i}\right)$,
	\begin{equation} \label{contmeas}
\mu {\big (}B\left(x, \gamma_{i}\right) {\big )} \le \liminf_{n\rightarrow \infty} \mu_n \left(B\left(x, \gamma_{i}\right)\right) \le C\cdot \mu_{i} \left(\overline{B}\left(x, \gamma_{i}\right)\right) \le C\cdot e^{-m_{i}(x)\cdot s_i}.
\end{equation}
Note $K \subset \bigcup_{x \in K_{i}} \overline{B} \left(x, \frac{\gamma_{i}}{2}\right)$. By (d) and \eqref{contmeas}, for each $x \in K_{i}$ and $z \in \overline{B}\left(x, \frac{\gamma_{i}}{2}\right)$ one has
	$$
\mu \left(\overline{B}_{m_{i}(x)}(z, \varepsilon)\right) \le \mu \left(\overline{B}\left(x, \frac{\gamma_{i}}{2}\right)\right) \le \mu {\big (}B\left(x, \gamma_{i}\right) {\big )} \le C\cdot e^{-m_{i}(x)\cdot s_i}.
	$$

Now fix arbitrarily given $z \in K$. For each $i \in \mathbb{N}$, the point $z$ should be contained in $\overline{B} \left(x, \frac{\gamma_{i}}{2}\right)$ for some $x\in K_i$, then
	$
\mu \left(B_{m_{i}(x)}(z, \varepsilon)\right) \le C\cdot e^{-m_{i}(x)\cdot s_i},
	$ and so
$\overline{h}_{\mu}(T, z) \ge \overline{h}_{\mu}(T, z, \varepsilon)\ge s_i$.
Thus $\overline{h}_{\mu}(T)\ge s_i$, from which it is trivial to obtain $\overline{h}_{\mu}(T) \ge h^P_\mathrm{top}(T, Z)$.

\section{Proof of Theorem \ref{MME Bowen} for compact subsets} \label{compact}

We have proved in \S \ref{easy} the directions $(1) \Rightarrow (2)$ and $(3) \Rightarrow (2)$ of Theorem \ref{MME Bowen}. Now, given
 compact subsets with positive Bowen entropy,
 let us prove the direction $(3) \Rightarrow (1)$ of Theorem \ref{MME Bowen}, and the direction $(2) \Longrightarrow (3)$ of Theorem \ref{MME Bowen} for $h$-expansive systems.

\smallskip

For convenience, we restate it as follows, which is main result of this section.

\begin{theorem}\label{MME Bowen-s6}
Let $Z\subset X$ be a compact subset with $h_\mathrm{top}^B(T, Z)>0$. Then $(3) \Rightarrow (1)$, furthermore, $(2) \Rightarrow (3)$ (and then $(2) \Rightarrow (1)$) once the system $(X,T)$ is $h$-expansive:
	\begin{enumerate}

		\item[(1)] The subset $Z$ has measures of maximal Bowen entropy.

		\item[(2)] The subset $Z$ has no increasing countable slice of Bowen entropy.

		\item[(3)] There exists some gauge function $b$ such that $\mathcal{M}^b(Z)>0$ and $\lim\limits_{n\to \infty}e^{ns}b(n)=0$ for any $s < h_\mathrm{top}^B (T, Z)$.
	\end{enumerate}
\end{theorem}

As we did in the previous section, throughout the whole section we are only interested in Bowen entropy, and then we will omit the term \emph{Bowen entropy} if there is no confusion.

\subsection{Proof of $(3) \Longrightarrow (1)$ of Theorem \ref{MME Bowen-s6}} \label{sub-1}\

\smallskip

In this subsection we prove $(3) \Longrightarrow (1)$ of Theorem \ref{MME Bowen-s6} (that is, $(3) \Longrightarrow (1)$ of Theorem \ref{MME Bowen} for compact $Z\subset X$). Firstly let us give an equivalent formulation of the quantity
$\M^{b}(Z)$ introduced in \S\S \ref{subs-Bowen} for each $Z\subset X$ and any gauge function $b: \N\to \mathbb{R}_+$.

\subsubsection{Equivalent formulation of $\M^{b}(Z)$}\

\smallskip

Let $f: X \rightarrow \mathbb{R}$ be a bounded function. For any gauge function $b$, $N \in \mathbb{N}$ and $\varepsilon>0$, put
\begin{equation}\label{def of WB ent func}
	\mathcal{W}_{N, \varepsilon}^b(f)=\inf \sum_{i\in \mathcal{I}} c_i b(n_i),
\end{equation}
where the infimum is taken over all  countable families $\left\{\left(B_{n_i}\left(x_i, \varepsilon\right), c_i\right)\right\}_{i\in \mathcal{I}}$ such that $0<$ $c_i<\infty, x_i \in X, n_i \ge N$ and
$
\sum_{i\in \mathcal{I}} c_i \chi_{B_{n_i}\left(x_i, \varepsilon\right)} \ge f
$ in which $\chi_A$ denotes the characteristic function of $A$ (i.e., $\chi_A(x)=1$ if $x \in A$ and 0 if $x \in X \backslash A$).

For any $Z \subset X$ we set $\mathcal{W}_{N, \varepsilon}^b(Z)=\mathcal{W}_{N, \varepsilon}^b\left(\chi_Z\right)$. Clearly the quantity $\mathcal{W}_{N, \varepsilon}^s(Z)$ increases as $N$ increases and $\varepsilon$ decreases, hence the following limits exist
$$
\mathcal{W}_\varepsilon^b(Z)=\lim _{N \rightarrow \infty} \mathcal{W}_{N, \varepsilon}^b(Z), \quad \mathcal{W}^b(Z)=\lim _{\varepsilon \rightarrow 0} \mathcal{W}_\varepsilon^b(Z) .
$$

The main result of this subsection is the following proposition.

\begin{prop}\label{equi of weighted and ordinary entropy function}
	Let $K \subset X$ be a compact subset and $b: \N\to \mathbb{R}_+$ be a gauge function. Then $\mathcal{M}^{b}(K)=\mathcal{W}^b(K)$. In fact, for any $N\in \N$ and $\varepsilon>0$, one has
	$$
	\mathcal{M}_{N, 3 \varepsilon}^{b}(K) \le \mathcal{W}_{N, \varepsilon}^b(K) \le \mathcal{M}_{N, \varepsilon}^b(K).
	$$
\end{prop}

To prove Proposition \ref{equi of weighted and ordinary entropy function}, we need the following lemma, which resembles the proof of the classical Vitali covering lemma (for details see for example \cite[Theorem 2.1]{Mattila1995}).

\begin{lemma}\label{Modified Vitali covering lemma}
	Let $\varepsilon>0$, and let $\mathcal{B}=$ $\left\{B_{n_i}\left(x_i, \varepsilon\right)\right\}_{i \in \mathcal{I}}$ be a family of open Bowen balls in the system $(X, T)$. Then there exists in $\mathcal{B}$ a countable subfamily $\mathcal{B}^{\prime}=\left\{B_{n_i}\left(x_i, \varepsilon\right)\right\}_{i \in \mathcal{I}^{\prime}}$ of pairwise disjoint balls such that
	$$
	\bigcup_{B \in \mathcal{B}} B \subset \bigcup_{i \in \mathcal{I}^{\prime}} B_{n_i}\left(x_i, 3 \varepsilon\right)
	$$
\end{lemma}
\begin{proof}
For simplicity, for each $B=B_{n_i}\left(x_i, \varepsilon\right)\in \mathcal{B}$, we denote $c(B)\doteq  x_i$ as the center of the ball $B$, and set $n(B)= n_i$. Now let us construct the required subfamily $\mathcal{B}^{\prime}$ by induction.
	
Denote $\mathcal{B}_0=\mathcal{B}$. Obviously we could choose some $B_1\in \mathcal{B}$ such that $
	n(B_1)=\min\limits_{B\in \mathcal{B}_0} n(B).$
	Put $\mathcal{B}_1= \left\lbrace B\in \mathcal{B}_0:B\cap B_1=\emptyset \right\rbrace$. 	
	If $\mathcal{B}_1=\emptyset$, we stop the induction process and set $\mathcal{B}^{\prime}=\{B_1\}$. Otherwise, we could choose $B_2\in \mathcal{B}_1$ such that $
	n(B_2)=\min\limits_{B\in \mathcal{B}_1} n(B).$

	Now assume that we have chosen $\mathcal{B}_0, \mathcal{B}_1, \cdots, \mathcal{B}_{p - 1}$ and $B_1, B_2, \cdots, B_p$, $p\in \mathbb{N}$, satisfying $$
\mathcal{B}_i = \left\lbrace B\in \mathcal{B}_{i-1}:B\cap B_i=\emptyset     \right\rbrace,\ \ \ \  \forall 1\le i\le p - 1,$$
$$n(B_i)=\min\limits_{B\in \mathcal{B}_{i-1}} n(B),\ \ \ \  \forall 1\le i\le p.$$

Put $
	\mathcal{B}_p = \left\lbrace B\in \mathcal{B}_{p-1}:B\cap B_p=\emptyset     \right\rbrace. $
	If $\mathcal{B}_p=\emptyset$, we stop the induction process and set $\mathcal{B}^{\prime}= \{B_1,\cdots, B_p\}$. Otherwise, we could choose $B_{p+1}\in \mathcal{B}_p$ such that $
	n(B_{p+1})=\min\limits_{B\in \mathcal{B}_p} n(B).$
If the above procedure never stops by this induction process, we shall get a countable subfamily from $\mathcal{B}$, denoted by $\mathcal{B}^\prime$, consisting of countably many disjoint open Bowen balls.

As the proof is completely the same, we only manage the case that $\mathcal{B}^\prime$ is an infinite family, that is, $\mathcal{B}^\prime = \{B_i: i\in \mathbb{N}\}$. Fix arbitrarily given $B\in \mathcal{B}$. We prove firstly that $B$ must intersect $B_i$ for some $i\in \mathbb{N}$.
Write $B_i = B_{n_i}\left(x_i, \varepsilon\right)$ for each $i\in \mathbb{N}$. Note that by the construction $n_1\le n_2\le \cdots$. In fact, $n_i$ tends to $\infty$ when $i$ goes to $\infty$, as $\mathcal{B}^\prime$ is an infinite pairwise disjoint family (by the above construction) in the compact set $X$. However, for each $i\in \mathbb{N}$ one has $n (B)\ge n_{i+ 1}$ once $B$ has empty intersection with all of $B_1, \cdots, B_i$, which implies that $B$ must intersect $B_i$ for some $i\in \mathbb{N}$ (otherwise, we conclude $n (B) = \infty$, impossible). Now let $i^*\in \mathbb{N}$ satisfy $B\cap B_{i^* + 1}\neq \emptyset$ and $B\cap B_j= \emptyset$ for each $1\leq j\le i^*$ (such $i^*\in \mathbb{N}$ exists uniquely), then by above construction $n(B)\geq n_{i^* + 1}$ and $c (B)\in B_{n_{i^* + 1}} (x_{i^* + 1}, 2 \varepsilon)$ (we remark that $B\cap B_{i^* + 1}\neq \emptyset$ and $n(B)\geq n_{i^* + 1}$), and so
$$
	B_{n_{i^* + 1}} (x_{i^* + 1}, 3 \varepsilon)\supset B_{n_{i^* + 1}} (c(B), \varepsilon)\supset B.$$
Finally, the conclusion follows from the arbitrariness of $B\in \mathcal{B}$.
\end{proof}

Now we are ready to prove Proposition \ref{equi of weighted and ordinary entropy function}, which originates from \cite[Proposition 3.2]{Feng-Huang2012}.

\begin{proof}[Proof of Proposition \ref{equi of weighted and ordinary entropy function}]
The direction of $\mathcal{W}_{N, \varepsilon}^b(K) \le \mathcal{M}_{N, \varepsilon}^b(K)$ follows directly from the definition (\ref{def of WB ent func}). Now let us prove the direction of $\mathcal{M}_{N, 3 \varepsilon}^{b}(K) \le \mathcal{W}_{N, \varepsilon}^b(K)$.

It suffices to prove that, for any given countable family $\left\{\left(B_{n_i}\left(x_i, \varepsilon\right), c_i\right)\right\}_{i\in \mathcal{I}}$ satisfying
$\sum_{i\in \mathcal{I}} c_i \chi_{B_{n_i}\left(x_i, \varepsilon\right)} \ge \chi_K,
$
with $0< c_i <\infty, x_i \in X, n_i \ge N$ for each $i\in \mathcal{I}$, one has
\begin{equation} \label{need-7}
\mathcal{M}_{N, 3 \varepsilon}^{b}(K) \le \sum_{i \in \mathcal{I}} c_i b(n_i).
\end{equation}

Fix arbitrarily given $0<t<1$. Observe that, for each nonempty set $\mathcal{J}\subset \mathcal{I}$, the set
$$
B_\mathcal{J}\doteq \left\{x\in X: \sum_{i\in \mathcal{J}} c_i \chi_{B_{n_i}\left(x_i, \varepsilon\right)} (x) > t\right\}
$$
is open, which covers $K$ when $\mathcal{J} = \mathcal{I}$ (by the assumption). Note that $\mathcal{I}$ is a countable index set, by the compactness of $K$, there exists finite nonempty $\mathcal{K}\subset \mathcal{I}$ satisfying $B_\mathcal{K}\supset K$.
We may assume, by approximating $c_i$'s from above, that each $c_i, i\in \mathcal{K}$ is a positive rational; furthermore, by multiplying with a common denominator on both sides, we have $c_i\in \mathbb{N}$ for each $i\in \mathcal{K}$ and then take $m$ to be the least integer with $m \geq t$.
  For simplicity, we write $\mathcal{K} = \{1, \cdots, k\}$ with $k\in \mathbb{N}$, and $B_i = B_{n_i}\left(x_i, \varepsilon\right)$ for each $i\in \mathcal{K}$. Obviously we may assume $B_1, \cdots, B_k$ are different Bowen balls. Thus, with these newly chosen parameters one has
$$
\left\{x\in X: \sum_{i = 1}^k c_i \chi_{B_i} (x) \ge m\right\}\supset K.
$$

Write $\mathcal{B}= \left\{B_1, \cdots, B_k\right\}$ and $3 B_i = B_{n_i}\left(x_i, 3 \varepsilon\right)$ for $B_i = B_{n_i}\left(x_i, \varepsilon\right)$.
 Slightly abusing the gauge function $b$, we may assume that $b$ is also defined on $\mathcal{B}$ via $b: B_i\mapsto n_i$.

 Define $u: \mathcal{B} \rightarrow \mathbb{N}, B_i\mapsto c_i$. In the following, let us start with $v_0=u$.
As the set $K$ is covered by the family $\mathcal{B}$ by above construction, we choose by Lemma \ref{Modified Vitali covering lemma} a disjoint subfamily $\mathcal{B}_1\subset \mathcal{B}$ with $K \subset \bigcup_{B \in \mathcal{B}_1} 3 B$.
By this procedure, we may define inductively, for each $j= 2, \cdots, m$, a function $v_{j- 1}: \mathcal{B} \rightarrow \mathbb{N}$ and a disjoint subfamily $\mathcal{B}_j\subset \mathcal{B}$, such that
\begin{eqnarray}
	& & v_{j-1}(B)= \begin{cases}
v_{j-2}(B)-1\ \ \ \ & \text {    if } B \in \mathcal{B}_{j-1} \\
v_{j-2}(B)\ \ \ \ & \text {    if } B \in \mathcal{B} \backslash \mathcal{B}_{j-1}
\end{cases}, \label{constr}\\
	& & \mathcal{B}_j \subset\left\{B \in \mathcal{B}: v_{j-1}(B) \geq 1\right\} \ \text{ and }\ K \subset \bigcup_{B \in \mathcal{B}_j} 3 B. \label{constr1}
\end{eqnarray}
Furthermore, we can define a function $v_m: \mathcal{B} \rightarrow \mathbb{Z}_+$ similarly to the formula \eqref{constr}.
This is possible since the family $\left\{B \in \mathcal{B}: v_j(B) \geq 1\right\}$ covers $K$, which in fact follows from
\begin{equation} \label{ineq}
\sum_{v_j(B) \geq 1} v_j(B)\cdot \chi_B (x)\ge m - j
\end{equation}
for each $j= 0, 1, \cdots, m - 1$ and any $x\in K$. From the definition of $v_0$, it is easy to obtain \eqref{ineq} for $j = 0$; and then, combined with the disjointness of the family $\mathcal{B}_{j- 1}$, it is not hard to obtain \eqref{ineq} for $j= 1, \cdots, m - 1$ by induction.
With the above construction, we have
	$$
	\begin{aligned}
		m\cdot \mathcal{M}_{N, 3 \varepsilon}^{b}(K) & \leq \sum_{j=1}^m \sum_{B \in \mathcal{B}_j} b(B) \ (\text{combining definition of $\mathcal{M}_{N, 3 \varepsilon}^{b}(K)$ with \eqref{constr1}}) \\
 & = \sum_{j=1}^m \sum_{B \in \mathcal{B}_j}{\big (}v_{j-1}(B)-v_j(B) {\big )} b(B) \ (\text{using \eqref{constr}}) \\
 & \le \sum_{j=1}^m \sum_{B \in \mathcal{B}}{\big (}v_{j-1}(B)-v_j(B) {\big )} b(B) \\
		& = \sum_{B \in \mathcal{B}} \sum_{j=1}^m {\big (}v_{j-1}(B)-v_j(B){\big )} b(B)\leq \sum_{B \in \mathcal{B}} u(B) b(B)\le \sum_{i \in \mathcal{I}} c_i b(n_i).
	\end{aligned}
	$$
That is,
$$\mathcal{M}_{N, 3 \varepsilon}^{b}(K) \leq \sum_{i \in \mathcal{I}} \frac{c_i}{m} b(n_i).$$
Note that we may assume, by approximating $c_i$'s close enough from above using positive rationals, then the inequality \eqref{need-7} follows from arbitrariness of $0<t<1$.
\end{proof}

\subsubsection{Proof of $(3) \Longrightarrow (1)$ of Theorem \ref{MME Bowen-s6}}\

\smallskip

To prove $(3) \Longrightarrow (1)$ of Theorem \ref{MME Bowen-s6}, we also need the following dynamical Frostman's lemma generalizing \cite[Lemma 3.4]{Feng-Huang2012} where they have dealt with the special case of $b(n) = e^{-ns}$ for each $n\in \mathbb{N}$. As commented in \cite{Feng-Huang2012}, the proof originates from Howroyd's elegant argument of traditional Frostman's Lemma (see for example \cite[Theorem 2]{Howroyd1995} or \cite[Theorem 8.17]{Mattila1995}). We provide here a proof of it for completeness.

\begin{prop}\label{Dynamic Frostman lemma}
Let $K \subset X$ be a compact set. Suppose that $b: \N\to \mathbb{R}_+$ is a gauge function, and
$N \in \mathbb{N}, \varepsilon>0$ satisfy $\mathcal{W}_{N, \varepsilon}^b(K)>0$. Then there exists $\mu\in \mathcal{M} (X)$ such that
	$$\mu (K) =1\ \text{and}\
	\mu\left(B_n(x, \varepsilon)\right) \le \frac{1}{c} b(n)
	$$
for any $x \in X$ and each $n \ge N$, where $c = \mathcal{W}_{N, \varepsilon}^b(K)$.
\end{prop}

\begin{proof}
		Note $c<\infty$ as $K \subset X$ is compact.
Let $C(X)$ be the set of all continuous real-valued functions on $X$, equipped with the supremum norm $\|\bullet\|_{\infty}$. We introduce a functional $p$ as
	$$
	p: f\mapsto \frac{1}{c}\cdot \mathcal{W}_{N, \varepsilon}^b\left(\chi_K \cdot f\right).
	$$
If let $\boldsymbol{1} \in C(X)$ denote the constant 1 function over the space $X$. It is easy to verify that
	\begin{itemize}

		\item $p(\boldsymbol{1})=1$ and $p(f+g) \le p(f)+p(g)$ for any $f, g \in C(X)$.

		\item $p(t f)=t p(f)$ for any $t \ge 0$ and each $f \in C(X)$.

		\item $0 \le p(f) \le\|f\|_{\infty}$ for any $f \in C(X)$, and $p(g)=0$ for $g \in C(X)$ with $g \le 0$.
	\end{itemize}

	By the Hahn-Banach theorem, we can extend the linear functional $t \cdot \boldsymbol{1} \mapsto t \cdot p(\mathbf{1}), t \in \mathbb{R}$ from the subspace of all constant functions, to a linear functional $L: C(X) \rightarrow \mathbb{R}$, satisfying
	$$
	L(\boldsymbol{1})=p(\boldsymbol{1})=1 \quad \text { and } \quad-p(-f) \le L(f) \le p(f) \text { for any } f \in C(X) .
	$$
	If $f \in C(X)$ satisfies $f \ge 0$, then $p(-f)=0$ and so $L(f) \ge 0$. Hence, combining with the fact that $L(\boldsymbol{1})=1$, we can use the Riesz representation theorem to find $\mu\in \mathcal{M} (X)$ with $L(f)=\int f d \mu$ for each $f \in C(X)$. Now we prove that $\mu$ is the required measure.
	
For any given $x \in X$ and $n \ge N$, we take arbitrarily compact $E \subset B_n(x, \varepsilon)$. By the Uryson lemma, there exists $f \in C(X)$ satisfying $0 \le f \le 1$ and $f(y)=1$ for any $y \in E$ and $f(y)=0$ for any $y \in X \backslash B_n(x, \varepsilon)$. Clearly $\mu(E) \le \int f d \mu = L(f) \le p(f)$. Since $f \cdot \chi_K \le \chi_{B_n(x, \varepsilon)}$ and $n \ge N$, we have by the definition $\mathcal{W}_{N, \varepsilon}^b \left(\chi_K \cdot f\right) \le b(n)$, then $p(f) \le \frac{1}{c} b(n)$ and so $\mu(E) \le \frac{1}{c} b(n)$. It follows that
$\mu\left(B_n(x, \varepsilon)\right) \le \frac{1}{c} b(n)$ by the arbitrariness of compact $E \subset B_n(x, \varepsilon)$.
Similarly, take arbitrarily compact $F \subset X \backslash K$, by the Uryson lemma we choose $g \in C(X)$ satisfying $0 \le g \le 1$ and $g(x)=1$ for any $x \in F$ and $g(x)=0$ for any $x \in K$.
Then $g \cdot \chi_K \equiv 0$ and so $p(g)=0$. Hence $\mu(F) = \int g d \mu \le L(g) \le p(g)=0$. It follows that $\mu(X \backslash K)=0$, equivalently, $\mu(K)=1$.	This finishes the proof.
\end{proof}

Now we are ready to prove $(3) \Rightarrow (1)$  of Theorem \ref{MME Bowen-s6}.

\begin{proof}[Proof of $(3) \Rightarrow (1)$ of Theorem \ref{MME Bowen-s6}]
Assume, by the item $(3)$, that there exists some gauge function $b$ such that $0< \mathcal{M}^b(Z) {\big (} = \mathcal{W}^b(Z)$ by Proposition \ref{equi of weighted and ordinary entropy function} ${\big )}$ and $\lim\limits_{n\to \infty}e^{ns}b(n)=0$ for any $s < h_\mathrm{top}^B (T, Z)$ (we fix $s$ temporarily). Then there exists $N \in \mathbb{N}, \varepsilon>0$ such that $c:=
	\mathcal{W}_{N, \varepsilon}^b(Z)>0$, and so by applying Proposition \ref{Dynamic Frostman lemma} to the compact set $Z$, there exists $\mu \in \M(X)$ such that $\mu(Z)=1$ and $\mu\left(B_n(x, \varepsilon)\right) \le \frac{1}{c} b(n)$ for any $x \in X$ and each $n \ge N$. As we have assumed $\lim\limits_{n\to \infty}e^{ns}b(n)=0$, one has that, for $\mu$-a.e. $x\in X$,
$$\underline{h}_\mu(T, x) \ge \liminf_{n \rightarrow\infty} - \frac{1}{n} \log \mu\left(B_{n}(x, \varepsilon)\right)
\ge \liminf\limits_{n \to \infty} -\frac{1}{n} \log \frac{1}{c}b(n)\geq s,$$
and so $\underline{h}_\mu(T)\ge s$. Finally, $\underline{h}_\mu(T)\ge h_\mathrm{top}^B (T, Z)$ follows from arbitrariness of $s < h_\mathrm{top}^B (T, Z)$ and then the measure $\mu$ is a measure of maximal entropy for $Z$. This ends the proof.
	\end{proof}

\subsection{Proof of $(2) \Longrightarrow (3)$ of Theorem \ref{MME Bowen-s6} for $h$-expansive systems}\

\smallskip

Note that this in fact follows from the result stated below.

\begin{theorem}\label{construction_b_nonzero_meas}
	Let $K\subset X$ be a compact subset and $\varepsilon>0$. For each $i\in \mathbb{N}$ let $b_i$ be a gauge function satisfying $\lim\limits_{n\to \infty}\frac{b_{i+1}(n)}{b_i(n)}=0$. Then one of the following statements holds:
\begin{enumerate}

\item
The set $K$ is \emph{$(2\varepsilon, \{b_i\}_{i\in \mathbb{N}})$-null} (or just simply \emph{$2\varepsilon$-null}), that is,
 the set $K$ can be expressed as the union of a sequence $\{K_i\}_{i\in \mathbb{N}}$ of analytic sets, in the form $K = \bigcup_{i\in \mathbb{N}} K_i$ with $\M_{2\varepsilon}^{b_i} (K_i) = 0$ for each $i\in \mathbb{N}$.

\item
There is a gauge function $b$ such that $\M_{\varepsilon}^{b}(K)>0$ and
$\lim\limits_{n\to\infty}\frac{b(n)}{b_i(n)}=0$ for any $i\in \mathbb{N}$.	
\end{enumerate}
\end{theorem}

Let us prove firstly $(2) \Rightarrow (3)$ in Theorem \ref{MME Bowen-s6}, with help of Theorem \ref{construction_b_nonzero_meas}. We pick $0< s_1< s_2< \cdots \nearrow h_\mathrm{top}^B(T, Z)$, and choose the parameter $\varepsilon>0$ by Lemma \ref{wangtao expan} for $h$-expansive system $(X, T)$ such that
   $h_\mathrm{top}^B(T, E, 2 \varepsilon)=h_\mathrm{top}^B(T, E)$ for all $E\subset X$.
Now assume, by Theorem \ref{MME Bowen-s6} (2), that the set $Z$ has no increasing countable slice, and so it cannot be written as the union of a sequence of analytic sets $\bigcup_{i\in \mathbb{N}} K_i$ with $\M_{2 \varepsilon}^{s_i}(K_i)=0$ for each $i\in \mathbb{N}$ (otherwise, $h_\mathrm{top}^B(T, K_i, 2 \varepsilon)=h_\mathrm{top}^B(T, K_i) \le s_i< h_\mathrm{top}^B(T, Z)$ for each $i\in \mathbb{N}$, a contradiction to the assumption).
That is, if we define a gauge function $b_i: n\mapsto e^{- n s_i}$ for each $i\in \mathbb{N}$, then $Z$ is not $(2 \varepsilon, \{b_i\}_{i\in \mathbb{N}})$-null, and so by Theorem \ref{construction_b_nonzero_meas} there is a gauge function $b$ such that $\M_{\varepsilon}^{b} (Z)>0$ and $\lim\limits_{n\to\infty} b(n) e^{n s_i}=0$ for any $i\in \mathbb{N}$. This arrives at Theorem \ref{MME Bowen-s6} (3).

\smallskip
	
Now let us prove Theorem \ref{construction_b_nonzero_meas}.

\begin{proof}[Proof of Theorem \ref{construction_b_nonzero_meas}]
Assume that $K$ is a compact set which is not $2\varepsilon$-null. We are going to construct inductively a sequence of integers $\{\ell_i\}_{i\in \mathbb{N}}$ and a sequence of gauge functions $\{d_i\}_{i\in \mathbb{N}}$, based on two arbitrarily chosen sequences $\{\varepsilon_i\}_{i\in \mathbb{N}}$ and $\{\theta_i\}_{i\in \mathbb{N}}$ with
\begin{equation} \label{para}
1 = \theta_1> \theta_2> \cdots \searrow \theta> 0\ \ \text{ and }\ \ 2 \varepsilon = \varepsilon_1> \varepsilon_2> \cdots \searrow \varepsilon.
\end{equation}
The ideal gauge function $b$ will be defined as the limit function of $\{d_i\}_{i\in \mathbb{N}}$ with $\M_{\varepsilon}^{b}(K)\ge \theta> 0$. Our construction is inspired by Rogers' work on Hausdorff measures \cite{Rogers1962}.
	
Firstly if we assume that there exists $L\in \mathbb{N}$ such that, for any integer $\ell\geq L$, there exists a sequence of Bowen balls ${\bigg \{}F_j^{(\ell)}=B_{n_j^{(\ell)}} \left(x_j^{(\ell)},\varepsilon_1\right){\bigg \}}_{j\in \mathbb{N}}$ satisfying
$$\sum_{j\in \mathbb{N}} b_1\left(n_j^{(\ell)}\right)<\theta_1,\ \ \ \min_{j\in \mathbb{N}} n_j^{(\ell)}\geq \ell, \ \ \
K_\ell:= K\setminus \bigcup_{j\in \mathbb{N}} F_j^{(\ell)}\ \text{ is $\varepsilon_1$-null}.$$
It is easy to check that, by Proposition \ref{fact} (2), $\bigcup_{\ell\geq L}K_\ell$, as union of a countable family of $\varepsilon_1$-null sets, is also $\varepsilon_1$-null. Note that $$K_0\doteq K\setminus \bigcup_{\ell\geq L} K_\ell \subset \bigcup_{j\in \mathbb{N}} F_j^{(i)}$$
for each $i\ge L$, it follows that $\M_{\varepsilon_1}^{b_1}(K_0)\le \theta_1$ as $\sum_{j\in \mathbb{N}} b_1\left(n_j^{(i)}\right)<\theta_1$ for any $i\ge L$ and then $\M_{\varepsilon_1}^{b_2}(K_0)= 0$ by Lemma \ref{com}. Particularly, the set $K_0$ is $\varepsilon_1$-null, a contradiction to the assumption that $K = K_0\cup \bigcup_{\ell\geq L} K_\ell$ is not $\varepsilon_1$-null (reminder that $\varepsilon_1 = 2 \varepsilon$).

Thus, by above arguments and the assumption $\lim\limits_{n\to \infty}\frac{b_{2}(n)}{b_1(n)}=0$, we can pick $\ell_1\in \mathbb{N}$ large enough such that $b_1(n)>b_2(n)$ for all $n\geq \ell_1$ and, for the gauge function $d_1 := b_1$, no sequence $\{F_j=B_{n_j}(x_j,\varepsilon_1)\}_{j\in \mathbb{N}}$ of Bowen balls satisfying
$$\sum_{j\in \mathbb{N}}d_1(n_j)<\theta_1,\ \ \
\min_{j\in \mathbb{N}} n_j\geq \ell_1,\ \ \
K\setminus \bigcup_{j\in \mathbb{N}} F_j\ \text{ is $\varepsilon_1$-null}.$$

Now, for each $i\geq 2$, we assume that integers $\ell_1, \cdots, \ell_{i-1}$ and gauge functions $d_1, \cdots, d_{i-1}$ have been constructed, such that
\begin{enumerate}

		\item[(P1)] Integers $1<\ell_1< \cdots <\ell_{i-1}$ satisfy $b_{p-1} (n) > b_p (n)$
for all $n\ge \ell_{p- 1}$ and $2\le p\le i$.
		
		\item[(P2)] Gauge functions $d_1, \cdots, d_{i-1}$ satisfy that, for every $i\ge p\geq 3$,
$$\begin{aligned}
	d_{p-1}(n) = \begin{cases}d_{p-2}(n) &\text { if } 1\le n <\ell_{p-2}\\
	b_{p-1}(n)&\text { if } n \geq \ell_{p-2}\end{cases}.
\end{aligned}$$
		
		\item[(P3)] For each $i\ge p\geq 2$, no sequence $\{F_j=B_{n_j}(x_j,\varepsilon_{p-1})\}_{j\in \mathbb{N}}$ of Bowen balls satisfying
$$\sum_{j\in \mathbb{N}} d_{p-1} (n_j)< \theta_{p-1}, \ \ \ \min_{j\in \mathbb{N}} n_j\geq \ell_1,\ \ \
K\setminus \bigcup_{j\in \mathbb{N}} F_j\ \text{ is $\varepsilon_{p-1}$-null}.$$
	\end{enumerate}

\begin{claim} \label{cl}
By inductive construction, there exists integer $\ell_i$ and gauge function $d_i$ with the required properties similar to above items (P1), (P2) and (P3).
\end{claim}

Let us postpone temporarily the proof of Claim \ref{cl}, and continue firstly the proof of Theorem \ref{construction_b_nonzero_meas}.
In particular, by applying Claim \ref{cl}, there exist integers $1<\ell_1<\ell_2< \cdots$ and gauge functions $d_1, d_2, \cdots$ with properties (P1), (P2) and (P3).
Note that $d_{i+1}(n)\leq d_i(n)$ for all $i\in \mathbb{N}$ and $n\in \mathbb{N}$ by construction of integers $\ell_1, \ell_2, \cdots$ and the property (P2). And so we may define a function $b$ as the limit of the sequence $\{d_i\}_{i\in \mathbb{N}}$, which is easily verified to be a gauge function.

It remains to prove $\M^{b}_{\varepsilon}(K)\geq\theta$, which ends the proof.

Otherwise, we can find a countable family of Bowen balls $\{B_{n_j}(x_j,\varepsilon)\}_{j\in \mathcal{J}}$ covering $K$ such that $\min\limits_{j\in \mathcal{J}} n_j\geq \ell_1$ and $\sum\limits_{j\in \mathcal{J}} b(n_j)<\theta.$ Note that $K$ is a compact set, we may assume that the family $\mathcal{J}$ is finite, and then there exists $M\geq 1$ such that $\max\limits_{j\in \mathcal{J}} n_j< \ell_M$. Thus
\begin{equation} \label{r1}
\min_{j\in \mathcal{J}} n_j\geq \ell_1, \ \ \ K\setminus \bigcup_{j\in \mathcal{J}} B_{n_j} (x_j, \varepsilon_M)\ \text{ is empty and hence obviously $\varepsilon_{M}$-null}
\end{equation}
(reminder that $\varepsilon_M\ge \varepsilon$ by the construction \eqref{para});
furthermore, as $\max\limits_{j\in \mathcal{J}} n_j< \ell_M$, and so by the construction of sequence $d_1, d_2, \cdots$ of gauge functions, one has $d_M (n_j) = d_{M+ 1} (n_j) = d_{M+ 2} (n_j) = \cdots$ and hence $d_M (n_j) = b (n_j)$ for each $n_j, j\in \mathcal{J}$, which implies that
\begin{equation} \label{r2}
\sum_{j\in \mathcal{J}} d_M(n_j) = \sum_{j\in \mathcal{J}} b(n_j)<\theta<\theta_M \ (\text{by \eqref{para}}).
\end{equation}
In particular, \eqref{r1} and \eqref{r2} contradict to the construction of the property (P3) for $d_M$ (if necessary, we extend $\mathcal{J}$ to $\mathbb{N}$ by choosing $\left\{B_{n_j}(x_j,\varepsilon_{M})\right\}_{j\in \mathbb{N}\setminus \mathcal{J}}$ such that
$$\min_{j\in \mathbb{N}\setminus \mathcal{J}} n_j\geq \ell_1\ \ \text{and}\ \ \sum_{j\in \mathbb{N}\setminus \mathcal{J}} d_M(n_j) < \theta- \sum_{j\in \mathcal{J}} b(n_j) = \theta - \sum_{j\in \mathcal{J}} d_M(n_j),$$
where $\mathcal{J}$ can be viewed naturally a subset of $\mathbb{N}$; this is possible once $\min\limits_{j\in \mathbb{N}\setminus \mathcal{J}} n_j$ is large enough, as $d_M$ is a gauge function). Thus it arrives $\M^{b}_{\varepsilon}(K)\geq\theta$.
	\end{proof}

Now let us prove Claim \ref{cl}, and so finish the proof of Theorem \ref{construction_b_nonzero_meas}.

\begin{proof}[Proof of Claim \ref{cl}]
Now fix arbitrarily given integer $i\geq 2$, for which integers $\ell_1, \cdots, \ell_{i-1}$ and gauge functions $d_1, \cdots, d_{i-1}$ have been constructed with properties (P1), (P2) and (P3). As $\lim\limits_{n\to \infty} \frac{b_{i+1}(n)}{b_i(n)}=0$, we choose integer $L_i>\ell_{i-1}$ with $b_{i+1}(n)<b_i(n)$ for all $n\geq L_i$.

Let us assume that the inductive construction fails from the $(i-1)$-th step to the $i$-th step, that is,
for each $\ell\geq L_i$, if we define a gauge function
\begin{equation} \label{mq1}
		d_{i}^{(\ell)}(n) = \begin{cases}d_{i-1}(n) &\text { when } n <\ell\\
			b_{i}(n) &\text { when } n \geq \ell\end{cases},
\end{equation}
	then there exists a sequence of Bowen balls ${\bigg \{}F_j^{(\ell)}=B_{n_j^{(\ell)}} \left(x_j^{(\ell)},\varepsilon_i\right){\bigg \}}_{j\in \mathbb{N}}$ such that
\begin{equation} \label{f2}
\sum_{j\in \mathbb{N}}d_{i}^{(\ell)}(n_j^{(\ell)})<\theta_{i},\ \ \ \min_{j\in \mathbb{N}} n_j^{(\ell)}\geq \ell_1,\ \ \ \text{and}\ E_\ell: = K\setminus \bigcup_{j\in \mathbb{N}} F_j^{(\ell)}\ \text{ is $\varepsilon_{i}$-null}.
\end{equation}

It makes no difference to assume $n_1^{(\ell)}\leq n_2^{(\ell)}\leq \cdots $.
Note that the family of all nonempty compact subsets of $X$ is compact with respect to the Hausdorff distance, there exists a compact set $F_1^*$ and a subsequence $\ell_{k(1)}^*$ of integers, such that the closure of $F_1^{\left(\ell^*_{k(1)}\right)}$ tends to $F_1^*$ in the Hausdorff distance. Similarly, there exists a compact set $F_2^*$ and a subsequence $\ell_{k(2)}^*$ of $\ell_{k(1)}^*$, such that the closure of $F_2^{\left(\ell_{k(2)}^*\right)}$ tends to $F_2^*$ in the Hausdorff distance. We could proceed similarly. Then, after picking successive subsequences and using the standard diagonal argument, we can choose a sequence $\ell_k^*$ of integers such that, for every $j\in \mathbb{N}$, the closure of $F_j^{\left(\ell_{k}^*\right)}$ tends to $F_j^*$ in the Hausdorff distance as $k$ tends to $\infty$.
	Moreover, by choosing a suitable subsequence if necessary, we may assume additionally that
$$x_j^{(\ell_k^*)}\to x_j^*\in X\text{ and }n_j^{(\ell_k^*)}\to n_j^*\in\Z_+\cup\{\infty\}.$$
	Clearly, we have either $n_j^*\ (\ge \ell_1)\ \in\mathbb{N}$ ${\Big (}$in which case $F_j^*=\overline{B_{n_j^*}(x_j^*,\varepsilon_i)}$${\Big )}$ or $n_j^*=\infty$ ${\Big (}$in this case, by the definition of Hausdorff distance, one has $F_j^*\subset \bigcap_{N\geq 1}\overline{B_{N}(x_j^*,\varepsilon_i)}$${\Big )}$. Set
$$\mathcal{F}= \{j\in \mathbb{N}: n_j^*\in \mathbb{N}\}.$$
	
For each $J\in \mathbb{N}$, we choose $k= k(J)$ large enough such that $\max\limits_{1\leq j\leq J, j\in \mathcal{F}} n_j^* < \ell_k^*$ and $n_j^* = n_j^{(\ell_k^*)}$ for all $1\leq j\leq J$ with $j\in \mathcal{F}$, and hence
$$\sum_{1\leq j\leq J, j\in \mathcal{F}} d_{i-1}(n_j^*)
= \sum_{1\leq j\leq J, j\in \mathcal{F}} d_{i}^{(\ell_k^*)}(n_j^*) = \sum_{1\leq j\leq J, j\in \mathcal{F}}
d_{i}^{(\ell_k^*)}(n_j^{(\ell_k^*)})< \theta_{i}$$
by the construction \eqref{mq1} of $d_{i}^{(\ell)}(n)$,
thus by letting $J\to \infty$ we obtain
 \begin{equation} \label{finite}
\sum_{j\in \mathcal{F}} d_{i-1}(n_j^*)\leq \theta_{i}<\theta_{i-1}\ \ \ \left(\text{where, by convention,} \sum_{j\in \mathcal{F}} d_{i-1}(n_j^*)
= 0\ \text{if}\ \mathcal{F} = \emptyset\right).
	\end{equation}

As $d_{i - 1}$ is a gauge function, we can choose $m_j^*\ge \ell_1$ for each $j\in \mathbb{N}\setminus \mathcal{F}$ such that
 \begin{equation*}
\sum_{j\in \mathbb{N}\setminus \mathcal{F}} d_{i-1}(m_j^*)< \theta_{i - 1} - \sum_{j\in \mathcal{F}} d_{i-1}(n_j^*)\ (\text{using \eqref{finite}}),
	\end{equation*}
 and hence if we set $m_j^*=n_j^*$ for any $j\in \mathcal{F}$ then by the construction we have
\begin{equation} \label{2025}
\sum\limits_{j\in \mathbb{N}} d_{i-1}(m_j^*)<\theta_{i-1}\ \ \ \text{and}\ \ \ \min_{j\in \mathbb{N}} m_j^*\geq \ell_1.
\end{equation}

Notice that, once $\ell\geq L_i$ and $n\geq \ell_{i-1}$, we always have $d_i^{(\ell)}(n)= b_i(n)$ by applying the property (P2) for $p = i$ and the construction \eqref{mq1}. In fact, if $n\ge \ell$ then $d_i^{(\ell)}(n)= b_i(n)$, and if $n< \ell$ then
$d_i^{(\ell)}(n) = d_{i - 1} (n) = b_{i-1} (n)$ where the first identity follows by applying \eqref{mq1} to $n\ge \ell$ and the second identity follows by applying (P2) and \eqref{mq1} to $\ell> n\ge \ell_{i-1}$.

Now let us prove $\M^{b_i}_{\varepsilon_{i-1}} (K^*)\leq \theta_{i}$, where
$$K^*=K \setminus \left(\bigcup_{j\in \mathbb{N}} B_{m_j^*} (x_j^*, \varepsilon_{i-1}) \cup E\right)\ \ \ \text{with}\ \ E = \bigcup_{k\in \mathbb{N}} E_{\ell_k^*}.$$

In fact, in the following we show that for any fixed $N\geq \ell_{i-1}$, once $k\in \mathbb{N}$ is large enough, then the family $\left\{F_j^{(\ell_k^*)}: j\in \mathbb{N}\ \text{satisfies}\ n_j^{(\ell_k^*)}\geq N\right\}$ covers $K^*$ (and $\ell_k^*\ge L_i$), which implies
\begin{eqnarray*}
	\M^{b_i}_{N,\varepsilon_{i-1}}(K^*) & \leq &
\sum_{j\in \mathbb{N}\ \text{satisfies}\ n_j^{(\ell_k^*)}\geq N} b_i\left(n_j^{(\ell_k^*)}\right)\ (\text{note that $\varepsilon_{i- 1}> \varepsilon_i$}) \\
& = & \sum_{j\in \mathbb{N}\ \text{satisfies}\ n_j^{(\ell_k^*)}\geq N} d_i^{(\ell_k^*)}\left(n_j^{(\ell_k^*)}\right)
	\leq \sum_{j\in \mathbb{N}} d_i^{(\ell_k^*)}\left(n_j^{(\ell_k^*)}\right)<\theta_{i}.
\end{eqnarray*}

Fix $N\geq \ell_{i-1}$. Once $\ell\ge L_i$, we have shown $d_i^{(\ell)}(N)= b_i(N)$, and then
\begin{eqnarray*}
& & \#\left\{j\in \mathbb{N} : n_j^{(\ell)}<N\right\}\cdot b_i (N) \\
& = & \sum_{j\in \mathbb{N} , n_j^{(\ell)}<N} d_i^{(\ell)}(N) \leq \sum_{j\in \mathbb{N} , n_j^{(\ell)}<N} d_i^{(\ell)}\left(n_j^{(\ell)}\right)\ \ \left(\text{as}\ d_i^{(\ell)}\ \text{is a gauge function}\right) \\
& \leq & \sum_{j\in \mathbb{N}} d_i^{(\ell)} \left(n_j^{(\ell)}\right)< \theta_{i}\ \ (\text{by}\ \eqref{f2}),
\end{eqnarray*}
where $\# (\bullet)$ denotes the cardinality of a subset $\bullet$, in particular,
$$\#\left\{j\in \mathbb{N} : n_j^{(\ell)}<N\right\}< J^* := \frac{\theta_{i}}{b_i (N)}.$$
Since we have assumed $n_1^{\left(\ell\right)}\leq n_2^{\left(\ell\right)}\leq \cdots$,
thus $n_j^{(\ell)}\ge N$ for each $j> J^*$.
It is easy to check $F_j^*\subset \overline{B_{m_j^*}(x_j^*,\varepsilon_i)}\subset B_{m_j^*}(x_j^*, \varepsilon_{i-1})$
by the construction of $m_j^*$ for each $j\in \mathbb{N}$.
Reminder that the closure of $F_j^{\left(\ell_{k}^*\right)}$ tends
to $F_j^*$ in the Hausdorff distance as $k$ tends to $\infty$. Obviously, once $k$ is large enough, then for each $1\le j\le J^*$, the set $F_j^{\left(\ell_{k}^*\right)}$ is contained in the set $B_{m_j^*}(x_j^*,\varepsilon_{i-1})$.
Note that, for each $k\in \mathbb{N}$, by the definition of $E$ one has
$$E\supset E_{\ell_k^*} = K\setminus \bigcup_{j\in \mathbb{N}} F_j^{(\ell_k^*)},$$
and then  the countable family $\left\{F_j^{(\ell_k^*)}\right\}_{j\in \mathbb{N}}$ covers $K\setminus E$.
Thus, once $k\in \mathbb{N}$ is large enough, then the family $\left\{F_j^{(\ell_k^*)}: j\in \mathbb{N}\ \text{satisfies}\ n_j^{(\ell_k^*)}\geq N\right\}$ contains $\left\{F_j^{(\ell_k^*)}: j> J^*\right\}$, and so covers
$$K \setminus \left(E\cup \bigcup_{j= 1}^{J^*} B_{m_j^*}(x_j^*,\varepsilon_{i-1})\right) \supset K^*.$$

In particular, we obtain the required inequality $\M^{b_i}_{\varepsilon_{i-1}} (K^*)\leq \theta_{i}$.
Thus $\M^{b_{i+1}}_{\varepsilon_{i-1}}(K^*)=0$ by Lemma \ref{com}.
In particular, $K^*$ is $\varepsilon_{i - 1}$-null.
Note that $E_\ell = K\setminus \bigcup_{j\in \mathbb{N}} F_j^{(\ell)}$ is $\varepsilon_{i}$-null and then $\varepsilon_{i-1}$-null for any $\ell\ (\ge L_i)$ by the assumption. Recall
$$K^*=K \setminus \left(\bigcup_{j\in \mathbb{N}} B_{m_j^*} (x_j^*, \varepsilon_{i-1}) \cup E\right)\ \ \ \text{with}\ \ E = \bigcup_{k\in \mathbb{N}} E_{\ell_k^*}.$$
One has that
$K \setminus \bigcup_{j\in \mathbb{N}} B_{m_j^*} (x_j^*, \varepsilon_{i-1}) \subset K^* \cup E$
is also $\varepsilon_{i-1}$-null. This is a contraction to the inductive assumption for $d_{i - 1}$ and $\theta_{i - 1}$ by observing \eqref{2025}. Thus we can proceed the inductive construction, and so finish the proof of Claim \ref{cl}.
\end{proof}

\section{Proof of Theorem \ref{MME Bowen} for analytic subsets in $h$-expansive systems}\label{Relationship between items (1) and (2)}

As we have proved $(1) \Rightarrow (2)$ and $(3) \Rightarrow (2)$ of Theorem \ref{MME Bowen} in \S \ref{easy},
the remainder parts of
Theorem \ref{MME Bowen}
 follow from Theorem \ref{MME Bowen-s6} and Theorem \ref{MME Bowen-s7} stated as below.

\begin{theorem}\label{MME Bowen-s7}
Assume that the system $(X,T)$ is $h$-expansive. Let $Z\subset X$ be an analytic subset with $h_\mathrm{top}^B(T, Z)>0$. Then $(2) \Rightarrow (1) \Rightarrow (3)$:
	\begin{enumerate}

		\item[(1)] The subset $Z$ has measures of maximal Bowen entropy.

		\item[(2)] The subset $Z$ has no increasing countable slice of Bowen entropy.

		\item[(3)] There exists some gauge function $b$ such that $\mathcal{M}^b(Z)>0$ and $\lim\limits_{n\to \infty}e^{ns}b(n)=0$ for any $s < h_\mathrm{top}^B (T, Z)$.
	\end{enumerate}
\end{theorem}

As before, we will omit the term \emph{Bowen entropy} if there is no confusion.

\subsection{Proof of $(1) \Longrightarrow (3)$ of Theorem \ref{MME Bowen-s7}}\

\smallskip

Note that in the definition of $\M^b(Z)$ (of a given gauge function $b$), the gauge function $b$ only characterize the change of $n_i$ with respect to a Bowen ball $B_{n_i}(x_i,\varepsilon)$, which contains no information about the radius $\varepsilon$. However in the proof of $(1) \Longrightarrow (3)$ in Theorem \ref{MME Bowen-s7}, we have to construct a gauge function based on a measure of maximal entropy, where in the definition we always require the radius $\varepsilon$ go to zero eventually. Therefore one of our key ingredients in the proof is to fix a radius $\delta$, and make sure the measure does not change too much with respect to such radius. This is done as follows.

\begin{lemma}\label{techlem MME mu to nu}
Under the assumptions of Theorem \ref{MME Bowen-s7}, let $\mu\in\mathcal{M}(X)$ be a measure of maximal entropy, with $\mu(Z)=1$ and $\underline{h}_\mu(T)=h_\mathrm{top}^B(T, Z)$. Then we could modify $\mu$ into another measure $\nu\in\mathcal{M}(X)$ of maximal entropy, satisfying $\nu(Z)=1$ and $\underline{h}_\nu(T)=h_\mathrm{top}^B(T, Z)$, furthermore, there exists $\delta>0$ such that
$$\underline{h}_{\nu}(T, x, \delta)=\underline{h}_{\nu}(T, x)=h_\mathrm{top}^B(T, Z)\ \ \text{for $\nu$-a.e. $x\in X$}.$$
\end{lemma}

\begin{proof}
	Choose a sequence $\{s_i\}_{i\in \mathbb{N}}$ of positive real numbers increasing up to $h_\mathrm{top}^B(T, Z)$. By Theorem \ref{dynamical density thm}, $\underline{h}_{\mu}(T, x)= h_\mathrm{top}^B(T, Z)$ for $\mu$-a.e. $x\in X$. Therefore for any $i\in \mathbb{N}$,
	$$\mu\left( \bigcup\limits_{\varepsilon>0} \left\lbrace x\in Z:  \underline{h}_{\mu}(T, x, \varepsilon) > s_i\right\rbrace\right) =1.
	$$
Since $\underline{h}_{\mu}(T, x, \varepsilon)$ increases to $\underline{h}_{\mu}(T, x)$ as $\varepsilon$ decreases to $0$, we may choose inductively a sequence $\varepsilon_1>\varepsilon_2>\cdots$ of positive real numbers such that $$\mu(E_i)>1-\frac{1}{2^{i+1}}\ \ \ \text{where}\ 	E_i=\left\lbrace x\in Z:  \underline{h}_{\mu}(T, x, \varepsilon_i) > s_i\right\rbrace\ \text{for each}\ i\in \mathbb{N}.$$
	
Set $E=\bigcap_{i\in \mathbb{N}} E_i$. One has $\mu(E)> \frac{1}{2}$. We let $\mu_1$ be the normalized measure of $\mu$ restricted to $E$. Therefore, for $\mu_1$-a.e. $x\in Z$, it holds that $$\underline{h}_{\mu_1}(T, x, \varepsilon_i) = \liminf _{n \rightarrow\infty}-\frac{1}{n} \log \mu_1 \left(B_{n}(x, \varepsilon_i)\right) \geq \underline{h}_{\mu}(T, x, \varepsilon_i) > s_i\ \ \ \text{for each}\ i\in \mathbb{N}.$$
Similarly, we choose inductively a sequence $N_1<N_2<\cdots$ of natural integers such that $\mu_1(F_i)>1-\frac{1}{2^{i+1}}$ where $$F_i=\bigcap_{n\geq N_i}\left\lbrace x\in Z:  -\frac{1}{n}\log\mu_1(B_n(x,\varepsilon_i)) > s_i\right\rbrace\ \text{for each}\ i\in \mathbb{N}.$$
	
Set $F=\bigcap_{i\in \mathbb{N}} F_i$. One has $\mu_1(F)> \frac{1}{2}$. We let $\nu$ be the normalized measure of $\mu_1$ restricted to $F$. Therefore, for $\nu$-a.e. $x\in Z$, it holds that, for any $i\in \mathbb{N}$ and each $n\geq N_i$,
	\begin{equation}\label{cont for nu}	
\nu(B_n(x,\varepsilon_i)) = \frac{\mu_1(B_n(x,\varepsilon_i)\cap F)}{\mu_1(F)}\leq 2\mu_1(B_n(x,\varepsilon_i))<2e^{-ns_i}.
	\end{equation}
	
	Since the system $(X,T)$ is $h$-expansive, there exists $\delta>0$ such that $h^*_T (\delta) = 0$. Therefore, by \cite[Propsition 2.2]{R.Bowen1972-TAMS}, for each $i\in \mathbb{N}$ and any $\beta>0$, there exists a finite constant $c(\varepsilon_i, \beta)$ such that $\#\Lambda_{n, x} \leq c(\varepsilon_i,\beta)\cdot e^{\beta n}$ for each $n\in \mathbb{N}$ and any $x\in X$, where $\Lambda_{n, x}$ is an $(n, \frac{\varepsilon_i}{3})$-spanning set of $B_n (x,\delta)$ with minimal cardinality.
	
For all $n\in \mathbb{N}$ and $x\in X$, we obtain a new set $\Lambda'_{n, x}\subset \Lambda_{n, x}$ by throwing away all the points $y\in\Lambda_{n, x}$ for which $\nu\left(\overline{B}_n(y,\frac{\varepsilon_i}{3})\right)=0$.
Whenever $\Lambda'_{n, x}\neq \emptyset$, for each $y\in\Lambda'_{n, x}$, $\nu\left(\overline{B}_n(y,\frac{\varepsilon_i}{3})\right) >0$ and so there exists $z(y)\in\overline{B}_n(y, \frac{\varepsilon_i}{3})$ such that \eqref{cont for nu} holds for $z (y)$ whenever $i\in \mathbb{N}$ and $n\geq N_i$. Note that $\Lambda_{n, x}$ is an $(n, \frac{\varepsilon_i}{3})$-spanning set of $B_n (x,\delta)$, $B_n (x,\delta)\subset \bigcup_{y\in \Lambda_{n, x}} \overline{B}_n(y,\frac{\varepsilon_i}{3})$, and then
clearly, for $\nu$-a.e. $x\in Z$, each set $\Lambda'_{n, x}$ is a nonempty set for each $n\in \mathbb{N}$. Thus, for $\nu$-a.e. $x\in Z$, once $n\geq N_i$, one has
	\begin{equation}
		\begin{aligned}
			\nu(B_n(x,\delta))&\leq \sum_{y\in\Lambda'_{n, x}}\nu\left(\overline{B}_n {\big (}y,\frac{\varepsilon_i}{3}{\big )}\right) \\
&\leq \sum_{y\in\Lambda'_{n, x}}\nu(B_n(z(y),\varepsilon_i))\ \left(\text{as $z(y)\in\overline{B}_n {\big (}y, \frac{\varepsilon_i}{3}{\big )}$}\right)\\
			&\leq \#\Lambda'_{n, x}\cdot 2e^{-ns_i}\ (\text{applying \eqref{cont for nu} to $z (y)$})\ \leq 2c(\varepsilon_i,\beta)\cdot e^{-n(s_i-\beta)},
		\end{aligned}
	\end{equation}
and then $\underline{h}_{\nu}(T, x, \delta)\geq s_i-\beta$. Let $i\to\infty$ and $\beta\to 0$, we obtain that $\underline{h}_{\nu}(T, x, \delta)\geq h_\mathrm{top}^B(T, Z)$ and then $\underline{h}_{\nu}(T, x, \delta)= \underline{h}_{\nu}(T, x)= h_\mathrm{top}^B(T, Z)$ for $\nu$-a.e. $x\in X$. This finishes the proof.
	\end{proof}

Now we are ready to prove $(1) \Rightarrow (3)$ of Theorem \ref{MME Bowen-s7}.

\begin{proof}[Proof of $(1) \Rightarrow (3)$ in Theorem \ref{MME Bowen-s7}]
Under the assumptions of Theorem \ref{MME Bowen-s7}, we let $\nu\in \mathcal{M} (X)$ and $\delta>0$ as in Lemma \ref{techlem MME mu to nu}. Take a sequence $\{s_i\}_{i\in \mathbb{N}}$ of positive real numbers increasing up to $h_\mathrm{top}^B(T, Z)$, we could show by induction that, there exists a measurable set $A\subset Z$ with $\nu(A)>0$ and a sequence $\{N_i\}_{i\in \mathbb{N}}$ of positive integers increasing to $\infty$, such that it holds $\nu(B_n(x,\delta))\leq e^{-ns_i}$ for each $x\in A$ and all $i\in \mathbb{N}$ and every $N_i\leq n < N_{i+1}$.

In fact, by construction of $\nu\in\mathcal{M}(X)$ and $\delta>0$ from Lemma \ref{techlem MME mu to nu}, if we set
\begin{equation*}
A_0 =\left\lbrace x\in Z:  \underline{h}_{\nu}(T, x, \delta) = h_\mathrm{top}^B(T, Z) \right\rbrace = \bigcap_{i\in \mathbb{N}} \bigcup_{N\geq 1} \bigcap_{n\geq N} \left\lbrace x\in Z: \nu\left(B_{n}(x, \delta)\right) \leq e^{-ns_i}\right\rbrace,
		\end{equation*}	
then $A_0\subset Z$ is a measurable set satisfying $\nu (A_0)=1$ and $\underline{h}_{\nu}(T, x, \delta)= h_\mathrm{top}^B(T, Z)$ for each $x\in A_0$. From this, it is trivial to see that, we may construct by induction a sequence $N_1< N_2< \cdots$ of large enough integers such that
$\nu (A_i) \geq 1-\frac{1}{2^{i+1}}$ for each $i\in \mathbb{N}$, where
$$A_i= \bigcap_{n\geq N_i} \left\lbrace x\in Z:  \nu\left(B_{n}(x, \delta)\right) \leq e^{-ns_i}    \right\rbrace  \cap A_0$$
is a measurable set.
	Finally, let $A=\bigcap_{i\in \mathbb{N}} A_i$, which is clearly a measurable set. It is direct to check that $A$ has the required properties.

Now we set $b(n) = e^{-ns_i}$ for all $i\in \mathbb{N}$ and $N_i\leq n < N_{i+1}$, which is easily shown to be a gauge function. In the following we prove that the function $b$ has required properties.

As the sequence $\{s_i\}_{i\in \mathbb{N}}$ increases up to $h_\mathrm{top}^B(T, Z)$, for any $s<h_\mathrm{top}^B(T, Z)$, there exists $i_0\in \mathbb{N}$ such that $s_i>s$ for all $i\geq i_0$. Therefore, once $n\geq N_{i_0}$, we take uniquely $i^*\ge i_0$ with $N_{i^*}\leq n < N_{i^*+1}$, and so $e^{ns}b(n) = e^{ns}\cdot e^{-ns_{i^*}} \leq e^{-n(s_{i_0}-s)}$ tends to $0$ as $n$ goes to $\infty$.

It remains to prove $\M^b(Z)>0$.
Take arbitrarily countable family $\left\{B_{n_j}\left(x_j, \frac{\delta}{2}\right)\right\}_{j\in \mathcal{J}}$ covering $A$. It makes no difference to assume that for each $j\in \mathcal{J}$, the Bowen ball $B_{n_j}\left(x_j, \frac{\delta}{2}\right)$ intersects the measurable set $A$ and say $y_j \in B_{n_j}\left(x_j, \frac{\delta}{2}\right) \cap A $. Then, by the construction of the set $A\subset Z$ and the sequence $\{N_i\}_{i\in \mathbb{N}}$ and the definition of $b$, one has
	$$\sum_{j\in \mathcal{J}} b(n_j)  \ge \sum_{j\in \mathcal{J}} \nu\left(B_{n_j}\left(y_j, \delta\right)\right)\ (\text{as $y_j\in A$})\ \ge \sum_{j\in \mathcal{J}} \nu\left(B_{n_j} {\big (}x_j, \frac{\delta}{2}{\big )}\right)
		\ge \nu\left(A\right)>0.$$
By arbitrariness of the family, we obtain $\M^b(A)\geq \nu (A)>0$ and then $\M^b(Z)>0$.
\end{proof}

\subsection{Proof of $(2) \Longrightarrow (1)$ of Theorem \ref{MME Bowen-s7}}\

\smallskip

In the proof, we shall use the concept of principal extension.

\smallskip

Let $\pi: (Y, S) \rightarrow(X, T)$ be an extension between two TDSs. Recall from \cite{Ledrappier1979} that $\pi: (Y, S) \rightarrow(X, T)$ is called a \emph{principal extension} if $h_\mathrm{top}(S,\pi^{-1}(x))=0$ for every $x\in X$.\footnote{\ Principal extension was defined in \cite{Ledrappier1979} originally via the language of relative measure-theoretic entropy, here we provide an equivalent definition of it using the concept of (relative) topological entropy (of fibers) via the so-called Outer variational principle established by \cite[Theorem 3]{Downarowicz-Serafin2002}.}
In this case, if, furthermore, there exists some $M\in \mathbb{N}\setminus \{1\}$ such that $S$ is the shift transformation $\sigma$ over $\{1, \cdots, M\}^\mathbb{Z}$ and $Y\subset \{1, \cdots, M\}^\mathbb{Z}$ is a nonempty closed invariant subset, then
TDS $(Y, S)$ is called a \emph{principal symbolic extension of TDS $(X, T)$} and $\pi: (Y, S) \rightarrow(X, T)$ is called a \emph{principal symbolic extension}.\footnote{\ As explained by Downarowicz in his book (for details see \cite[Page 273]{Downarowicz2011}, in particular, \cite[Definition 9.1.1]{Downarowicz2011}), though the setting of a topological dynamical system we consider in this paper is a compact metric space along with a continuous self-map,
we have to use a bilateral subshift when we characterize an asymptotically $h$-expansive TDS via a principal symbolic extension.}
It is well known that a system is asymptotically $h$-expansive if and only if it admits a principal symbolic extension, for details see \cite[Theorem 9.3.3]{Downarowicz2011} (see also \cite{Downarowicz2001}, \cite{Boyle-Fiebig-Fiebig2002} and \cite{Boyle-Downarowicz2004}).\footnote{\ Such a characterization was given in \cite{Downarowicz2001}, \cite{Boyle-Fiebig-Fiebig2002} and \cite{Boyle-Downarowicz2004} for a topological dynamical system in the sense of a compact metric space along with a self-homeomorphism, and then was given by \cite[Theorem 9.3.3]{Downarowicz2011} for a compact metric space along with a continuous self-map.}

We also need the following result. Recall that a set $Z$ has a \emph{increasing countable slice of Hausdorff dimension}, if there exists a countable collection $\{ Z_i\}_{i\in \mathbb{N}}$ of analytic sets such that $Z=\bigcup_{i\in \mathbb{N}} Z_i$ and $\dim_{\mathcal{H}}(Z_i)<\dim_{\mathcal{H}}(Z)$ for each $i\in \mathbb{N}$.

	\begin{prop}\label{find_compact_in_analytic_Haus}
		If an analytic set $Z$ has no increasing countable slice of Hausdorff dimension, then there exists a compact set $K\subset Z$ with $\dim_{\mathcal{H}}(K)=\dim_{\mathcal{H}}(Z)$ such that $K$ also has no increasing countable slice of Hausdorff dimension.
	\end{prop}
	
	\begin{proof}
		Let $\{\alpha_n\}_{n\in\mathbb{N}}$ be a sequence of positive real numbers strictly increasing up to $\dim_{\mathcal{H}}(Z)$. Since $Z$ has no increasing countable slice of Hausdorff dimension, it cannot be expressed as the form $\bigcup_{n\in \mathbb{N}} Z_n$ with each $Z_n$ an analytic set having $\sigma$-finite $\mathcal{H}^{\alpha_n}$-measure, where recall that $\mathcal{H}^{\alpha_n}$ is the $\alpha_n$-dimensional Hausdorff measure.
 Otherwise, according to Proposition \ref{comp_gauge_Haus}, we have that $\dim_{\mathcal{H}}(Z_n)\leq \alpha_n < \dim_{\mathcal{H}}(Z)$ for each $n\in \mathbb{N}$, which contradicts to our assumption that the set $Z$ has no increasing countable slice of Hausdorff dimension.

By applying \cite[Theorem 6.3]{Sion-Sjerve1962}, we may find a compact set $K\subset Z$
and a gauge function $h: [0,\infty) \to [0,\infty)$ (that is, a continuous function defined on $[0,\infty)$
which is monotonically increasing on $[0, \infty)$, strictly positive on $(0, \infty)$ and satisfies $h(0)=0$), such that $K$ has non-$\sigma$-finite $\mathcal{H}^{h}$-measure and  \begin{equation}\label{h_big_alp_n}
				\lim\limits_{t\to 0}\frac{h(t)}{t^{\alpha_n}}=0\ \ \ \text{for each}\ n \in \mathbb{N}.
		\end{equation}
	
Now let us check that the subset $K$ is the required compact set. Since $K$ has non-$\sigma$-finite $\mathcal{H}^{h}$-measure, in particular, $\mathcal{H}^{h}(K)>0$. For each $n\in \mathbb{N}$, concerning \eqref{h_big_alp_n} and applying Proposition \ref{comp_gauge_Haus} to the gauge functions $h$ and $t^{\alpha_n}$ we have that $\mathcal{H}^{\alpha_n}(K)>0$ and then $\dim_{\mathcal{H}}(K) \geq \alpha_n$. As $K$ is a subset of $Z$, we concludes $\dim_{\mathcal{H}}(K)=\dim_{\mathcal{H}}(Z)$. It remains to prove that the set $K$ has no increasing countable slice of Hausdorff dimension.
If we assume the contrary, equivalently, there exists a countable collection $\{ K_i\}_{i\in \mathbb{N}}$ of analytic sets such that $K=\bigcup_{i\in \mathbb{N}} K_i$ and $\dim_{\mathcal{H}}(K_i)<\dim_{\mathcal{H}}(K) = \dim_{\mathcal{H}}(Z)$ for each $i\in \mathbb{N}$. By the selection of the sequence $\{\alpha_n\}_{n\in\mathbb{N}}$, for each $i\in \mathbb{N}$ we choose $n = n(i)\in \mathbb{N}$ such that $\alpha_n>\dim_{\mathcal{H}}(K_i)$, and then $\mathcal{H}^{\alpha_n}(K_i)=0$. Again by \eqref{h_big_alp_n} and Proposition \ref{comp_gauge_Haus}, we know that $\mathcal{H}^{h}(K_i)=0$.
As each Hausdorff measure with respect to $h$ is an outer measure,
$$\mathcal{H}^{h}(K)\leq \sum_{i \in \mathbb{N}} \mathcal{H}^{h}(K_i) = 0,$$
which contradicts the construction that the set $K$ has non-$\sigma$-finite $\mathcal{H}^{h}$-measure. Thus the set $K$ has no increasing countable slice of Hausdorff dimension, which ends the proof.
	\end{proof}

\smallskip

Note that $(2) \Rightarrow (1)$ of Theorem \ref{MME Bowen-s7} has been proved by Theorem \ref{MME Bowen-s6} for a compact subset, the general case of an analytic subset follows directly from Proposition \ref{find_compact_in_analytic} below.

	\begin{prop}\label{find_compact_in_analytic}
Assume that the system $(X,T)$ is asymptotically $h$-expansive. Let $Z\subset X$ be an analytic subset with $h_\mathrm{top}^B(T, Z)>0$. If $Z$ has no increasing countable slice, then there exists a compact set $K\subset Z$ with $h_\mathrm{top}^B(T, K)=h_\mathrm{top}^B(T, Z)$ such that $K$ also has no increasing countable slice.
	\end{prop}

\begin{proof}
Let us assume that TDS $(X, T)$ is asymptotically $h$-expansive.

We shall lift all subsets of the system $(X,T)$ into symbolic system via the technique of principal extension, then convert our problem of Bowen entropy into that of Hausdorff dimension, and finally deduce the conclusion with the help of Proposition \ref{find_compact_in_analytic_Haus}.

By the well-known characterization of asymptotically $h$-expansive TDS given by \cite[Theorem 9.3.3]{Downarowicz2011}, we may take a symbolic system $(Y, \sigma)$ with $Y\subset \{1, \cdots, M\}^\mathbb{Z}$ for some $M\in \mathbb{N}\setminus \{1\}$, where $\sigma$ is the shift transformation over $\{1, \cdots, M\}^\mathbb{Z}$ and $Y\subset \{1, \cdots, M\}^\mathbb{Z}$ is a nonempty closed invariant subset,
and an extension $\pi: (Y, \sigma) \rightarrow(X, T)$ between TDSs, such that the extension $\pi$ is principal, equivalently, $h_\mathrm{top}(\sigma,\pi^{-1}(x))=0$ for every $x\in X$. Applying  Proposition \ref{Bowen_ent_factor_inequality} one has that
\begin{equation} \label{eq3}
h_{\mathrm {top }}^B(\sigma, E) = h_{\mathrm {top }}^B(T, \pi(E))\ \ \ \text{for each subset}\ E\subset Y.
\end{equation}

Now assume that $Z\subset X$ is an analytic subset with $h_\mathrm{top}^B(T, Z)>0$, and that $Z$ has no increasing countable slice of Bowen entropy. Note that the analytic subsets in a Polish space $X$ are closed under continuous
images, and inverse images (see \cite[14.A]{Kechris1995}). Clearly, $\pi^{- 1} (Z)\subset Y$ is an analytic subset with $h_\mathrm{top}^B(\sigma, \pi^{- 1} Z) = h_\mathrm{top}^B(T, Z) >0$, and that $\pi^{- 1} (Z)$ has no increasing countable slice of Bowen entropy (by applying \eqref{eq3}).

We consider the following compatible metric $d$ over $Y$: set
$$
d\left(\left(x_p\right)_{p \in \mathbb{Z}},\left(y_p\right)_{p \in \mathbb{Z}}\right)=2^{-N},\ \text{where}\ N=\sup \{k \in \mathbb{Z}_{+}: x_m=y_m, \forall -\sqrt{k} <m<k\}$$
  for any sequences $\left(x_p\right)_{p \in \mathbb{Z}},\left(y_p\right)_{p \in \mathbb{Z}}$ in $Y$, where $N=0$ means $x_0 \neq y_0$ and $N=\infty$ means that $\left(x_p\right)_{p\in \mathbb{Z}}$ and $\left(y_p\right)_{p \in \mathbb{Z}}$ are the same sequences.
We remark that the metric $d$ is different from the usual one where $N$ is defined in another way as $\sup \{k \in \mathbb{Z}_{+}: x_m=y_m, \forall -k <m<k\}$.

The following result seems to be well-known, however we failed to find it in the literature, and so we provide here a proof of it for completeness.

\begin{claim} \label{need}
For the above defined metric $d$ over the system $Y$, it holds
\begin{equation}\label{eq1}
h_{\mathrm{top}}^B(\sigma, E)=\log 2 \cdot \operatorname{dim}_{\mathcal{H}}(E) \text { for each subset } E \subset Y.
\end{equation}
\end{claim}

\begin{proof}
Firstly we introduce some notations. For a finite collection of symbols $w_m\in \{1,\cdots,M\}$, $m=-i,\cdots,j$ where $i, j\in \mathbb{Z}_+$, denote $[w_{-i},\cdots,\underline{w_0},\cdots,
w_j]$ as the column set in $Y$ of length $j+i+1$, in other words, $[w_{-i},\cdots,\underline{w_0},\cdots,
w_j]$ is the collection of all points $\left(x_p\right)_{p \in \mathbb{Z}} \in Y$ satisfying $x_m=w_m$ for all $-i\leq m\leq j$ (where $\underline{*}$ denotes the symbol at position 0).
In particular, for each $F\subset Y\subset \{1,\cdots,M\}^{\mathbb{Z}}$ and $n\in \mathbb{N}$, we have $\operatorname{diam} (F)\leq 2^{-n}$ if and only if $d\left(\left(x_p\right)_{p \in \mathbb{Z}},\left(y_p\right)_{p \in \mathbb{Z}}\right)\leq 2^{-n}$ for some point $\left(x_p\right)_{p \in \mathbb{Z}} \in F$ and all $\left(y_p\right)_{p \in \mathbb{Z}}\in F$, that is, for each $\left(x_p\right)_{p \in \mathbb{Z}} \in F$,
if $n^*$ is the minimal integer strictly larger than $-\sqrt{n}$ (in particular, $n^*= -\lfloor \sqrt{n}\rfloor+1$ if $\sqrt{n}\in \mathbb{N}$ and $n^*= -\lfloor \sqrt{n}\rfloor$ if $\sqrt{n}\notin \mathbb{N}$), then
$$
F\subset [x_{n^*},\cdots,\underline{x_0},\cdots,x_{n-1}].$$

Therefore, on one hand, when discussing the Hausdorff measure $\mathcal{H}^s(\cdot)$ (for each $s> 0$) with respect to the metric $d$, we may safely assume that each covering set (in the definition of Hausdorff measure) is of the form $[x_{n^*},\cdots,\underline{x_0},\cdots,x_{n-1}]$ for some $\left(x_p\right)_{p \in \mathbb{Z}} \in Y$ and some $n\in \mathbb{N}$.
On the other hand, when concerning Bowen entropy, a basic computation shows that $h_{\text {top}}^B\left(\sigma, E\right)=h_{\text {top}}^B(\sigma, E, \frac{1}{2})$ for all $E\subset Y$. Hence, in this case, the covering Bowen ball $B_n(\left(x_p\right)_{p \in \mathbb{Z}}, \frac{1}{2})$ (in the definition of Bowen entropy) is of the form $[\underline{x_0},\cdots,x_{n-1}]$.

Clearly, $[x_{n^*},\cdots,\underline{x_0},\cdots,x_{n-1}]$ is a subset of $[\underline{x_0},\cdots,x_{n-1}]$, and then $$\mathcal{H}^s_{2^{-N}}(E)\geq \M^{s\log 2}_{N,\frac{1}{2}}(E)$$ for all $N\geq 1$, $s>0$ and $E\subset Y$, where the quantities $\mathcal{H}^s_{\varepsilon}(\cdot )$ and $\M^{s \log 2}_{N,\varepsilon}(\cdot)$ are introduced in subsection $\S\S \ref{defin}$. This implies that $h_{\mathrm{top}}^B(\sigma, E)\leq \log 2 \cdot \operatorname{dim}_{\mathcal{H}}(E)$.

Conversely, fix arbitrarily $\delta>0$, we choose $N\in \mathbb{N}$ large enough such that $e^{n\delta}\geq M^{\sqrt{n}}$ for all $n\geq N$. Note that each set $[\underline{x_0},\cdots,x_{n-1}]$ can be covered by at most $M^{\sqrt{n}}$ sets of the form $[*,\cdots,*,\underline{x_0},\cdots,x_{n-1}]$, where $[*,\cdots,*]$ is some word of length $|n^*|\ (\le \sqrt{n})$. In other words, each Bowen ball $B_{n}\left(\left(x_p\right)_{p \in \mathbb{Z}}, \frac{1}{2}\right)$  can be covered by at most $M^{\sqrt{n}}$ sets with diameters at most $2^{-n}$. Therefore, for any $E\subset Y$, given a countable family of Bowen balls $\left\{ B_{n_{i}}\left(\left(x_p\right)_{p \in \mathbb{Z}}, \frac{1}{2}\right)\right\}_{i\in \mathcal{I}}$ covering $E$ with each $n_i\ge N$, there exists for each $i\in \mathcal{I}$ a countable family $\{E_i^{(j)}\}_{j\in \mathcal{J}_i}$ of sets with diameters at most $2^{-n_i}\le 2^{- N}$ such that $\# \mathcal{J}_i\leq M^{\sqrt{n_i}}$ and the family $\{E_i^{(j)}\}_{j\in \mathcal{J}_i}$ covers
the Bowen ball $B_{n_{i}}\left(\left(x_p\right)_{p \in \mathbb{Z}}, \frac{1}{2}\right)$. This implies that
$$
\sum_{i\in \mathcal{I}} e^{-n_{i}(s\log 2-\delta)}\geq \sum_{i\in \mathcal{I}} 2^{- n_{i}s}\cdot M^{\sqrt{n_i}}\geq \sum_{i\in \mathcal{I}} \sum_{j\in \mathcal{J}_i} 2^{- n_{i}s}.$$
By the arbitrariness of the countable family of Bowen balls and the definitions, one has
$$\mathcal{H}^s_{2^{-N}}(E)\leq \M^{s\log 2-\delta}_{N,\frac{1}{2}}(E)$$
for all $N\in \mathbb{N}$ large enough (this depends on the parameter $\delta> 0$). Letting $N$ tend to infinity and noting the arbitrariness of $\delta> 0$, we have $h_{\mathrm{top}}^B(\sigma, E)\geq \log 2 \cdot \operatorname{dim}_{\mathcal{H}}(E) - \delta$ and then $h_{\mathrm{top}}^B(\sigma, E)\geq \log 2 \cdot \operatorname{dim}_{\mathcal{H}}(E)$. This ends the proof of the claim.
\end{proof}

Now let us continue the proof of Proposition \ref{find_compact_in_analytic}.
Combining \eqref{eq1} with the assumption that $\pi^{- 1} (Z)$ has no increasing countable slice of Bowen entropy, one has that the set $\pi^{- 1} (Z)$ has no increasing countable slice of Hausdorff dimension, and then by Proposition \ref{find_compact_in_analytic_Haus},
 there exists a compact subset $K^*\subset \pi^{- 1} (Z)$ such that $K^*$ has no increasing countable slice of Hausdorff dimension and $\dim_{\mathcal{H}}(K^*)=\dim_{\mathcal{H}}(\pi^{- 1} Z)$. Using again \eqref{eq1}, we know that $h_\mathrm{top}^B(\sigma, K^*)=h_\mathrm{top}^B(\sigma, \pi^{- 1} Z)$ and that $K^*$ has no increasing countable slice of Bowen entropy. Now set $K= \pi (K^*)\subset Z$, which is clearly a compact set with the required properties (by \eqref{eq3}). This finishes the proof.
	\end{proof}	

	\begin{remark}
Though we only need to assume in Proposition \ref{find_compact_in_analytic} that the system $(X,T)$ is asymptotically $h$-expansive, we can only prove the conclusion $(2) \Rightarrow (1)$ of Theorem \ref{MME Bowen-s7} for an $h$-expansive system. Because we have remarked that, here we use the case of $(2) \Rightarrow (1)$ in Theorem \ref{MME Bowen-s7} for a compact subset, which has been proved by Theorem \ref{MME Bowen-s6} under the assumption that the system $(X,T)$ is $h$-expansive.
	\end{remark}

\section*{Appendix. Proof of Proposition \ref{constuct_Local_const}}

In this section let us finish the proof of Proposition \ref{constuct_Local_const}. We have remarked that the construction may be standard, and we provide here a proof of it for completeness.

 Our proof will be divided into two parts: we construct firstly a non-decreasing $C^\infty$ function $\psi$ which satisfies conditions \eqref{clc_1} and \eqref{clc_2} of Proposition \ref{constuct_Local_const} and other properties required in later construction, and then we modify those local constants in Proposition \ref{constuct_Local_const} \eqref{clc_2} so that we obtain a non-decreasing $C^\infty$ function $\phi$ which satisfies not only conditions \eqref{clc_1} and \eqref{clc_2} of Proposition \ref{constuct_Local_const} but also condition \eqref{clc_3} of Proposition \ref{constuct_Local_const}.

\medskip
	
The main idea of the first part is to concatenate constant intervals.

We take arbitrarily a non-decreasing $C^\infty$ function $s: [0,1] \rightarrow [0,1]$ satisfying
$s(0)=0$, $s(1)=1$ and $s^{(i)}(0) = s^{(i)}(1)=0$ for all $i\in \mathbb{N}$, where $s^{(i)}(x)$ denotes the $i^{th}$ derivative of $s$ at the point $x$. Such a function is easy to construct.

In the following let us use $s$ as a concatenating function to construct $\psi$ inductively.

\medskip

Recall that $0<u_1<v_1<u_2<v_2<\cdots \nearrow 1$. Set $\gamma_0=0$.

Let us firstly connect $0$ and $u_1$, define $\psi$ on $[0,v_1]$ and define $\gamma_1$.
For each $x\in[0, u_1]$, simply let $\psi(x)=\frac{1}{2}\cdot s {\Big (}\frac{x}{u_1} {\Big )},$
which is a $C^\infty$ function sending $[0,u_1]$ to $[0,\frac{1}{2}]$.
Moreover, for each $i\in \mathbb{N}$, since $s^{(i)}(1)=0$, one has
$\psi^{(i)}(u_1-)=0$. Set $\gamma_1= \frac{1}{2}$, and define $\psi(x)=\gamma_1$ for $x\in[u_1,v_1]$. Therefore, $\psi$ is a non-decreasing $C^\infty$ function well defined on $[0,v_1].$

Assume that $\psi$ has been defined on $[0,v_{p-1}]$, $\gamma_i$ has been determined for all $1\leq i\leq p-1$, where $p\ge 2$.
Now let us connect $v_{p-1}$ and $u_p$, define $\psi$ on $[0,v_p]$ and define $\gamma_p$ in a way similar to above procedure.
That is, for each $x\in[v_{p-1},u_p]$, we simply let
$$\psi(x)=d_p\cdot s {\Big (}\frac{x-v_{p-1}}{u_p-v_{p-1}} {\Big )}+\gamma_{p-1},$$
where $d_p>0$ is small enough such that
$$
\frac{\psi(u_p)-\gamma_{p-1}}{\gamma_{p-1}-\gamma_{p-2}}< \frac{1}{2}\ \ \text{and}\ \ |\psi^{(j)}(x)|<\frac{1-v_p}{p}\ \ \ \ \ \ (\forall x\in[v_{p-1},u_p],\ \forall 1\leq j\leq p).$$
Set $\gamma_p=\psi(u_p)$ and $\psi(x)=\gamma_p$ for $x\in[u_p,v_p]$. So $\psi$ is a $C^\infty$ function well defined on $[0,v_p].$
	
We have successfully defined $\psi$ on the interval $[0,v_p]$ for all $p\in \mathbb{N}$. Then $\psi$ is a non-decreasing $C^\infty$ function defined on $[0, 1)$. Moreover, we have that for all $i\in \mathbb{N}$,
\begin{enumerate}
	
	\item[(a)] $|\psi^{(j)}(x)|<\frac{1-v_{i+1}}{i+1}\leq \frac{1}{2}$ for all $v_i\leq x\leq u_{i+1}$ and $1\leq j\leq i+1$.
	
	\item[(b)] $|\psi^{(j)}(x)|=0<\frac{1-v_{i+1}}{i+1}$ for all $u_i\leq x\leq v_i$ and $j\in \mathbb{N}$, since $\psi$ is constant on $[u_i, v_i]$.
	
	\item[(c)] $2(\gamma_{i+1}-\gamma_{i})<\gamma_{i}-\gamma_{i-1}$.
\end{enumerate}

Set $\psi(1)\doteq \lim\limits_{x\to 1-} \psi(x)$, and then
$\psi$ is a non-decreasing continuous function on $[0, 1]$, as $\psi$ is non-decreasing on $[0, 1)$ and by (a), (b) and above construction
	$$\psi(1)= \psi(v_1)+\sup_{v_1\leq t<1}\int_{v_1}^{t}\psi'(x)dx\le \frac{1}{2} + \frac{1- v_1}{2}\leq 1.$$

Below we show that $\psi$ is smooth at $x=1$ by proving $\psi^{(k)}(1)=0$ for each $k\in \mathbb{N}$.
When $k=1$, we have $\psi'(1)=0$ by applying again (a) and (b) above to obtain
 \begin{align*}
& \hskip -36 pt \limsup_{x\to 1-} {\Big |}\frac{\psi(x)-\psi(1)}{x-1} {\Big |}
		 = \limsup_{x\to 1-} {\Big |}\frac{\int_{x}^{1}\psi'(t)dt}{x-1} {\Big |}\\
		 \leq & \limsup_{i\to \infty} \sup_{v_i\leq x\leq v_{i+1}} {\Big |}\frac{(1-x)\cdot \frac{1-v_{i+1}}{i+1} }{1-x} {\Big |}=\limsup_{i\to \infty} {\Big |}\frac{1-v_{i+1}}{i+1} {\Big |}=0.	
\end{align*}
When $k\geq 2$, by induction we have $\psi^{(k)}(1)=0$ similarly to the above process as follows
\begin{equation*}
\begin{aligned}
& \limsup_{x\to 1-} {\Big |}\frac{\psi^{(k-1)}(x)- \psi^{(k-1)}(1)}{x-1} {\Big |} \\
= & \limsup_{x\to 1-} {\Big |}\frac{\psi^{(k-1)}(x)}{x-1} {\Big |} \ (\text{by induction to assume $\psi^{(k-1)}(1) = 0$}) \\
\leq & \limsup_{i\to \infty} \sup_{v_i\leq x\leq v_{i+1}}  {\Big |}\frac{ \frac{1-v_{i+1}}{i+1} }{1-x} {\Big |}\ (\text{applying (a) and (b)})\\
\leq & \limsup_{i\to \infty} {\Big |}\frac{1-v_{i+1}}{(i+1)(1-v_{i+1})} {\Big |}=0.
\end{aligned}
\end{equation*}

Clearly, $\frac{1}{\psi(1)}\cdot \psi: [0, 1]\rightarrow [0, 1]$ is an ideal function satisfying conditions \eqref{clc_1} and \eqref{clc_2} of Proposition \ref{constuct_Local_const}. For brevity, we denote it still by $\psi$ and for each $i\in \mathbb{N}$ use still $\gamma_i$ the new local constants, in particular, $\{\gamma_i\}_{i\in \mathbb{N}}$ increases strictly to $1$ and it also holds that $2(\gamma_{i+1}-\gamma_{i})<\gamma_{i}-\gamma_{i-1}$ for each $i\in \mathbb{N}$.

\medskip

Let us move forward to the second part of the proof. The following fact will be useful.
	
\begin{lemma}\label{modify_local_const}
	Let $\psi: [0, 1]\rightarrow [0, 1]$ be the function defined as above, the set $E$ be as in Proposition \ref{constuct_Local_const}, and $i\geq 2$ satisfy $E\cap [ \gamma_i,\gamma_{i+1}] \neq\emptyset$. Then there exists a non-decreasing $C^\infty$ function $\eta: [0,1] \rightarrow [0,1]$ such that
\begin{enumerate}
		\item $\eta(x)=\psi(x)$ for all $x\notin [v_{i-1},u_{i+1}]$ and $\eta(x)$ takes a constant value $c\in E$ on $[u_i,v_i]$.

		\item $|\eta^{(k)}(x)|\leq 2\cdot|\psi^{(k)}(x)|$ for all $k \in \mathbb{N}$ and $x\in[0,1]$.
	\end{enumerate}	
\end{lemma}
	
\begin{proof}
Take arbitrarily $c\in E\cap [\gamma_i,\gamma_{i+1}]$. The function $\eta$ is given by condition (1) and 	
	$$
			\eta(x) = \begin{cases}\frac{c-\gamma_{i-1}}{\gamma_i-\gamma_{i-1}}\cdot (\psi(x)-\gamma_{i-1}) +\gamma_{i-1}&\text { when } x\in[v_{i-1},u_i] \\
			\frac{\gamma_{i+1}-c}{\gamma_{i+1}-\gamma_{i}}\cdot (\psi(x)-\gamma_{i}) +c &\text { when } x\in[v_i, u_{i+1}]
\end{cases}.
	$$
As $2 (\gamma_{i+1}-\gamma_{i})<\gamma_{i}-\gamma_{i-1}$ and $\gamma_i\le c\le \gamma_{i + 1}$, one has
$\gamma_i+(\gamma_i-\gamma_{i-1})> \gamma_i+(\gamma_{i + 1} -\gamma_{i}) = \gamma_{i+1}\geq c$ and then $2 (\gamma_i-\gamma_{i-1})> c - \gamma_{i - 1}$, thus
$$0\leq\frac{\gamma_{i+1}-c}{\gamma_{i+1}-\gamma_{i}}\leq 1\leq\frac{c-\gamma_{i-1}}{\gamma_i-\gamma_{i-1}}\leq 2.$$
From which and construction of $\psi$ as above, it is easy to check required properties of $\eta$.
\end{proof}

In the following, we will construct by induction a non-decreasing $C^\infty$ function $\phi$ satisfying all properties of Proposition \ref{constuct_Local_const}, where we lift a countable number of local constants,
via Lemma \ref{modify_local_const}, without destroying the smoothness of point $1$.

Let us firstly choose $c_1\in E$, $n_1\in \mathbb{N}$ and define a non-decreasing $C^\infty$ function $\eta_1: [0,1] \rightarrow [0, 1]$.
As the closure $\overline{E}$ of the subset $E\subset [0, 1)$ contains the point $1$, we can take arbitrarily $c_1\in E$ with $\gamma_2<c_1<1$, and then pick $n_1\geq 2$ such that $c_1\in E\cap [\gamma_{n_1},\gamma_{n_1+1}].$ By Lemma \ref{modify_local_const}, we may obtain the function $\eta_1$ by modifying the function $\psi$ such that
$\eta_1$ takes the constant value $c_1$ on $[u_{n_1},v_{n_1}]$ and that $|\eta_1^{(k)}(x)|\leq 2\cdot|\psi^{(k)}(x)|$ for all $k \in \mathbb{N}$ and $x\in[0,1]$.

Assume that $c_i\in E$ and $n_i\in \mathbb{N}$ have been chosen and a non-decreasing $C^\infty$ function $\eta_i: [0, 1]\rightarrow [0, 1]$ has been defined for all $1\leq i\leq p-1$ (with $p\ge 2$).
Similarly we may take arbitrarily $c_p\in E$ with $\gamma_{n_{p-1}+2}<c_p<1$, and then pick $n_p\ge n_{p-1}+2$ such that  $c_p\in E\cap [ \gamma_{n_p},\gamma_{n_p+1}].$
Similarly, by Lemma \ref{modify_local_const}, we may modify $\eta_{p - 1}$ into a non-decreasing $C^\infty$ function $\eta_p: [0, 1]\rightarrow [0, 1]$ such that $\eta_p$ takes the constant value $c_p$ on $[u_{n_p},v_{n_p}]$ and that $|\eta_p^{(k)}(x)|\leq 2\cdot|\psi^{(k)}(x)|$ for all $k \in \mathbb{N}$ and $x\in[0,1]$.

Finally, the required non-decreasing $C^\infty$ function $\phi: [0, 1]\rightarrow [0, 1]$ is defined by setting $\phi (1) = 1$ and $\phi(x)=\eta_i(x)$ for all $x\in[v_{n_{i-1}+1},v_{n_i+1}]$ and each $i\ge 2$.
After a countable steps of modifications as above, it is easy to see that $\phi$ is a non-decreasing $C^\infty$ function on $[0,1)$ and that $\phi(x)=c_i\in [\gamma_{n_i}, \gamma_{n_i + 1}]$ for all $x\in[u_{n_i},v_{n_i}]$ and each $i\ge 2$.

To finish the proof of Proposition \ref{constuct_Local_const}, it suffices to verify the smoothness of $\phi$ at $x=1$.
Note that $\{\gamma_i\}_{i\in \mathbb{N}}$ increases up to $1$, so does $\{c_i\}_{i\in \mathbb{N}}$, and then it is easy to check from $\phi(1)=1$ that $\phi$ is continuous on $[0, 1]$.
Moreover, by above construction, we know that $|\phi^{(k)}(x)|\leq 2\cdot|\psi^{(k)}(x)|$ for all $k\geq 1$ and $0 \leq x < 1$, and then it could be easily verified that $\phi^{(k)}(1)=\psi^{(k)}(1)=0$ for all $k\geq 1$. This ends the whole proof.

\section*{Acknowledgements}

We would like to thank Professors Yongluo Cao, Ercai Chen, Wen Huang, Weixiao Shen and Xiangdong Ye for valuable comments, in particular, we thank Professor Weixiao Shen for explanations about dynamics of the unimodal interval maps. We also thank Professor Theodore A. Slaman for drawing our attention to the study of Hausdorff measure with respect to a gauge function.

Guohua Zhang is supported by the National Key Research and Development Program of China (No. 2021YFA1003204). This
work has been supported by the New Cornerstone Science Foundation through the New
Cornerstone Investigator Program.

\bibliographystyle{amsplain} 

\printindex

\end{document}